\pdfoutput=1
\documentclass[oneside,notitlepage,12pt]{article}

\usepackage[T1]{fontenc}
\usepackage[utf8]{inputenc}

\usepackage{amssymb}
\usepackage[leqno]{amsmath}
\usepackage{amsfonts}
\usepackage{amsopn}
\usepackage{amstext}
\usepackage{amsthm}
\usepackage{enumitem}
\usepackage{txfonts} 

\usepackage{tikz}
\usetikzlibrary{cd}
\usetikzlibrary{graphs}
\usetikzlibrary{backgrounds} 
\usetikzlibrary{shapes.geometric} 
\providecommand{\ar}{\arrow}

\definecolor{blue}{rgb}{0, 0.2, 0.8}
\definecolor{red}{rgb}{0.8, 0, 0.2}
\definecolor{green}{rgb}{0, 0.7, 0.2}

\newlength{\ru} 
\tikzset{
	text diagram/.style = {
		text height = 1.5ex,
		text depth = 0.25ex,
	},
	trees/.style = {
		grow' = up,
		level distance = 5\ru,
		level 1/.style = {sibling distance = 7\ru},
		level 2/.style = {sibling distance = 5\ru},
		circle node/.style = {
			circle, fill, inner sep = 0.4\ru,
		},
		hollow circle node/.style = {
			circle, fill=none, draw, line width = 0.4\ru, inner sep = 0.3\ru,
		},
		triangle node/.style = {
			isosceles triangle,
			isosceles triangle apex angle=60,
			fill, draw=none,
			inner sep = 0.35\ru,
		},
		up triangle node/.style = {
			triangle node, rotate=90,
		},
		left triangle node/.style = {	
			triangle node, rotate=180,
		},
		right triangle node/.style = {	
			triangle node, rotate=0,
		},
		diamond node/.style = {
			diamond, fill, inner sep = 0.4\ru,
		},
		dot nodes/.style = {
			every node/.style = {circle node},
		},
		edge from parent/.style = {
			line width = rule_thickness,
			draw,
		},
		extension/.style = { 
			every node/.append style=##1,
			every child/.style=##1,
		},
		subtree/.style 2 args = { 
			every node/.style=##1,
			every child/.style=##1,
			every node/.append style=##2,
		},
		shift children/.style = {
			every child/.append style = {xshift=##1},
		},
	},
}

\usepackage[colorlinks, backref]{hyperref}
\hypersetup{
	linkcolor = [rgb]{0, 0, 0.7},
	citecolor = [rgb]{0, 0.7, 0},
}

\usepackage{calrsfs}
\usepackage{fourier-orns}
\usepackage{hieroglf} 
\usepackage{clock} 

\setlist{itemsep = 0pt}
\setlist[enumerate, 1]{label=\upshape (\roman*), ref=(\roman*)}

\providecommand{\cal}{\mathcal}
\renewcommand{\Bbb}{\mathbb}

\newenvironment{pf}{\begin{proof}}{\end{proof}}

\newcommand{\Ef}{{\cal{F}}}

\newcommand{\Vee}{{\cal{V}}}

\newcommand{\Nat}{{\Bbb{N}}}
\newcommand{\Qyu}{{\Bbb{Q}}}

\newcommand{\al}{\alpha}

\newcommand{\eps}{\varepsilon}
\renewcommand{\phi}{\varphi}
\renewcommand{\rho}{\varrho}

\newcommand{\ntr}{{n\in\omega}}

\newcommand{\loe}{\leqslant}
\newcommand{\goe}{\geqslant}

\newcommand{\subs}{\subseteq}
\newcommand{\sups}{\supseteq}
\newcommand{\nnempty}{\ne\emptyset}

\renewcommand{\iff}{\,\Longleftrightarrow\,}
\newcommand{\letiff}{\,:\Longleftrightarrow\,}

\newcommand{\id}[1]{{\operatorname{i\!d}_{#1}}} 

\newcommand{\dom}{\operatorname{dom}}
\newcommand{\cod}{\operatorname{cod}}

\newcommand{\supp}{\operatorname{supp}}

\newcommand{\oraz}{\qquad\text{and}\qquad}

\newcommand{\meet}{\wedge}

\newcommand{\join}{\vee}

\newcommand{\by}[1]{/{#1}}

\newcommand{\set}[1]{\{#1\}}
\newcommand{\setof}[2]{\{#1\colon #2\}}
\newcommand{\bigsetof}[2]{\Bigl\{#1\colon #2\Bigr\}}

\newcommand{\sett}[2]{\{#1\}_{#2}}
\newcommand{\sn}[1]{\{#1\}} 
\newcommand{\dn}[2]{\{#1,#2\}} 
\newcommand{\map}[3]{#1\colon #2 \to #3} 
\newcommand{\img}[2]{#1[#2]} 

\newcommand{\fra}{Fra\"iss\'e}

\providecommand{\nat}{\omega}

\newcommand{\ciag}[1]{{\sett{{#1}_n}{\ntr}}}

\newcommand{\aut}{\operatorname{Aut}}

\newcommand{\fL}{{\mathfrak{L}}}

\newcommand{\fS}{{\mathfrak{S}}}
\newcommand{\fC}{{\mathfrak{C}}}

\newcommand{\cmp}{\circ} 

\newcommand{\aseparator}{\begin{center} \leafright \leafright \decotwo \leafleft \leafleft \end{center}}

\newcommand{\ob}[1]{\operatorname{Obj}(#1)}

\newcommand{\acts}{\curvearrowright}

\newcommand{\proto}[1]{{\mathbb S_\kappa}}

\newtheorem{tw}{Theorem}[section]
\newtheorem{twm}{Theorem}
\newtheorem{wn}[tw]{Corollary}
\newtheorem{lm}[tw]{Lemma}
\newtheorem{prop}[tw]{Proposition}

\theoremstyle{definition}
\newtheorem{df}[tw]{Definition}
\newtheorem{ex}[tw]{Example}

\newtheorem{question}[tw]{Question}

\newtheorem{observation}[tw]{Observation}
\newtheorem{remark}[tw]{Remark}
\newtheorem{con}[tw]{Construction}
\theoremstyle{remark}

\newcommand{\maps}{\colon} 
\newcommand{\restr}[1]{\mathop{\upharpoonright}{#1}} 
\newcommand{\card}[1]{\lvert #1\rvert} 

\newcommand{\idvec}[1]{\vec{\operatorname{i\!d}}_{#1}} 
\newcommand{\matches}{\vartriangleright} 
\newcommand{\vecinfty}[1]{\vec{#1}^{\,\infty}} 
\newcommand{\vecinv}[1]{\vec{#1}^{\,-1}} 

\newcommand{\am}{\operatorname{Am}} 
\newcommand{\arex}[1]{#1^\uparrow} 

\newcommand{\LO}{\mathsf{LO}} 
\newcommand{\aLO}{\mathsf{LO_3}} 
\newcommand{\tLO}{\mathsf{LO_{3R}}} 
\newcommand{\FinLO}{\mathsf{FinLO}} 
\newcommand{\FinaLO}{\mathsf{FinLO_3}} 
\newcommand{\FintLO}{\mathsf{FinLO_{3R}}} 

\renewcommand{\leq}{\leqslant}
\renewcommand{\geq}{\geqslant}
\renewcommand{\phi}{\varphi}
\renewcommand{\rho}{\varrho}

\hypersetup{
	pdftitle = {The weak Ramsey property and extreme amenability},
	pdfauthor = {Adam Bartoš, Tristan Bice, Keegan Dasilva Barbosa, Wiesław Kubiś},
}

\title{\vspace{-2em}The weak Ramsey property and extreme amenability}
\author{Adam Bartoš \\
		\small \href{mailto:bartos@math.cas.cz}{\nolinkurl{bartos@math.cas.cz}} \\
		\small Institute of Mathematics, \\
		\small Czech Academy of Sciences, \\
		\small Žitná 25, 115 67 Prague, Czech Republic
	\and Tristan Bice \\
		\small \href{mailto:bice@math.cas.cz}{\nolinkurl{bice@math.cas.cz}} \\
		\small Institute of Mathematics, \\
		\small Czech Academy of Sciences, \\
		\small Žitná 25, 115 67 Prague, Czech Republic
	\and Keegan Dasilva Barbosa \\
		\small \href{mailto:keegan.dasilvabarbosa@mail.utoronto.ca}{\nolinkurl{keegan.dasilvabarbosa@mail.utoronto.ca}} \\
		\small Department of Mathematics, \\
		\small University of Toronto, \\
		\small 40 St. George Street, Toronto, Ontario, Canada M5S 2E4
	\and Wiesław Kubiś \\
		\small \href{mailto:kubis@math.cas.cz}{\nolinkurl{kubis@math.cas.cz}} \\
		\small Institute of Mathematics, \\
		\small Czech Academy of Sciences, \\
		\small Žitná 25, 115 67 Prague, Czech Republic
}
\date{\today\ \clocktime}

\begin{document}

\maketitle

\begin{abstract}
	We extend the Kechris--Pestov--Todorčević correspondence to weak \fra\ categories and automorphism groups of generic objects. The new ingredient is the weak Ramsey property.
	We demonstrate the theory on several examples including monoid categories, the category of almost linear orders, and categories of strong embeddings of trees.
	
	\emph{Keywords:} The weak Ramsey property, generic object, extreme amenability.
	
	\emph{2020 AMS Subject Classification:} 
		22F50, 
		05D10, 
		18A30, 
		54H11. 
\end{abstract}

\thanks{Research of the first, second and fourth author was supported by GA \v CR (Czech Science Foundation) grant EXPRO 20-31529X and RVO: 67985840. The third author was supported by the Ontario Graduate Scholarship.

We would like to thank the anonymous referee for carefully reading the manuscript and for their comments that helped to improve the paper.}

\tableofcontents

\section{Introduction}

The main motivation for this note is the seminal work of Kechris, Pestov, Todor\v cevi\'c~\cite{KPT} exhibiting the connection between extreme amenability of automorphism groups of countable homogeneous structures and Ramsey-type properties of their finite substructures.
This work has been very recently extended by Ma\v sulovi\'c~\cite{MasKPT} in the language of category theory. Our goal is to push it even further, namely, by making minimal assumptions both on the category of ``small'' structures and weakening the homogeneity of the ``generic'' object.

The phenomenon discovered by Kechris, Pestov, Todor\v cevi\'c~\cite{KPT} can be described in its simplest form as follows. We have a class $\Ef$ of finite structures of a fixed language. We assume the class has the subsequent nice properties: Every two structures in $\Ef$ can be embedded into a single one (the joint embedding property); every two extensions of a structure in $\Ef$ can be combined into a single one (the amalgamation property); every substructure of a structure in $\Ef$ is again in $\Ef$ (the class is hereditary); there are countably many isomorphic types in $\Ef$.
When these conditions are met, there exists a unique countable structure $U$ whose finite substructures are all in $\Ef$, such that every structure from $\Ef$ embeds into $U$, and $U$ is ultrahomogeneous with respect to $\Ef$, i.e. every isomorphism between finite substructures of $U$ extends to an automorphism of $U$. These are the foundational objects of study in Fra\"iss\'e theory.
Consider the group $G = \aut(U)$ with the pointwise convergence topology.
The \emph{KPT correspondence} states that the group $G$ is extremely amenable (i.e. every continuous action of $G$ on a compact Hausdorff space has a fixed point) if and only if the class $\Ef$ has the Ramsey property and the ordering property. The ordering property ensures that all structures in $\Ef$ are rigid (i.e. have trivial automorphism groups).
The Ramsey property is the structural variant of the classical Ramsey theorem, where we color structures of a fixed isomorphism type from $\Ef$ instead of coloring subsets.

For example, if $\Ef$ is the class of finite linearly ordered sets then the Ramsey property asserts in particular that, for every finite linearly ordered set $X$ and any positive integers $m$ and $k$, there exists a bigger finite linearly ordered set $Y$ such that when we color all copies of the $m$-element linear ordering inside $Y$ with at most $k$ many colors, we can always find an embedding $e$ of $X$ into $Y$ (namely, a subset with the same number of elements as $X$) such that all $m$-element subsets of $\img e X$ have the same color. 
This is equivalent to the classical finite Ramsey theorem.
Its formulation in the language of linear orderings has two advantages: First, it allows us to talk about embeddings instead of subsets, as the domain of an embedding can be identified with its image. Second, what is perhaps more important, it is purely category-theoretic.

The last observation leads to a natural idea: Replace the class of finite structures $\Ef$ by an abstract category $\fC$, whose arrows are meant to be some sort of ``embeddings''.
A natural assumption, made by Ma\v sulovi\'c~\cite{MasKPT} in this categorical framework, is that every object has only finitely many arrows going into it. Note that this is immediately true in the case that the category in question is a category of finite models with embeddings for arrows. As it so happens, this assumption, and the weaker assumption asserting that there are only finitely many arrows between two prescribed objects, is not necessary to capture the KPT correspondence in a categorical framework. We can do so by instead making the Ramsey property a bit more technical, involving only finite subsets of arrows rather than all arrows between two structures.

Nevertheless, it turns out that the KPT correspondence holds in a fairly large class of categories in which the notion of ``being finite'' is replaced by a factorization property with respect to a fixed sequence. Despite this subtle change, we still arrive to the same connection between extreme amenability and the Ramsey property.
In order to make the theory as general as possible, we shall work with so-called \emph{weak} \fra\ categories~\cite{KweakFra}, where the amalgamation property holds in a weaker form. We obtain the equivalence of extreme amenability of the automorphism group of the generic object with the weak version of the Ramsey property, which involves particular arrows (called \emph{amalgamable arrows}).
We also show that a weak \fra\ category $\fC$ gives rise to a natural \fra\ category $\am(\arex{\fC})$.
Moreover, we show that $\fC$ has the \emph{weak} Ramsey property if and only if $\am(\arex{\fC})$ has the Ramsey property.

Our main result can be roughly summarized as follows.

\begin{twm}
Assume $\fS$ is a weak \fra\ category and let $U$ be generic over $\fS$. The following properties are equivalent.
\begin{enumerate}[itemsep=0pt]
    \item[{\rm(a)}] $\aut(U)$ is extremely amenable.
    \item[{\rm(b)}] $\fS$ has the weak Ramsey property.
\end{enumerate}
\end{twm}

Of course, some minimal technical assumptions are needed here, so that there is a good interplay between the topology of $\aut(U)$ and the category $\fS$. In particular, $U$ is an object of a larger category.
Besides that, there are practically no further assumptions on $\fS$, however, the weak Ramsey property involves finite sets of arrows. The precise statement is Theorem~\ref{thm:KPT} below.

\aseparator

The paper is organized as follows:
After the Preliminaries section (where we introduce the setup) we prove the main results in a series of lemmas showing that the weak Ramsey property (see Definition~\ref{DFweakRmsyPy}) is equivalent to extreme amenability of the automorphism group of the generic object.
The last section contains a discussion of the main results and concrete applications.
This includes an analysis of monoid categories, almost linear orders, as well as finite trees under strong embeddings.
The latter exhibits an interesting interplay between Milliken's theorem for trees \cite{MilTrees} and the universal Ważewski dendrites \cite{Duchesne}, \cite{Kwiatkowska}.

\section{Preliminaries}

We shall use very basic concepts from category theory.
For undefined notions we refer to MacLane~\cite{MacLane}.
Categories will be denoted by $\fC$, $\fS$, $\fL$, etc.
A category $\fC$ will be identified with its class of arrows (morphisms) and $\ob{\fC}$ will denote its class of objects. Given two $\fC$-objects $a$, $b$ we denote by $\fC(a,b)$ the set of all arrows $f$ with domain $a$ (i.e., $\dom(f) = a$) and codomain $b$ (i.e., $\cod(f) = b$). We sometimes write $\map f a b$ instead of $f \in \fC(a,b)$, as long as the category $\fC$ is understood from the context.
Composition of arrows is performed in the usual order and denoted by $\cmp$, i.e. $\dom(f \cmp g) = \dom(g)$ and $\cod(f \cmp g) = \cod(f)$.
The identity arrow of an object $x$ will be denoted by $\id{x}$.
All our categories are supposed to be locally small, i.e. the class of all arrows from a fixed object $a$ to a fixed object $b$ is a set, not a proper class.

A subcategory $\fC \subs \fS$ is called \emph{full} if $\fC(a, b) = \fS(a, b)$ for every $a, b \in \ob{\fC}$, i.e. if we are restricting only objects, while $\fC \subs \fS$ is called \emph{wide} if $\ob{\fC} = \ob{\fS}$, i.e. if we are restricting only arrows.

A \emph{sequence} in a category $\fC$ is a covariant functor $\map{F}{\nat}{\fC}$, where the set of natural numbers $\nat$ is treated as a poset category. Namely, $F$ consists of a sequence of objects $\sett{F(n)}{\ntr}$
and a sequence of $\fC$-arrows $\sett{F(n,m)}{n \loe m < \nat}$ such that $F(n,m) \in \fC(F(n), F(m))$, $F(k,k) = \id{F(k)}$, and $F(k,m) = F(\ell, m) \cmp F(k,\ell)$ for every $k \loe \ell \loe m$.
We shall use the following convention: A sequence will be denoted by $\vec u$ (possibly with $u$ replaced by another letter) and in that case we denote $u_n = \vec u(n)$ and $u_n^m = \vec u(n,m)$.

A \emph{cone} for a sequence $\vec{u}$ in $\fC$ is a pair $(\vecinfty{u}, U)$ where $U$ is a fixed object and $\vecinfty{u}$ is a family of $\fC$-arrows $\set{u^\infty_n\maps u_n \to U}_{n \in \nat}$ such that $u_n^\infty = u_m^\infty \cmp u_n^m$ for every $n \leq m$, as in Figure~\ref{fig:coned_sequence}.
A cone $(\vecinfty{u}, U)$ for $\vec{u}$ is a \emph{colimit} of $\vec{u}$ if, for every cone $(\vecinfty{v}, V)$ for $\vec{u}$, there is a unique $\fC$-arrow $f\maps U \to V$ such that $v^\infty_n = f \cmp u^\infty_n$ for every $n \in \nat$.

\begin{figure}[htp]
	\centering
	\begin{tikzcd}
		& & U & & \\
		& & & & \\
		& & & & \\
		u_0 \arrow[r, "u_0^1"] \arrow[rruuu, "u_0^\infty", pos=0.3] 
			& u_1 \arrow[r, "u_1^2"] \arrow[ruuu, "u_1^\infty", pos=0.3] 
			& u_2 \arrow[r, "u_2^3"] \arrow[uuu, "u_2^\infty", pos=0.4] 
			& u_3 \arrow[r] \arrow[luuu, "u_3^\infty"] 
			& ... \arrow[lluuu]
	\end{tikzcd}
	\caption{A cone for a sequence in a category. The diagram is commutative.}
	\label{fig:coned_sequence}
\end{figure}
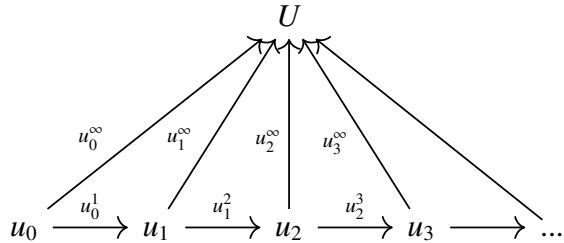

We adopt the standard convention and denote the automorphism group of an object $x$ by $\aut(x)$. Our main interest will be $\aut(U)$, where $U$ is a distinguished ``generic'' object (the precise meaning is described below). So, with one exception, the objects of categories will be denoted by small letters.

\subsection{The setup}

Throughout the paper we find ourselves in the following situation: we have a category of “small” objects $\fS$ and a fixed “large” object $U$, both living in an ambient category $\fL$.
The main theorems relate properties of $\fS$, like the (weak) Ramsey property, with properties of $U$, like extreme amenability of its automorphism group.
Note that (weak) Fraïssé theory follows the same pattern: properties of $\fS$, like the (weak) amalgamation property or existence of a (weak) Fraïssé sequence, are related to properties of the “generic” object $U$, like (weak) homogeneity and (weak) injectivity.

A connection between $\fS$ and $U$ in $\fL$ is established by fixing a \emph{coned sequence} in $(\fS, \fL)$: a triple $(\vec{u}, \vecinfty{u}, U)$ where $\vec{u}$ is a sequence in $\fS$ and $(\vecinfty{u}, U)$ is a cone for $\vec{u}$ in $\fL$.
The object $U$ can be thought of as a “quasi-limit” of $\vec{u}$ in $\fL$.
Sometimes it even is an actual colimit, but it is not necessary.

In this subsection we impose a minimalistic set of conditions on the coned sequence $(\vec{u}, \vecinfty{u}, U)$ for our theory to work.
A concise list is given in the following definition.
A rather long remark (that can be safely skipped) giving some insight follows.

\begin{df}
	Let $\fS \subs \fL$ be categories.
	We say that a coned sequence $(\vec{u}, \vecinfty{u}, U)$ in $(\fS, \fL)$ is \emph{matching} if it satisfies the following conditions.
	\begin{enumerate}
		\item[(F1)] For every $\fL$-arrow $\map f x U$ with $x \in \ob{\fS}$ there exist $n$ and an $\fS$-arrow $\map{\tilde{f}}{x}{u_n}$ such that $f = u_n^\infty \cmp \tilde{f}$.
		\item[(F2)] For every $n$ and every $\fS$-arrows $f, f'\maps x \to u_n$ such that $u^\infty_n \cmp f = u^\infty_n \cmp f'$ there is $n' \geq n$ such that $u^{n'}_n \cmp f = u^{n'}_n \cmp f'$.
		\item[(BF)] If $\ciag f$, $\ciag g$ are sequences of $\fS$-arrows such that $$\map{f_n}{u_{k_n}}{u_{\ell_n}}, \; \map{g_n}{u_{\ell_n}}{u_{k_{n+1}}}, \; g_n \cmp f_n = u_{k_n}^{k_{n+1}}, \; f_{n+1} \cmp g_n = u_{\ell_n}^{\ell_{n+1}},$$
		for some increasing cofinal sequences $\set{k_n}_{n \in \nat}, \set{\ell_n}_{n \in \nat} \subs \nat$,
		then there exists an $\fL$-arrow $f_\infty \in \aut(U)$ such that
		$$f_\infty \cmp u_{k_n}^\infty = u_{\ell_n}^\infty \cmp f_n \oraz f_\infty^{-1} \cmp u_{\ell_n}^\infty = u_{k_{n+1}}^\infty \cmp g_n$$
		for every $\ntr$.
		\item[(H)] For every $h \in \aut(U) \setminus \sn {\id U}$ there is $n$ such that $h \cmp u_n^\infty \ne u_n^\infty$.
	\end{enumerate}
	The letter F stands for ``factorization'', and (F1) and (F2) are called the ``factorization (existence) condition'' and the ``factorization uniqueness condition'', respectively.
	The letters BF stand for ``back-and-forth'' and H stands for ``Hausdorff'' (see Construction~\ref{aut_topology}).
\end{df}

\begin{remark} \label{RMmatching}
	Let us give some insight to the conditions defining a matching sequence.
	Given a category $\fS$, let us define the induced \emph{category of sequences} $\sigma_0\fS$.
	The objects are all sequences $\vec{u}$ in $\fS$.
	The morphisms are transformations between sequences modulo a certain equivalence.
	A transformation $\vec{\phi}\maps \vec{u} \to \vec{v}$ is a sequence of $\fS$-arrows $\set{\phi_n\maps u_n \to v_{\phi(n)}}_{n \in \nat}$, where $\set{\phi(n)}_{n \in \nat} \subs \nat$ is an increasing cofinal sequence, such that $v^{\phi(m)}_{\phi(n)} \cmp \phi_n = \phi_m \cmp u^m_n$ for every $n \leq m \in \nat$, as in Figure~\ref{fig:matching_morphisms}, i.e. it is a natural transformation from $\vec{u}$ to a subsequence of $\vec{v}$.
	The composition is obvious: $(\vec{\psi} \cmp \vec{\phi})_n = \psi_{\phi(n)} \cmp \phi_n$ for every $n \in \nat$.
	We say that two transformations $\vec{\phi}, \vec{\psi}\maps \vec{u} \to \vec{v}$ are equivalent (and we write $\vec{\phi} \approx \vec{\psi}$) if for every $n \in \nat$ there is $m \geq \phi(n), \psi(n)$ such that $v^m_{\phi(n)} \cmp \phi_n = v^m_{\psi(n)} \cmp \psi_n$.
	It is easy to see that this is a well-defined congruence of a category, and so defining morphisms of $\sigma_0\fS$ as transformations modulo this equivalence is correct.
	
	Note that we may identify every $\fS$-object $x$ with the constant sequence $\idvec{x}$ and every $\fS$-arrow $f\maps x \to y$ with the constant transformation $\idvec{f}\maps \idvec{x} \to \idvec{y}$.
	This way we identify $\fS$ with a subcategory of $\sigma_0\fS$.
	Moreover, this subcategory is full since every $\sigma_0\fS$-arrow $\vec{\phi}\maps x \to \vec{u}$ from a constant identity sequence is uniquely determined (as a transformation up to the equivalence) by the $\fS$-arrow $\phi_0\maps x \to u_{\phi(0)}$.
	Also note that every constant sequence $\idvec{x}$ admits the canonical limit cone $(\idvec{x}^\infty, x)$ in $\fL$ and that $(\idvec{x}, \idvec{x}^\infty, x)$ is a matching sequence in $(\fS, \fL)$.
	
	For two coned sequences $(\vec{u}, \vecinfty{u}, U)$ and $(\vec{v}, \vecinfty{v}, V)$ in $(\fS, \fL)$
	we consider the \emph{matching relation} $\matches$ between $\sigma_0\fS$-arrows $\vec{u} \to \vec{v}$ and $\fL$-arrows $U \to V$: we put 
		$$\vec{\phi} \matches \phi_\infty \quad\text{ if }\quad v^\infty_{\phi(n)} \cmp \phi_n = \phi_\infty \cmp u^\infty_n \text{ for every $n \in \nat$},$$
	see Figure~\ref{fig:matching_morphisms}.
	The matching relation is functorial: we have $\id{\vec{u}} \matches \id{U}$, and if $\vec{\phi} \matches \phi_\infty$ and $\vec{\psi} \matches \psi_\infty$, then $\vec{\psi} \cmp \vec{\phi} \matches \psi_\infty \cmp \phi_\infty$.
	
	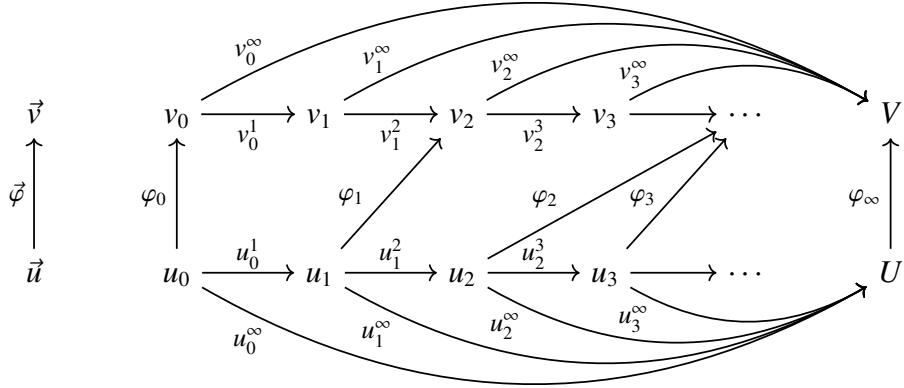
\begin{figure}
	\centering
	
	\begin{tikzpicture}[
			x = {(4.5em, 0)},
			y = {(0, 5em)},
			text diagram,
			label/.style = {
				edge label = {#1}, 
				every node/.style = {node font=\footnotesize},
			},
			label'/.style = {
				label = {#1},
				swap,
			},
			line width = rule_thickness,
		]
		\path
			++(0, 0) node (u0) {$u_0$}
			++(1, 0) node (u1) {$u_1$}
			++(1, 0) node (u2) {$u_2$}
			++(1, 0) node (u3) {$u_3$}
			++(1, 0) node (udots) {$\cdots$}
			++(1, 0) node (U) {$U$}
		(0, 0)
			++(0, 1) node (v0) {$v_0$}
			++(1, 0) node (v1) {$v_1$}
			++(1, 0) node (v2) {$v_2$}
			++(1, 0) node (v3) {$v_3$}
			++(1, 0) node (vdots) {$\cdots$}
			++(1, 0) node (V) {$V$}
		(u0)
			++(-1, 0) node (useq) {$\vec{u}$}
			++(0, 1) node (vseq) {$\vec{v}$}
		;
		\graph{
			(u0) ->[label=$u_0^1$] (u1) ->[label=$u_1^2$] (u2) ->[label=$u_2^3$] (u3) -> (udots),
			(v0) ->[label'=$v_0^1$] (v1) ->[label'=$v_1^2$] (v2) ->[label'=$v_2^3$] (v3) -> (vdots),
			(u0) ->[label=$\phi_0$] (v0),
			(u1) ->[label=$\phi_1$, pos=0.33] (v2),
			(u2) ->[label=$\phi_2$, pos=0.35] (vdots),
			(u3) ->[label=$\phi_3$, pos=0.35, inner sep=0.4ex] (vdots),
			(useq) ->[label=$\vec{\phi}$] (vseq),
			(U) ->[label=$\phi_\infty$] (V),
		};
		
		\begin{scope}[
				label/.append style={inner sep=0.4ex},
			]
			\graph{
				(u0) ->[label'=$u_0^\infty$, pos=0.1, bend right] (U),
				(u1) ->[label'=$u_1^\infty$, pos=0.1, bend right] (U),
				(u2) ->[label'=$u_2^\infty$, pos=0.1, bend right] (U),
				(u3) ->[label'=$u_3^\infty$, pos=0.1, bend right] (U),
			};
			\graph{
				(v0) ->[label=$v_0^\infty$, pos=0.1, bend left] (V),
				(v1) ->[label=$v_1^\infty$, pos=0.1, bend left] (V),
				(v2) ->[label=$v_2^\infty$, pos=0.1, bend left] (V),
				(v3) ->[label=$v_3^\infty$, pos=0.1, bend left] (V),
			};
		\end{scope}
	\end{tikzpicture}
	
	\caption{A morphism between sequences and a matching morphism between associated cones.}
	\label{fig:matching_morphisms}
\end{figure}
	
	Now let $(\vec{u}, \vecinfty{u}, U)$ be a coned sequence in $(\fS, \fL)$.
	For every $x \in \ob{\fS}$ the matching relation for $(\idvec{x}, \idvec{x}^\infty, x)$ and $(\vec{u}, \vecinfty{u}, U)$ is a function $\sigma_0\fS(x, \vec{u}) \to \fL(x, U)$: every $\sigma_0\fS$-arrow $\phi\maps x \to \vec{u}$ is determined by an $\fS$-arrow $f\maps x \to u_n$ for some $n$, and the unique $\fL$-map $\phi_\infty$ such that $\vec{\phi} \matches \phi_\infty$ is $u^\infty_n \cmp f$.
	The condition (F1) says that this matching function is surjective, and the condition (F2) says that the matching function is one-to-one.
	Together, (F1) and (F2) hold if and only if for every $x \in \ob{\fS}$ the matching relation is a bijection between $\sigma_0\fS(x, \vec{u})$ and $\fL(x, U)$.
	
	Moreover, under (F1), for every $\fL$-arrow $\phi_\infty\maps U \to U$ there is a transformation $\vec{\phi}\maps \vec{u} \to \vec{u}$ such that $\vec{\phi} \matches \phi_\infty$: for every $n \in \nat$ there is an $\fS$-arrow $\phi_n\maps u_n \to u_{\phi(n)}$ with $u^\infty_{\phi(n)} \cmp \phi_n = \phi_\infty \cmp u^\infty_n$ and we can make sure that the sequence $\set{\phi(n)}_{n \in \nat}$ is increasing and cofinal.
	Under (F2), for every $\fL$-arrow $\phi_\infty\maps U \to U$ there is at most one transformation $\vec{\phi}\maps \vec{u} \to \vec{u}$ up to the equivalence such that $\vec{\phi} \matches \phi_\infty$: for a different such transformation $\vec{\psi}$ and $n \in \nat$ we have $u^\infty_{\phi(n)} \cmp \phi_n = \phi_{\infty} \cmp u^\infty_n = u^\infty_{\psi(n)} \cmp \psi_n$, and so by (F2) there is $m$ such that $u^m_{\phi(n)} \cmp \phi_n = u^m_{\psi(n)} \cmp \psi_n$.
	Together, the matching relation for $(\vec{u}, \vecinfty{u}, U)$ is a function $\fL(U, U) \to \sigma_0\fS(\vec{u}, \vec{u})$.
	By the functoriality the function is a monoid homomorphism, and it restricts to a group homomorphism $F\maps \aut_{\fL}(U) \to \aut_{\sigma_0\fS}(\vec{u})$.
	Observe that $\sigma_0\fS$-automorphisms $\vec{u} \to \vec{u}$ correspond to the zig-zag sequences used in condition (BF), and so (BF) states that the mapping $F$ is surjective (or even without (F1) and (F2) that for every automorphism $\vec{\phi}\maps \vec{u} \to \vec{u}$ there is a matching automorphism $\phi_\infty\maps U \to U$).
	Similarly, (H) states that the mapping $F$ is one-to-one (here we cannot omit (F2)).
	Together, under (F1) and (F2), the conditions (BF) and (H) hold if and only if the matching relation is a group isomorphism between $\aut_{\sigma_0\fS}(\vec{u})$ and $\aut_{\fL}(U)$.
\end{remark}

In the following lemma we observe that the notion of a matching sequence is robust under isomorphism, and that the large object $U$ determines the sequence of small objects $\vec{u}$ uniquely.
\begin{lm}
	Let $(\vec{u}, \vecinfty{u}, U)$ be a matching sequence in $(\fS, \fL)$.
	\begin{enumerate}
		\item If $f\maps U \to V$ is an $\fL$-isomorphism, then $(\vec{u}, f \cmp \vecinfty{u}, V)$ is a matching sequence.
		\item If $\vec{\phi}\maps \vec{v} \to \vec{u}$ is an $\sigma_0\fS$-isomorphism, then $(\vec{v}, \vecinfty{u} \cmp \vec\phi, U)$ is a matching sequence.
		\item For every other matching sequence $(\vec{v}, \vecinfty{v}, V)$ and every $\fL$-isomorphism $\phi_\infty\maps U \to V$ there is a matching $\sigma_0\fS$-isomorphism $\vec{\phi}\maps \vec{u} \to \vec{v}$.
	\end{enumerate}
	
	\begin{proof}
		The proof is straightforward, though technical.
		For example, to obtain (BF) in (ii), for an $\sigma_0\fS$-automorphism $\psi\maps \vec{v} \to \vec{v}$, we consider the automorphism $(\vec{\phi} \cmp \vec{\psi} \cmp \vecinv{\phi})\maps \vec{u} \to \vec{u}$ and its matching automorphism $h\maps U \to U$.
		For every $n \in \nat$ we have 
			$$u^\infty_{\phi(\psi(\phi^{-1}(n)))} \cmp (\vec\phi \cmp \vec\psi \cmp \vecinv{\phi})_n = h \cmp u^\infty_n.$$
		Also, $\vecinv{\phi} \cmp \vec{\phi}$ is equivalent to $\id{\vec{u}}$, and so we have
		\begin{align*}
			(\vecinfty{u} \cmp \vec\phi)_{\psi(n)} \cmp \psi_n
				&= (\vecinfty{u} \cmp \vec\phi \cmp \vec\psi \cmp \vecinv\phi \cmp \vec\phi)_n \\
				&= u^\infty_{\phi(\psi(\phi^{-1}(\phi(n))))} \cmp (\vec\phi \cmp \vec\psi \cmp \vecinv\phi)_{\phi(n)} \cmp \phi_n 
				= h \cmp u^\infty_{\phi(n)} \cmp \phi_n
				= h \cmp (\vecinfty{u} \cmp \vec\phi)_n,
		\end{align*}
		which we wanted.
		
		Claim (iii) in the case of automorphisms was already discussed in Remark~\ref{RMmatching}, and the proof for isomorphisms is analogous.
	\end{proof}
\end{lm}

It is not always the case that a sequence $\vec{u}$ can be completed to a matching sequence at most one way (up to an isomorphism).
However, it may have at most one colimit $(\vecinfty{u}, U)$ in $\fL$.
Moreover, if $(\vecinfty{u}, U)$ is a colimit of $\vec{u}$ in $\fL$, then $(\vec{u}, \vecinfty{u}, U)$ satisfies (BF) and (H).
In fact, for every transformation $\vec{\phi}\maps \vec{u} \to \vec{v}$ and every coned sequence $(\vec{v}, \vecinfty{v}, V)$ there is a unique matching $\fL$-arrow $\phi_\infty\maps U \to V$ since such matching arrow is the same thing as the colimit factorizing arrow for the cone $(\vecinfty{v} \cmp \vec{\phi}, V)$ for $\vec{u}$.

Let us continue discussing how to obtain the conditions defining a matching sequence.
Observe that if $\fL$ consists of monomorphisms, we obtain (F2) for free: if $u^\infty_n \cmp f = u^\infty_n \cmp f'$, then $f = f'$ since $u^\infty_n$ is a monomorphism.
Let us also recall the notion of \emph{finitely presentable} object \cite[Definition~1.1]{AdamRos} in a category $\fL$:
it is an object $x$ such that for every directed colimit $\vec{u} = (u^j_i, u^\infty_i, U)_{i \leq j \in I}$ every $\fL$-arrow $f\maps x \to U$ essentially uniquely factorizes through $\vec{u}$, i.e. there is an arrow $\tilde{f}\maps x \to u_i$ for some $i$ such that $u^\infty_i \cmp \tilde{f} = f$, and for every other such arrow $g\maps x \to u_j$ there is $k \geq i, j$ such that $u^k_i \cmp \tilde{f} = u^k_j \cmp g$.
In other words, a finitely presented object satisfies analogues of (F1) and (F2) for every directed system with a colimit.
Hence, if $\fS \subs \fL$ is a full subcategory whose objects are finitely presented in $\fL$, and $(\vec{u}, \vecinfty{u}, U)$ is a sequence in $\fS$ with a colimit in $\fL$, then it satisfies (F1) and (F2).

We summarize this discussion in the following lemma.

\begin{lm}
	Let $(\vec{u}, \vecinfty{u}, U)$ be a coned sequence in $(\fS, \fL)$.
	\begin{enumerate}
		\item If every map $u^\infty_n$ is a monomorphism, then $(\vec{u}, \vecinfty{u}, U)$ satisfies (F2).
		\item If $(\vecinfty{u}, U)$ is a colimit of $\vec{u}$ in $\fL$, then $(\vec{u}, \vecinfty{u}, U)$ satisfies (BF) and (H).
		\item If $\fS$ is a full subcategory of $\fL$, every $\fS$-object is finitely presentable in $\fL$, and $(\vecinfty{u}, U)$ is a colimit, then $(\vec{u}, \vecinfty{u}, U)$ satisfies also (F1) and (F2), and so is matching.
	\end{enumerate}
\end{lm}

A classical example of the described situation are categories of finitely generated $L$-structures in a first-order language $L$, with embeddings or one-to-one homomorphisms as arrows between objects.
Let $\fL$ be the category of all $L$-structures and homomorphisms.
Recall that a sequence $\vec{u}$ in $\fL$ consisting of embeddings can be without loss of generality viewed as an increasing $\subseteq$-chain of substructures.
Then its colimit in $\fL$ is the union of the chain.
In particular, the colimit maps $u^\infty_n$ can be taken to be inclusion maps.
So if we consider the wide subcategory $\fL' \subs \fL$ of all embeddings between $L$-structures, then $\fL$-colimits of $\fL'$-sequences (or of any directed systems in $\fL'$) are also $\fL'$-colimits.
Moreover, finitely generated $L$-structures are finitely presentable in $\fL'$: if an $L$-structure $A$ is a directed union of its substructures $(A_i \subs A)_{i \in I}$ and $B \subs A$ is a substructure generated by a finite set $F \subs A$, then every generator $x \in F$ is contained in some $A_{i_x}$, and so all of them are contained in some $A_i$.
Hence, $B \subs A_i$.
Together, we obtain the following.

\begin{con}[$\sigma$-closure] \label{sigma_closure}
	Let $L$ be a first-order language, let $\fL$ be the category of all $L$-structures and homomorphisms, and let $\fS \subs \fL$ be a subcategory of some finitely generated $L$-structures and all embeddings between them.
	We define $\sigma\fS$ to be the category of all $\fL$-structures that are $\fL$-colimits of $\fS$-sequences with all embeddings between them as morphisms.
	
	We have that every $\fS$-sequence $\vec{u}$ has a common colimit in $\sigma\fS$ and $\fL$, that every $\sigma\fS$-object $U$ is a colimit of a $\fS$-sequence, and that every such colimit sequence $(\vec{u}, \vecinfty{u}, U)$ in $(\fS, \sigma\fS)$ or equivalently in $(\fS, \fL)$ is matching.
	In fact, mapping every $\fS$-sequence to its colimit and every transformation of $\fS$-sequences to the unique matching $\fL$-arrow induces a functor $\sigma_0\fS \to \sigma\fS$, which is an equivalence of categories in this case. Hence, $\sigma\fS$ can be viewed as a concrete realization of $\sigma_0\fS$ in $\fL$.
	
	Note $\sigma\fS$-objects are exactly countably generated $L$-structures such that every finite subset is contained in a substructure from $\fS$.
	If $L$ is a relational language and $\ob{\fS}$ is hereditary, then $\sigma\fS$-objects are all countable $L$-structures whose finite substructures are in $\ob{\fS}$.
	If $\fS$ is the category of all finite groups, then $\sigma\fS$ is the category of all countable locally finite groups (meaning that every finite subset generates a finite subgroup).
	
	The construction of $\sigma\fS$ can be done also with one-to-one homomorphisms instead of embeddings.
	In that case the inclusion maps in the increasing chain may also refine the structure (make the relations finer).
\end{con}

\aseparator

\begin{con}[The topology of $\aut(U)$] \label{aut_topology}
	Let $(\vec{u}, \vecinfty{u}, U)$ be a matching sequence in $(\fS, \fL)$.
	The automorphism group $G := \aut(U)$ has a natural topology defined by the following neighborhood base of the identity:
	\[
		V_n = \bigsetof{g \in G}{g \cmp u_n^\infty = u_n^\infty}.
	\]
	The topology is Hausdorff, thanks to condition (H).
	Note that each $V_n$ is a subgroup of $G$, therefore we obtain a non-archimedean topological group.
	Note also that $G$ is completely metrizable and that it may not be separable, unless all the hom-sets $\fS(u_n,u_m)$ are countable (see the next lemma).
	
	To show that the open subgroups $V_n$ indeed form a base of the identity of the group topology, we need to show that for every $n \in \nat$ and $g \in G$ there is $m \in \nat$ such that $g^{-1} \cmp V_m \cmp g \subs V_n$, i.e. for every $h \in G$, if $h \cmp u_m^\infty = u_m^\infty$, then $h \cmp g \cmp u_n^\infty = g \cmp u_n^\infty$.
	By (F1), $g \cmp u_n^\infty = u^\infty_m \cmp f$ for some $m \in \nat$ and an $\fS$-arrow $f\maps u_n \to u_m$.
	Hence,
	\[
		h \cmp (g \cmp u_n^\infty) = (h \cmp u^\infty_m) \cmp f = u^\infty_m \cmp f = g \cmp u^\infty_n.
	\]
	Analogously, we can show that the topology on $G$ does not depend on the choice of a matching sequence for $U$:
	if $(\vec{v}, \vecinfty{v}, U)$ is another matching sequence, then the subgroups $V'_n = \set{g \in G: g \cmp v^\infty_n = v^\infty_n}$ induce the same topology.
	For every $n \in \nat$ we have $v^\infty_n = u^\infty_m \cmp f$ for some $m \in \nat$ and an $\fS$-arrow $f\maps v_n \to u_m$.
	Hence, $V_m \subs V'_n$.
	
	Note that in the case when $\fS$ is a category of finite first-order structures and all embeddings, and $(\vecinfty{u}, U)$ is a colimit, then $U$ is essentially the union of the chain of finite structures $u_n$, and the condition $g \cmp u^\infty_n = u^\infty_n$ says that the automorphism $g$ fixes the elements of $u_n$.
	Therefore, the induced topology is the topology of pointwise convergence inherited from $U^U$ where $U$ has the discrete topology.
\end{con}

\begin{lm}
	Let $(\vec{u}, \vecinfty{u}, U)$ be a matching sequence in $(\fS, \fL)$.
	\begin{enumerate}
		\item The topology on $\aut(U)$ is completely metrizable.
		\item If all the hom-sets $\fS(u_n, u_m)$ are countable, then $\aut(U)$ is also separable and so a Polish group.
	\end{enumerate}
	
	\begin{proof}
		The group $G := \aut(U)$ is metrizable by Birkhoff–Kakutani theorem~\cite[Theorem~3.3.12]{ArhangelskiiTkachenko} since we have a countable neighborhood base of the identity.
		We show that the two-sided uniformity of $G$ is induced by a complete metric.
		Let $\set{h_n}_{n \in \nat}$ be a Cauchy sequence with respect to the two-sided uniformity, i.e. for every $m \in \nat$ there is $\phi(m) \in \nat$ such that $h_{n'} \in (h_n \cmp V_m) \cap (V_m \cmp h_n)$ for every $n, n' \geq \phi(m)$, and so $h_n \cmp u^\infty_m = h_{n'} \cmp u^\infty_m$ and $h_n^{-1} \cmp u^\infty_m = h_{n'}^{-1} \cmp u^\infty_m$.
	
		Put $k_0 := 0$.
		By (F1) there is an $\fS$-arrow $f_0\maps u_{k_0} \to u_{\ell_0}$ for some $\ell_0 > k_0$ such that $u^\infty_{\ell_0} \cmp f_0 = h_{\phi(k_0)} \cmp u^\infty_{k_0}$.
		Then there is an $\fS$-arrow $g'_0\maps u_{\ell_0} \to u_{k'_1}$ for some $k'_1 > \ell_0$ such that $u^\infty_{k'_1} \cmp g'_0 = h_{\phi(\ell_0)}^{-1} \cmp u^\infty_{\ell_0}$.
		Together, we have
		\[
			u^\infty_{k'_1} \cmp g'_0 \cmp f_0 = h_{\phi(\ell_0)}^{-1} \cmp u^\infty_{\ell_0} \cmp f_0 = h_{\phi(\ell_0)}^{-1} \cmp h_{\phi(k_0)} \cmp u^\infty_{k_0} = u^\infty_{k_0}
		\]
		since $h_{\phi(k_0)} \cmp u^\infty_{k_0} = h_{\phi(\ell_0)} \cmp u^\infty_{k_0}$ since $\phi(\ell_0) \geq \phi(k_0)$.
		By (F2) there is $k_1 \geq k'_1$ such that $u^{k_1}_{k'_1} \cmp g'_0 \cmp f_0 = u^{k_1}_{k_0}$.
		We put $g_0 := u^{k_1}_{k'_1} \cmp g'_0$ so that $g_0 \cmp f_0 = u^{k_1}_{k_0}$.
		We continue this way and build $\fS$-sequences $\set{f_n}_{n \in \nat}$, $\set{g_n}_{n \in \nat}$ as in (BF).
		Hence, there is an automorphism $h_\infty \in G$ such that for every $m \in \nat$ we have
		\begin{align*}
			h_\infty \cmp u^\infty_{k_m} &= u^\infty_{\ell_m} \cmp f_m = h_{\phi(k_m)} \cmp u^\infty_{k_m} = h_n \cmp u^\infty_{k_m}
			 && \text{for every $n \geq \phi(k_m)$,} \\
			h_\infty^{-1} \cmp u^\infty_{\ell_m} &= u^\infty_{k_{m + 1}} \cmp g_m = h_{\phi(\ell_m)}^{-1} \cmp u^\infty_{\ell_m} = h_n^{-1} \cmp u^\infty_{\ell_m}
			 && \text{for every $n \geq \phi(\ell_m)$.}
		\end{align*}
		Since the sequences $\set{k_m}_{m \in \nat}$, $\set{\ell_m}_{m \in \nat}$ are strictly increasing, it follows that $\lim_{n \to \infty} h_n = h_\infty$.
		
		If the hom-sets $\fS(u_n, u_m)$ for $n \leq m$ are countable, then $G$ is separable since for every $n$, $\set{g \cmp V_n: g \in G}$ is a countable cover by basic open sets.
		This is because we have the one-to-one map $g \cmp V_n \mapsto g \cmp u^\infty_n$, and every arrow $g \cmp u^\infty_n$ is of the form $u^\infty_m \cmp f$ for some $f \in \fS(u_n, u_m)$.
	\end{proof}
\end{lm}

\begin{con}[The topological group $G(\vec{u}, \fS)$] \label{con:Gus}
	We observe that given matching sequences $(\vec{u}, \vecinfty{u}, U)$ in $(\fS, \fL)$ and $(\vec{v}, \vecinfty{v}, V)$ in $(\fS, \fL')$, every isomorphism of sequences $\vec{\phi}\maps \vec{u} \to \vec{v}$ canonically induces an isomorphism of the topological groups $\aut_{\fL}(U)$ and $\aut_{\fL'}(V)$.
	Moreover, any sequence $\vec{u}$ in any category $\fS$ admits a canonical matching sequence $(\vec{u}, \set{\vec{\imath}_n}_{n \in \nat}, \vec{u})$ in $(\fS, \sigma_0\fS)$ (see Remark~\ref{RMmatching}).
	Therefore, up to a canonical isomorphism, every $\fS$-sequence $\vec{u}$ determines a topological group $G(\vec{u}, \fS)$ that is the automorphism group of every associated matching sequence $(\vec{u}, \vecinfty{u}, U)$ in every $(\fS, \fL)$.
	Moreover, isomorphic sequences determine isomorphic topological groups.
	
	\begin{proof}
		Recall that $\vec{u}$ is both a sequence in $\fS \subs \sigma_0\fS$ and an object in $\sigma_0\fS$.
		For every $n$ we let $\vec{\imath}_n\maps u_n \to \vec{u}$ be the $\sigma_0\fS$-arrow corresponding to $\id{u_n}\maps u_n \to u_n$,
			i.e. $(\vec{\imath}_n)_k = u^{\max(k, n)}_n$ for $k \in \omega$.
		It is easy to see that $\vec{u}$ (viewed as a $\sigma_0 \fS$-object) is the colimit of itself (viewed as an $\fS$-sequence).
		Also for every sequences $\vec{u}, \vec{v}$ the corresponding matching relation is the identity on $\sigma_0\fS(\vec{u}, \vec{v})$.
		Hence by Remark~\ref{RMmatching}, $(\vec{u}, \set{\vec{\imath}_n}_{n \in \nat}, \vec{u})$ is a colimiting matching sequence in $(\fS, \sigma_0\fS)$.
		
		For every matching sequence $(\vec{u}, \vecinfty{u}, U)$ in $(\fS, \fL)$ the matching relation between $\aut(\vec{u})$ and $\aut(U)$ is an isomorphism of groups by Remark~\ref{RMmatching}.
		Recall that we write $\vec{\phi} \approx \vec{\psi}$ for every two transformations between sequences that are equivalent, i.e. representing the same $\sigma_0\fS$-arrow.
		Since for every matching pair of automorphisms $\vec{\phi}\maps \vec{u} \to \vec{u}$ and $\phi_\infty\maps U \to U$ we have
		\[
			\phi_\infty \cmp u^\infty_n = u^\infty_n \iff u^\infty_{\phi(n)} \cmp \phi_n = u^\infty_n \iff \exists m\; u^m_{\phi(n)} \cmp \phi_n = u^m_n \iff \vec{\phi} \cmp \vec{\imath}_n \approx \vec{\imath}_n
		\]
		for every $n \in \omega$, it follows that matching relation is also a homeomorphism between the automorphism groups.
		To see $\vec{\phi} \cmp \vec{\imath}_n \approx \vec{\imath}_n$, let $k \in \omega$ and $k' := \max(k, n)$, and note that for every $m \geq \phi(k')$ we have $u^m_{\phi(k')} \cmp (\vec{\phi} \cmp \vec{\imath}_n)_k = u^m_{\phi(k')} \cmp (\phi_{k'} \cmp u^{k'}_n) = u^m_{\phi(n)} \cmp \phi_n$, while $u^m_{k'} \cmp (\vec{\imath}_n)_{k'} = u^m_{k'} \cmp u^{k'}_n = u^m_n$.
		
		Finally, the isomorphism $\vec{\phi}\maps \vec{u} \to \vec{v}$ of $\fS$-sequences induces an isomorphism $\vec{\psi} \mapsto \vec{\phi} \cmp \vec{\psi} \cmp \vecinv{\phi}$ between $\aut(\vec{u})$ and $\aut(\vec{v})$.
		This isomorphism is also a homeomorphism.
		Let $\set{\vec{\imath}_n}_{n \in \nat}$ and $\set{\vec{\jmath}_n}_{n \in \nat}$ be the cones of the canonical matching sequences for $\vec{u}$ and $\vec{v}$.
		For every $n \in \nat$ we put $m := (\vecinv{\phi})(n)$, and we have $\vec{\jmath}_{\phi(m)} \cmp \phi_m \cmp (\vecinv{\phi})_n \approx \vec{\jmath}_n$.
		Hence, if $\vec{\psi}$ fixes $\vec{\imath}_m$, then $\vec{\phi} \cmp \vec{\psi} \cmp \vecinv{\phi}$ fixes $\vec{\phi} \cmp \vec{\imath}_m \approx \vec{\jmath}_{\phi(m)} \cmp \phi_m$, and so it fixes also $\vec{\jmath}_{\phi(m)} \cmp \phi_m \cmp (\vecinv{\phi})_n \approx \vec{\jmath}_n$.
	\end{proof}
\end{con}

\begin{remark}
	As a final remark we observe that given a matching sequence $(\vec{u}, \vecinfty{u}, U)$ in $(\fS, \fL)$, it is sufficient for our considerations to use only the part $\fL' \subs \fL$ where $\ob{\fL'} = \ob{\fS} \cup \sn U$, $\fS$ is a full subcategory of $\fL'$, for every $x \in \ob{\fS}$ we have $\fL'(x, U) = \fL(x, U) = \bigcup_{n \in \nat} (u^\infty_n \cmp \fS(x, u_n))$, and $\fL'(U, U) = \aut_{\fL}(U)$.
	In the cases we are not interested in the generic object U, only in its automorphism group, we may ignore $\fL$ completely – it is enough to consider $\fS$, $\vec{u}$, and the topological group $G(\vec{u}, \fS)$.
	On the other hand, when it comes to applications, one usually has in mind a larger subcategory $\fL$ consisting of all colimits of sequences in $\fS$.
\end{remark}

\subsection{Weak Fraïssé theory}

In this section we summarize the definitions and theorems of the \emph{weak Fraïssé theory}~\cite{KweakFra} that will be used later in the paper.
We recall the notion of a \emph{weak Fraïssé category} which is a generalization of a \emph{Fraïssé category}~\cite{K40} (which is itself an abstraction of the classical notion of a \emph{Fraïssé class}~\cite[\S 7]{Hodges} of first-order structures).
A weak Fraïssé category $\fS$ can be characterized by existence of a \emph{weak Fraïssé sequence} $\vec{u}$.
First we need to recall the concepts of \emph{(weak) amalgamation} and \emph{(weak) domination}.

The weak amalgamation property was introduced by Ivanov~\cite{Ivanov} and later independently by Kechris and Rosendal~\cite{KechrisRosendal}, in connection with existence of generic automorphisms of homogeneous first-order structures.
Since then, ideas of weak Fraïssé theory have been developed and used in many works, see for example \cite{Kruckman}, \cite{DiLiberti}, \cite{KK}, \cite{KweakFra}, \cite{WeakEx}, \cite{PT22}, \cite{Malicki}.

\begin{df}[Amalgamable arrows]
	An arrow $\map e z {z'}$ is called \emph{amalgamable} if for every arrows $\map f {z'} x$, $\map g {z'} y$ there exist arrows $\map {f'} x w$, $\map {g'} y w$ satisfying
	$$f' \cmp f \cmp e = g' \cmp g \cmp e.$$
	An object $z$ is \emph{amalgamable} if $\id z$ is an amalgamable arrow.
	A category $\fS$ has the \emph{amalgamation property} if all identity arrows are amalgamable.
	A category with a weak \fra\ sequence has the \emph{weak amalgamation property}, namely, for every object $a$ there exists an amalgamable arrow with domain $a$.
\end{df}

\begin{df}[Weak Fraïssé sequence]
	A sequence $\vec{u}$ in a category $\fS$ is called a \emph{weak Fraïssé sequence} if is satisfies the following conditions.
	\begin{enumerate}
		\item[(W0)] (\emph{Cofinality}) For every $x \in \ob{\fS}$ there is $\ntr$ such that $\fS(x,u_n) \nnempty$.
		\item[(W1)]	(\emph{Weak absorption}) For every $n$ there is $m \goe n$ such that for every $\fS$-arrow $\map f {u_m} y$ there are $\ell > m$ and an $\fS$-arrow $\map g y {u_\ell}$ such that $g \cmp f \cmp u_n^m = u_n^\ell$.
	\end{enumerate}
	A weak Fraïssé sequence $\vec{u}$ is called \emph{normalized} if $m = n + 1$ works in condition (W1).
	It is called a \emph{\fra\ sequence}~\cite{K40} if one can always take $m = n$ in condition (W1).
	Note that every cofinal subsequence of a weak \fra\ sequence is weak \fra, therefore we may restrict attention to normalized weak \fra\ sequences.
\end{df}

\begin{remark}
	Note that a number $m \geq n$ works in (W1) if and only if the arrow $u^m_n$ is amalgamable.
	To prove that $u^m_n$ is amalgamable one can consider $\map{f_0}{u_m}{y_0}$, $\map{f_1}{u_m}{y_1}$, obtaining $\map{g_0}{y_0}{u_\ell}$, $\map{g_1}{y_1}{u_\ell}$ satisfying
		$$g_0 \cmp f_0 \cmp u_n^m = u_n^\ell = g_1 \cmp f_1 \cmp u_n^m.$$
	The other implication is essentially \cite[Lemma~3.10]{KweakFra}.
\end{remark}

\begin{remark} \label{RMseq}
	By \cite[Proposition~3.6 and Corollary~3.12]{KweakFra} a sequence isomorphic to a weak Fraïssé sequence is itself a weak Fraïssé sequence, and every two weak Fraïssé sequences are isomorphic.
\end{remark}

The next definition is analogous to the definition of weak Fraïssé sequence, but applies to subcategories.
In fact, both concepts could be unified by a notion of a \emph{weakly dominating functor}.

\begin{df}[Weakly dominating subcategory]
	A subcategory $\fC \subs \fS$ is called \emph{weakly dominating} if it satisfies the following conditions.
	\begin{enumerate}
		\item[(D0)] (\emph{Cofinality}) For every $x \in \ob{\fS}$ there is $y \in \ob{\fC}$ such that $\fS(x, y) \nnempty$.
		\item[(D1)]	(\emph{Weak absorption}) For every $x \in \ob{\fC}$ there is a $\fC$-arrow $e\maps x \to y$ such that for every $\fS$-arrow $f\maps y \to z$ there is an $\fS$-arrow $g\maps z \to w$ such that $g \cmp f \cmp e \in \fC$.
	\end{enumerate}
	$\fC$ is called \emph{dominating} if additionally $e = \id{x}$ works in (D1).
	Note that every full cofinal subcategory is dominating.
\end{df}

\begin{lm}[{\cite[Lemma~3.4]{KweakFra}}] \label{thm:fseq_in_dominating}
	Let $\fC \subs \fS$ be a weakly dominating subcategory.
	Then every weak Fraïssé sequence in $\fC$ is also a weak Fraïssé sequence in $\fS$.
\end{lm}

The following lemma, which we shall use later, was essentially proved in the proof of \cite[Proposition~2.7]{KweakFra}.
\begin{lm} \label{thm:amalgamable_arrow}
	Let $\fC \subs \fS$ be a dominating subcategory and let $e\maps z \to z'$ be a $\fC$-arrow.
	If $e$ is amalgamable in $\fC$, then $e$ is amalgamable in $\fS$.
	
	\begin{proof}
		Let $f\maps z' \to x$ and $g\maps z' \to y$ be $\fS$-maps.
		By domination there are $\fS$-arrows $f'\maps x \to x'$ and $g'\maps y \to y'$ such that $f' \cmp f,\, g' \cmp g \in \fC$.
		Hence, there are $\fC$-arrows $f''\maps x' \to w$ and $g''\maps y' \to w$ such that $(f'' \cmp f') \cmp f \cmp e = (g'' \cmp g') \cmp g \cmp e$.
	\end{proof}
\end{lm}

\begin{df}[Weak Fraïssé category]
	We say that $\fS$ is a \emph{weak Fraïssé category} if it is \emph{directed} (i.e. for every $x, y \in \ob{\fS}$ there is $z \in \ob{\fS}$ such that $\fS(x, z) \nnempty$ and $\fS(y, z) \nnempty$), has the weak amalgamation property, and is weakly dominated by a countable subcategory.
	
	The category $\fS$ is a \emph{Fraïssé category} if it additionally has the amalgamation property.
	In this case, it is even dominated by a countable subcategory.
\end{df}

In the classical case of categories of structures and embeddings, being directed is often called the \emph{joint embedding property}.
Note that a category is identified with its collection of morphisms, so a category is countable if it has countably many morphisms (as opposed to just having countably many objects).
However, for a category of finite structures, having countably many isomorphism types is sufficient for being dominated by a countable subcategory.

\begin{tw}
	The following conditions are equivalent for a category $\fS$.
	\begin{enumerate}
		\item $\fS$ is a weak Fraïssé category.
		\item $\fS$ is weakly dominated by a countable weak Fraïssé category.
		\item $\fS$ has a weak Fraïssé sequence.
	\end{enumerate}
\end{tw}

This is \cite[Theorem~3.7]{KweakFra} together with the fact (implicitly used in the proof) that every countable weakly dominating subcategory of a weak Fraïssé category can be extended by countably many arrows to become directed and to have the weak amalgamation property.

\begin{con}[the topological group $G(\fS)$] \label{con:GS}
	Let $\fS$ be a weak Fraïssé category.
	By Construction~\ref{con:Gus} and Remark~\ref{RMseq} we have that every two weak Fraïssé sequences $\vec{u}$, $\vec{v}$ are isomorphic, and so $G(\vec{u}, \fS)$ are $G(\vec{v}, \fS)$ are isomorphic topological groups.
	Hence, we may denote this topological group determined uniquely up to isomorphism by $G(\fS)$.
	Recall that it is isomorphic to $\aut(U)$ for every matching weak Fraïssé sequence $(\vec{u}, \vecinfty{u}, U)$ in every $(\fS, \fL)$.
\end{con}

Next we recall the key properties of a \emph{weak Fraïssé limit}, generalizing the extension property / injectivity and homogeneity from Fraïssé theory.
Again, see \cite{KweakFra} for details.

\begin{df}
	Let $\fS \subs \fL$ be categories.
	An $\fL$-object $U$ is 
	\begin{itemize}
		\item \emph{cofinal} in $(\fS, \fL)$ if $\fL(x, U) \nnempty$ for every $\fS$-object $x$,
		\item \emph{weakly injective} in $(\fS, \fL)$ if for every $\fL$-arrow $f\maps x \to U$ from an $\fS$-object there is an $\fS$-arrow $e\maps x \to x'$ such that for every $\fS$-arrow $g\maps x' \to y$ there is an $\fL$-arrow $h\maps y \to U$ such that $f = h \cmp g \cmp e$,
		\item \emph{weakly homogeneous} if $(\fS, \fL)$ if for every $\fL$-arrow $f\maps x \to U$ from and $\fS$-object there is an $\fS$-arrow $e\maps x \to x'$ and an $\fL$-arrow $f'\maps x' \to U$ with $f = f' \cmp e$ such that for every $\fL$-arrow $g\maps x' \to U$ there is $h \in \aut(U)$ such that $f = h \cmp g \cmp e$.
	\end{itemize}
	Note that every cofinal weakly homogeneous object is weakly injective, and that the arrow $e$ witnessing the weak homogeneity for $f$ has the property that for every $\fL$-arrows $g, g'\maps x' \to U$ there is $h \in \aut(U)$ such that $h \cmp g = g'$.
	We say that $U$ is \emph{homogeneous at $e$} in that case.
	So $U$ is weakly homogeneous if every $\fL$-arrow from an $\fS$-object to $U$ factorizes through an $\fS$-arrow $U$ is homogeneous at.
\end{df}

\begin{tw}[Characterization of the weak Fraïssé limit] \label{thm:Flim}
	For a matching sequence $(\vec{u}, \vecinfty{u}, U)$ in a pair of categories $(\fS, \fL)$, the following conditions are equivalent:
	\begin{enumerate}
		\item $\vec{u}$ is a weak Fraïssé sequence in $\fS$.
		\item $U$ is a cofinal and weakly injective object in $(\fS, \fL)$.
		\item $U$ is a cofinal and weakly homogeneous object in $(\fS, \fL)$.
	\end{enumerate}
	Moreover, given the conditions above hold, we have the following.
	\begin{enumerate}[label=\rm(\alph*)]
		\item There exists an $\fL$-arrow $X \to U$ for every $\fL$-object $X$ that is an $\fS$-object or is an $\fL$-colimit of a sequence of amalgamable arrows in $\fS$.
		\item $U$ is homogeneous at an $\fS$-arrow $e$ if and only if $e$ is amalgamable.
			Hence, $e$ works for $f\maps x \to U$ in the weak homogeneity of $U$ if and only if $e$ is amalgamable and $f$ factorizes through $e$.
		\item An $\fS$-arrow $e\maps x \to x'$ works for $f\maps x \to U$ in the weak injectivity of $U$ if and only if $e$ is amalgamable and $f$ factorizes through $e$.
		In particular, $\id{x}$ for an amalgamable object $x$ works for every arrow $f$.
	\end{enumerate}
\end{tw}

The theorem is essentially proved in \cite{KweakFra}: see Theorem~4.2, Corollary~4.5, Theorem~4.6, and Theorem~5.1.
However, we use weaker assumptions here – we neither assume that $(\vecinfty{u}, U)$ is a colimit of $\vec{u}$, nor that $\fS$-sequences have colimits in $\fL$, nor that $\fL$-arrows are monic.
The only assumptions that are used in the original proof are covered by the notion of a matching sequence.
Given all the necessary definitions and conditions, the proof is quite direct.

\begin{remark} \label{rm:Flim}
	In the case that $\fS$ is a category of some finitely generated first-order structures and all embeddings (or all one-to-one homomorphism) and $\fL = \sigma\fS$ (as defined in Construction~\ref{sigma_closure}), then $\fL$-objects are exactly colimits of $\fS$-sequences, and every such colimit sequence $(\vec{u}, \vecinfty{u}, U)$ in $(\fS, \fL)$ is matching, so Theorem~\ref{thm:Flim} applies.
	Moreover, the weak Fraïssé limit $U$ is unique up to isomorphism in $\fL$ in this case, and every $\fS$-sequence with colimit $U$ is weak Fraïssé.
	If additionally $\fS$ has the amalgamation property (and so is a Fraïssé category), there is an $\fL$-map $X \to U$ for every $\fL$-object $X$, i.e. the Fraïssé limit $U$ is cofinal in $\fL$.
\end{remark}

\section{Main results}

The purpose of this section is to prove the announced characterization of extreme ame\-na\-bi\-lity.
We first define the weak Ramsey property of $\fS$ (see Definition~\ref{DFweakRmsyPy} below) and prove that it is equivalent to a Ramsey-like property for $\fL$-arrows into $U$.
We also note that the weak Ramsey property implies the weak amalgamation property.
Next we show that these properties are equivalent to the extreme amenability of $\aut(U)$.
Together, the weak Ramsey property of a weak Fraïssé category $\fS$ is equivalent to the extreme amenability of $G(\fS)$.

\subsection{The weak Ramsey property}

For an arrow $\al\maps a \to a'$ and an object $b$ in a category $\fC$ we shall write $\fC(\al, b)$ as a shortcut for $\fC(a', b) \cmp \al$.
Below we introduce the main definition.

\begin{df} \label{DFweakRmsyPy}
	We say that a category $\fS$ has the \emph{weak Ramsey property} if for every $a \in \ob{\fS}$ there exists an $\fS$-arrow $\al\maps a \to a'$ satisfying:
	\begin{enumerate}
		\item[(wR)] For every $b \in \ob{\fS}$, for every $k \in \nat$, for every finite $F \subs \fS(\al,b)$ there is $v \in \ob{\fS}$ such that for every $\phi\maps \fS(\al,v) \to k$ there exists $e\maps b \to v$ such that $\phi$ is constant on $e \cmp F$.
	\end{enumerate}
	We call such $\al$ a \emph{Ramsey arrow}.
	We say that $\fS$ has the \emph{Ramsey property} if for every $a \in \ob{\fS}$ the identity $\id{a}$ is a Ramsey arrow.
	
	Recall that a category $\fS$ is \emph{locally finite} if $\fS(a,b)$ is finite for every $a,b \in \ob{\fS}$.
	Note that for a locally finite category $\fS$, the condition (wR) simplifies since we may consider just $F = \fS(\al, b)$.
	Hence, our notion of Ramsey property simplifies to the standard one (e.g. \cite[\S 3]{MasKPT}).
\end{df}

The following lemma generalizes \cite[Theorem~4.2 (i)]{Nesetril}.
\begin{lm} \label{LMerigoeg}
	Let $\fS$ be a directed category and let $\al\maps a \to a'$ be an $\fS$-arrow.
	If $\al$ is a Ramsey arrow, then $\al$ is amalgamable.
	Hence, a directed category with the weak Ramsey property has the weak amalgamation property.
	
	\begin{pf}
		Suppose $\al\maps a \to a'$ is a Ramsey arrow and let $k = 2$.
		Fix $f_0, f_1 \in \fS$ with $\dom(f_0) = a' = \dom(f_1)$.
		Using directedness, choose $b \in \ob{\fS}$ and $g_0, g_1 \in \fS$ such that $g_i \cmp f_i \in \fS(a',b)$ for $i=0,1$. Of course, $g_0,g_1$ are independent of $f_0,f_1$ and there is no reason for the equality $g_0 \cmp f_0 = g_1 \cmp f_1$. Let $F = \dn{g_0 \cmp f_0 \cmp \al}{g_1 \cmp f_1 \cmp \al}$.
		
		Now find $v \in \ob{\fS}$ from the weak Ramsey property applied to $F$.
		Define $\map{\phi}{\fS(\al,v)}{2}$ by setting $\phi(g) = 1$ if and only if $g = g' \cmp f_1 \cmp \al$ for some $g' \in \fS$.
		By (wR) there exists $\map e b v$ such that $\phi$ is constant on $e \cmp F$.
		Note that $\phi(e \cmp (g_1 \cmp f_1) \cmp \al) = 1$, by associativity.
		Thus also $\phi(e \cmp (g_0 \cmp f_0) \cmp \al) = 1$, which means that there exists $h$ such that
			$$e \cmp g_0 \cmp f_0 \cmp \al = h \cmp f_1 \cmp \al.$$
		We are done, because $e\cmp g_0$ and $h$ witness the weak amalgamation.
	\end{pf}
\end{lm}

Recall that if one of arrows $\al, \beta$ is amalgamable, then so is $\beta \cmp \al$, see \cite[Lemma~2.5]{KweakFra}.
The same composition behavior is true also for Ramsey arrows.

\begin{lm} \label{LMcomposition}
	Let $\al\maps a \to a'$ and $\beta\maps a' \to a''$ be arrows in a category $\fS$.
	If $\al$ or $\beta$ is Ramsey, then $\beta \cmp \al$ is Ramsey.
	
	\begin{proof}
		Suppose $\al$ is Ramsey, and let $b \in \ob{\fS}$, $k \in \nat$, and $F \subs \fS(\beta \cmp \al, b)$ finite.
		Since $\fS(\beta \cmp \al, b) \subs \fS(\al, b)$, we may take $v \in \ob{\fS}$ for $\al, b, k, F$.
		Let $\phi\maps \fS(\beta \cmp \al, v) \to k$ be a coloring.
		We extend it to $\phi'\maps \fS(\al, v) \to k$.
		By the choice of $v$ there is $e\maps b \to v$ such that $\phi'$ is constant on $e \cmp F$.
		But $e \cmp F \subs \fS(\beta \cmp \al, v)$, so we are done.
		
		Next suppose that $\beta$ is Ramsey.
		Again, let $b \in \ob{\fS}$, $k \in \nat$, and $F \subs \fS(\beta \cmp \al, b)$ finite.
		There is finite $F' \subs \fS(\beta, b)$ such that $F' \cmp \al = F$.
		Take $v \in \ob{\fS}$ for $\beta, b, k, F'$.
		Let $\phi\maps \fS(\beta \cmp \al, v) \to k$ be a coloring, and define $\phi'\maps \fS(\beta, v) \to k$ by $\phi'(\xi) := \phi(\xi \cmp \al)$.
		By the choice of $v$ there is $e\maps b \to v$ such that $\phi'$ is constant on $e \cmp F'$.
		Hence, $\phi$ is constant on $e \cmp F' \cmp \al = e \cmp F$.
	\end{proof}
\end{lm}

\begin{lm} \label{LMhereditary}
	Let $\al\maps a \to a'$ and $\beta\maps a' \to a''$ be arrows in a category $\fS$.
	If $\al$ is amalgamable and $\beta \cmp \al$ is Ramsey, then $\al$ is Ramsey.
	
	\begin{proof}
		Let $b \in \ob{\fS}$, $k \in \nat$, and $F \subs \fS(\al, b)$ finite.
		First we prove that there is $b' \in \ob{\fS}$ and $g\maps b \to b'$ such that $F' := g \cmp F \subs \fS(\beta \cmp \al, b')$.
		Enumerate $F$ as $\set{f_i: i < n}$.
		Since $\al$ is amalgamable, there are $\fS$-arrows $g_0$ and $h_0$ such that $g_0 \cmp f_0 = h_0 \cmp \beta \cmp \al$.
		Then there are $g_1$ and $h_1$ such that $g_1 \cmp (g_0 \cmp f_1) = h_1 \cmp (h_0 \cmp \beta \cmp \al)$.
		Note that there is no reason why $g_1 \cmp g_0 \cmp f_0$ and $g_1 \cmp g_0 \cmp f_1$ should be equal, but they are certainly both factorizing through $\beta \cmp \al$.
		We continue the same way, and finally put $g := g_{n - 1} \cmp \cdots \cmp g_0$ and $b' := \cod(g)$.
		
		Since $\beta \cmp \al$ is Ramsey, there is $v \in \ob{\fS}$ for $b', k, F'$.
		Let $\phi\maps \fS(\al, v) \to k$ be a coloring.
		Since $\fS(\beta \cmp \al, v) \subs \fS(\al, v)$, there is $e'\maps b' \to v$ such that $\phi$ is constant on $e' \cmp F' = e' \cmp g \cmp F =: e \cmp F$.
	\end{proof}
\end{lm}

\begin{wn}
	Let $\fS$ be a directed category with the weak Ramsey property.
	An $\fS$-arrow $\al\maps a \to a'$ is Ramsey if and only if it is amalgamable.
	Hence, $\fS$ has the weak amalgamation property, and it has the Ramsey property if and only if it has the amalgamation property.
	
	\begin{proof}
		Every Ramsey arrow is amalgamable by Lemma~\ref{LMerigoeg}.
		Let $\al\maps a \to a'$ be amalgamable.
		By the weak Ramsey property there is a Ramsey arrows $\beta\maps a' \to a''$.
		By Lemmma~\ref{LMcomposition}, $\beta \cmp \al$ is Ramsey, and by Lemma~\ref{LMhereditary}, $\al$ is Ramsey.
	\end{proof}
\end{wn}

\begin{prop}\phantomsection 
	\begin{enumerate}
		\item Let $F\maps \fS \to \fS'$ be a full cofinal functor (where cofinal means that for every $\fS'$-object $x$ there is a $\fS$-object $y$ and an $\fS'$-arrow $f\maps x \to F(y)$).
			If $\al\maps a \to a'$ is a Ramsey arrow in $\fS$, then $F(\al)$ is a Ramsey arrow in $\fS'$.
		\item Let $\fS \subs \fS'$ be a full cofinal subcategory.
			An $\fS$-arrow $\al\maps a \to a'$ is Ramsey in $\fS$ if and only if it is Ramsey in $\fS'$.
	\end{enumerate}
	
	\begin{proof}
		(i) Let $b' \in \ob{\fS'}$, $H' \subs \fS'(F(\al), b')$ finite, and $k \in \omega$.
		By the cofinality there is $b \in \ob{\fS}$ and $f \in \fS'(b', F(b))$, and by the fullness there is finite $H \subs \fS(\al, b)$ such that $F[H] = f \cmp H'$.
		There is also $v \in \ob{\fS}$ witnessing that $\al$ is Ramsey for $b, H, k$.
		For every coloring $\phi'\maps \fS'(F(\al), F(v)) \to k$ there is the coloring $\phi := (\phi' \cmp F)\maps \fS(\al, v) \to k$, and $g \in \fS(b, v)$ such that $\phi$ is constant on $g \cmp H$, and so $\phi'$ is constant on $F[g \cmp H] = (F(g) \cmp f) \cmp H'$.
		
		The forward implication of (ii) follows from (i).
		For the backward implication let $b \in \ob{\fS}$, $H \subs \fS(\al, b)$ finite, and $k \in \omega$.
		Let $v' \in \ob{\fS'}$ be the corresponding witnessing object, and let $f \in \fS'(v', v)$ for some $v \in \ob{\fS}$.
		Every coloring $\phi\maps \fS(\al, v) \to k$ induces the coloring $\phi'\maps \fS'(\al, v') \to k$ defined by $\phi'(g) := \phi(f \cmp g)$.
		There is $g \in \fS(b, v')$ such that $\phi'$ is constant on $g \cmp H$, and so $\phi$ is constant on $(f \cmp g) \cmp H$.
	\end{proof}
\end{prop}

\begin{wn} \label{equiRamsey}
	Let $\fS \subs \fS'$ be a full cofinal subcategory.
	$\fS$ has the weak Ramsey property if and only if $\fS'$ has the weak Ramsey property.
\end{wn}

\begin{df}
	Let $\fS \subs \fL$ be categories and let $U$ be a fixed $\fL$-object.
	We say that an $\fS$-arrow $\al\maps a \to a'$ satisfies
	\begin{enumerate}
		\item[(wA)] if for every $k \in \nat$, for every finite $F \subs \fL(\al, U)$, for every $\phi\maps \fL(\al, U) \to k$ there is $e \in \aut(U)$ such that $\phi$ is constant on $e \cmp F$,
		\item[(wB)] if the map $e$ in (wA) is required to be only an endomorphism instead of an automorphism of $U$, i.e. if for every $k \in \nat$, for every finite $F \subs \fL(\al, U)$, for every $\phi\maps \fL(\al, U) \to k$ there is $e \in \fL(U, U)$ such that $\phi$ is constant on $e \cmp F$.
	\end{enumerate}
	We say that $U$ has the \emph{weak finitary big Ramsey property} in $(\fS, \fL)$ if for every $a \in \ob{\fS}$ there exists an $\fS$-arrow $\al\maps a \to a'$ satisfying (wB).
	Similarly, we say that $U$ has the \emph{finitary big Ramsey property} if every $\id{a}$, $a \in \ob{\fS}$, satisfies (wB).
	
	Note that for (wA) and (wB), given $\al, k, \phi$ there is a constant value $i \in k$ that works for every finite $F \subs \fL(\al, U)$.
	Otherwise, there would be a counterexample set $F_i$ for every $i \in k$, and so $F := \bigcup_{i \in k} F_i$ would be a counterexample for (wA) or (wB).
\end{df}

\begin{remark}
	The name ``(weak) finitary big Ramsey property'' was chosen to stress the formal similarily to the standard \emph{big Ramsey property}, i.e. \emph{big Ramsey degree} \cite[p.~176]{KPT} equal to one.
	In that (rare) case, $e$ does not depend on $F$, and $\phi$ is constant on the whole set $e \cmp \fL(\al, U)$.
\end{remark}

\begin{tw} \label{THRamsey}
	Let $(\vec{u}, \vecinfty{u}, U)$ be a matching sequence in $(\fS, \fL)$ for some categories $\fS \subs \fL$, and let $\al\maps a \to a'$ be an $\fS$-arrow.
	If $\vec{u}$ is a weak Fraïssé sequence, then the following conditions are equivalent:
	\begin{enumerate}
		\item $\al$ is a Ramsey arrow, i.e. it satisfies (wR),
		\item $\al$ satisfies (wA),
		\item $\al$ satisfies (wB).
	\end{enumerate}
	It follows that $\fS$ has the (weak) Ramsey property if and only if $U$ has the (weak) finitary big Ramsey property in $(\fS, \fL)$.
	
	\begin{proof}
		By (F1), for every $f \in \fL(\al, U)$ there exists $n_f \in \omega$ and $f' \in \fS(\al, u_n)$ such that $u^\infty_n \cmp f' = f$.
		For every $n \geq n_f$ let us put $\Delta^n(f) := u^n_{n_f} \cmp f'$, so we have $f = u^\infty_n \cmp \Delta^n(f)$ and $\Delta^{n'}(f) = u^{n'}_n \cmp \Delta^n(f)$ for every $n' \geq n \geq n_f$.
		Note that by (F2) we have also $\Delta^n(u^\infty_m \cmp f) = u^n_m \cmp f$ for every compatible $\fS$-arrow $f$ and every sufficiently large $n$.
		
		The implication (wA)$\implies$(wB) is trivial.
		
		Next we prove (wB)$\implies$(wR) by contradiction.
		Suppose $\al$ fails (wR) and it is witnessed by $k \in \nat$, $b \in \ob{\fS}$ and a finite $F \subs \fS(\al, b)$.
		Specifically, for every $n \in \omega$ there is $\map{\phi_n}{\fS(\al, u_n)}{k}$ such that $\img{\phi_n}{e \cmp F}$ is not a singleton, whenever $e \in \fS(b, u_n)$.
		Fix a non-principal ultrafilter $p$ on $\nat$ and define
		$\map{\phi}{\fL(\al, U)}{k}$ by setting
		$$\phi(f) = i \iff \setof{\ntr}{\phi_n(\Delta^n(f)) = i} \in p.$$
		Note that $\Delta^n(f)$ is defined for all but finitely many numbers $n$.
		Since $\vec{u}$ is a weak Fraïssé sequence, there is an $\fS$-map $e_0\maps b \to u_{n_0}$ for some $n_0 \in \omega$.
		We consider $F' := u^\infty_{n_0} \cmp e_0 \cmp F \subs \fL(\al, U)$.
		Since $\al$ satisfies (wB), there is $j \in k$ and $e' \in \fL(U, U)$ such that $\img{\phi}{e' \cmp F'} = \sn j$.
		By (F1) we find $m \in \nat$ and $e \in \fS(b, u_m)$ such that $e' \cmp u^\infty_{n_0} \cmp e_0 = u_m^\infty \cmp e$, as in the diagram below.
		$$
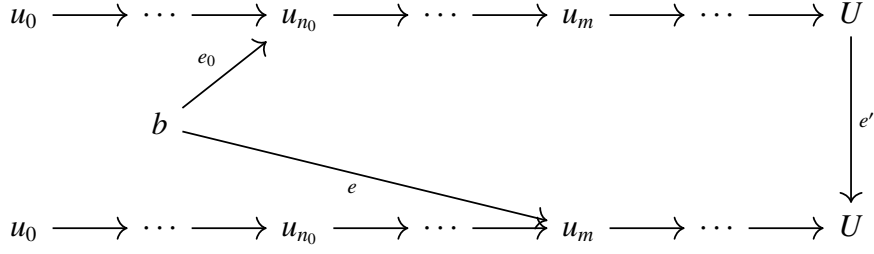
\begin{tikzcd}
			u_0 \ar[r] & \cdots \ar[r] & u_{n_0} \ar[r] & \cdots \ar[r] & u_m \ar[r] & \cdots \ar[r] & U \ar[dd, "e'"] \\
			& b \ar[ur, "e_0"] \ar[drrr, "e"'] \\
			u_0 \ar[r] & \cdots \ar[r] & u_{n_0} \ar[r] & \cdots \ar[r] & u_m \ar[r] & \cdots \ar[r] & U
		\end{tikzcd}$$
		Given $f \in F$, we have $\Delta^n(u_m^\infty \cmp e \cmp f) = u_m^n \cmp e \cmp f$ for every sufficiently large $n$, and hence
		\begin{align*}
			j = \phi(e' \cmp u_{n_0}^\infty \cmp e_0 \cmp f) &= \phi(u_m^\infty \cmp e \cmp f) = \lim_{n \to p} \phi_n(\Delta^n(u_m^\infty \cmp e  \cmp f)) = \lim_{n \to p} \phi_n(u_m^n \cmp e \cmp f).
		\end{align*}
		The last limit along $p$ means that the set $A_f := \setof{n \goe m}{\phi_n(u_m^n \cmp e \cmp f) = j} \in p$.
		Since $F$ is finite, we may find $\ell > m$ such that $\phi_\ell(u_m^\ell \cmp e \cmp f) = j$ for every $f \in F$.
		It is enough to choose $\ell \in \bigcap_{f \in F} A_f$.
		Together, $\phi_\ell$ restricted to $(u_m^\ell \cmp e) \cmp F$ is constant, which is a contradiction.

		To prove (wR)$\implies$(wA) let $\al\maps a \to a'$ be a Ramsey $\fS$-arrow.
		Fix $k \in \nat$, fix finite $F \subs \fL(\al, U)$, and fix $\phi\maps \fL(\al, U) \to k$.
		Our goal is to find $e \in \aut(U)$ such that $\phi$ is constant on $e \cmp F$.
		
		Since $F$ is finite, there is $m \in \omega$ such that $f = u^\infty_m \cmp \Delta^m(f)$ for every $f \in F$.
		Since $\vec{u}$ is a weak Fraïssé sequence, there is $m' \geq m$ such that $\beta := u_m^{m'}$ works in (W1), or equivalently is amalgamable.
		Let $b := u_{m'}$ and $F' := \setof{\beta \cmp \Delta^m(f)}{f \in F} \subs \fS(\al, b)$.
		Using (wR) with $F'$ and $b$ we obtain $v \in \ob{\fS}$ such that for every $\map{\psi}{\fS(\al, v)}{k}$ there is $e'\maps b \to v$ with $\psi$ constant on the set $e' \cmp F'$.
		Note that there exists at least one $\psi$ as above (unless $k = 0$, in which case (wA) is trivially true).
		Consequently, there exists an $\fS$-arrow $\map{\gamma}{b}{v}$.
		Recalling that $\beta = u^{m'}_m$ is amalgamable, we find an $\fL$-arrow $\map \delta v U$ such that
			$$u_m^\infty = \delta \cmp \gamma \cmp \beta.$$
		The following diagram should clarify the situation.
		$$
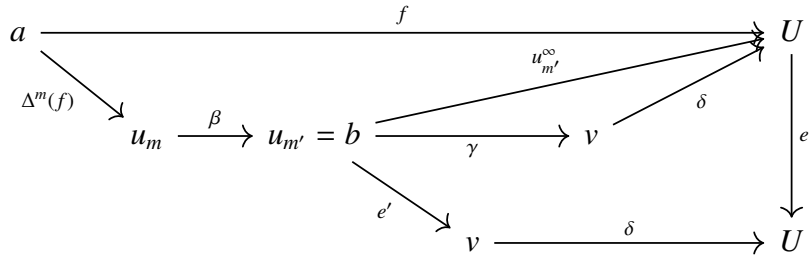
\begin{tikzcd}
			a \ar[rrrrr, "f"] \ar[dr, "{\Delta^m(f)}"'] & & & & & U \ar[dd, "e"] \\
			& u_m \ar[r,"\beta"'] \ar[urrrr,"u_m^\infty"] & u_{m'} = b \ar[dr, "e'"'] \ar[rr, "\gamma"'] & & v \ar[ru, "\delta"'] & \\
			& & & v \ar[rr, "\delta"] & & U
		\end{tikzcd}$$
		Define $\map{\tilde{\phi}}{\fS(\al,v)}{k}$ by
			$$\tilde{\phi}(\xi) = \phi(\delta \cmp \xi) \qquad \text{ for every } \xi \in \fS(\al,v).$$
		The weak Ramsey property gives $e'\maps b \to v$ such that
			$$\tilde{\phi}(e' \cmp \beta \cmp \Delta^m(f)) = j$$
		for every $f \in F$, where $j \in k$ is fixed.
		Now we use the weak homogeneity of $U$ (Theorem~\ref{thm:Flim}), knowing that $\beta$ is amalgamable.
		Namely, there exists $e \in \aut(U)$ such that
			$$e \cmp \delta \cmp \gamma \cmp \beta = \delta \cmp e' \cmp \beta.$$
		Finally, given $f \in F$, we have
		\begin{align*}
			\phi(e \cmp f) &= \phi(e \cmp u_m^\infty \cmp \Delta^m(f)) 
				= \phi(e \cmp \delta \cmp \gamma \cmp \beta \cmp \Delta^m(f)) \\
				&= \phi(\delta \cmp e' \cmp \beta \cmp \Delta^m(f))
				= \tilde{\phi}(e' \cmp \beta \cmp \Delta^m(f)) = j.
		\end{align*}
		This completes the proof.	
	\end{proof}
\end{tw}

\subsection{Extreme amenability}

Recall that an action $G \acts X$ of a group $G$ on a set $X$ is a group homomorphism $\eta\maps G \to \aut(X)$, where $\aut(X)$ is the the group of all bijections of $X$.
We shall write $g x$ or $g \cdot x$ instead of $\eta(g)(x)$.
This way an action can be equivalently viewed as a map $G \times X \to X$ such that $1 x = x$ and $g (h x) = (g h) x$ for $g, h \in G$ and $x \in X$.
Given $G \acts X$, the \emph{orbit} of $x_0 \in X$ is $G x_0 := \setof{g x_0}{g \in G}$.
The action is \emph{transitive} if $G x_0 = X$ for some $x_0$ (equivalently: for every $x_0$).

A morphism $\pi\maps \eta \to \eta'$ of actions $\eta\maps G \acts X$ and $\eta'\maps G \acts Y$ is a mapping $\pi\maps X \to Y$ such that $\pi(gx) = g\pi(x)$ for every $g \in G$ and $x \in X$, or equivalently $\pi \cmp \eta(g) = \eta'(g) \cmp \pi$ for every $g \in G$.

An action $G \acts X$ of a topological group $G$ on a topological space $X$ is \emph{continuous} if it is continuous when viewed as a mapping $G\times X \to X$ with respect to the product topology on $G \times X$.
Recall that a topological group $G$ is called \emph{extremely amenable} if every continuous action $G \acts X$ on a compact space $X$ has a \emph{fixed point}, i.e. there is a point $x_0 \in X$ such that $g x_0 = x_0$ for every $g \in G$.

\begin{df}
	An action $G \acts X$ of a group $G$ on a set $X$ is \emph{finitely oscillation stable} if 
	\begin{enumerate}
		\item[(FS)] for every $k \in \nat$, for every $\phi\maps X \to k$, and for every finite set $F \subs X$ there exists $g \in G$ such that $\phi$ is constant on $g F$.
	\end{enumerate}
	This is an equivalent formulation of the standard finite oscillation stability~\cite[1.1]{Pestov} of a discrete space $X$, see \cite[Theorem~1.1.18 (7)]{Pestov}.
	Note that in the situation of Definition~\ref{DFweakRmsyPy} an $\fS$-arrow $\al$ satisfies (wA) if and only if the action $\aut(U) \acts \fL(\al, U)$ satisfies (FS).
\end{df}

\begin{prop}\label{PROPwefheof}
	Let $G$ be a topological group with a neighborhood base $\Vee$ of its unit, consisting of open subgroups.
	The following properties are equivalent.
	\begin{enumerate}[label=\rm(\alph*)]
		\item $G$ is extremely amenable.
		\item For every $V \in \Vee$, the action $G \acts G \by V$ on left cosets ($g \cdot hV = ghV$) satisfies (FS).
	\end{enumerate}
\end{prop}

The proof can be found essentially in~\cite[Prop. 4.2]{KPT}, where it is assumed that $G$ is a closed subgroup of $S_\infty$, however the proof uses exclusively the existence of a neighborhood base $\Vee$ as above (see also the remarks after~\cite[Prop. 4.2]{KPT}).
Recall that a topological group $G$ embeds into $S_\infty$ as a closed subgroup if and only if it is a non-archimedean Polish group, i.e. if it has a countable neighborhood base $\Vee$ of the unit consisting of open subgroups and is separable.


\begin{remark}
	Let $(\vec{u}, \vecinfty{u}, U)$ be a matching sequence in $(\fS, \fL)$.
	Recall that a basic neighborhood of the identity $\id U \in G := \aut(U)$ is of the form
	\[
		V_m = \setof{g \in G}{g \cmp u_m^\infty = u_m^\infty},
	\]
	where $m \in \nat$.
	This is obviously a subgroup of $G$.
	In fact, it is the stabilizer of $u^\infty_m$ of the action $G \acts \fL(u_m, U)$.
	Hence, the map $\pi\maps G \by V_m \to G \cmp u^\infty_m$ defined by $h \cmp V_m \mapsto h \cmp u^\infty_m$ is an isomorphism of the action $G \acts G \by V_m$ on left cosets and the action $G \acts G \cmp u^\infty_m$ of the automorphism group on the orbit $G \cmp u^\infty_m \subs \fL(u_m, U)$.
	Moreover, if $m' \geq m$ is such that $u^{m'}_m$ is amalgamable, then by the weak homogeneity the orbit $G \cmp u^\infty_m$ is the whole $\fL(u^{m'}_m, U)$.
\end{remark}

\begin{tw} \label{thm:KPT}
	Let $(\vec{u}, \vecinfty{u}, U)$ be a matching sequence in $(\fS, \fL)$ for some categories $\fS \subs \fL$.
	If $\vec{u}$ is a weak Fraïssé sequence, then the following conditions are equivalent.
	\begin{enumerate}
		\item $\aut(U)$ is extremely amenable.
		\item $U$ has the weak finitary big Ramsey property in $(\fS, \fL)$.
		\item $\fS$ has the weak Ramsey property.
	\end{enumerate}
	
	\begin{proof}
		Put $G := \aut(U)$.
		Let $m \in \nat$ and $m' \geq m$ such that $u^{m'}_m$ is amalgamable.
		By the previous remark, we have an isomorphism of the actions $G \acts G \by V_m$ and $G \acts \fL(u^{m'}_m, U)$.
		Therefore, $G \acts G \by V_m$ satisfies (FS) if and only if $u^{m'}_m$ satisfies (wA), or by Theorem~\ref{THRamsey} equivalently (wR).
		
		Suppose $G$ is extremely amenable.
		For every $a \in \ob{\fS}$ there is an $\fS$-arrow $f\maps a \to u_m$ for some $m \in \nat$, and there is $m' \geq m$ such that $u^{m'}_m$ is amalgamable.
		By Proposition~\ref{PROPwefheof} and the claim above, $u^{m'}_m$ is a Ramsey arrow, and so $\al := u^{m'}_m \cmp f$ is a Ramsey arrow as well by Lemma~\ref{LMcomposition}.
		Hence $\fS$ has the weak Ramsey property.
		
		Suppose $\fS$ has the weak Ramsey property.
		For every $m \in \nat$ there is $m' \geq m$ such that $u^{m'}_m$ amalgamable.
		By Lemma~\ref{LMhereditary}, $u^{m'}_m$ is a Ramsey arrow, and so by the claim above, $G \acts G \by V_m$ satisfies (FS).
		It follows from Proposition~\ref{PROPwefheof} that $G$ is extremely amenable.
		
		We have (ii) $\iff$ (iii) already by Theorem~\ref{THRamsey}.
	\end{proof}
\end{tw}

Recalling that the topological group $\aut(U)$ does not depend on the choice of $\fL$ and of weak Fraïssé matching sequence $(\vec{u}, \vecinfty{u}, U)$ in $(\fS, \fL)$ (see Construction~\ref{con:GS}), we obtain the following.
\begin{wn}
	A weak Fraïssé category $\fS$ has the weak Ramsey property if and only if the topological group $G(\fS)$ is extremely amenable.
\end{wn}

In the classical case when $\fS$ is a family of finite first-order structures with all embeddings, our category is locally finite, which allows us to simplify the weak Ramsey property.
In this case, when expanded, we obtain the following corollary.
Also recall that the topology on $\aut(U)$ does not depend on the choice of a weak Fraïssé sequence.

\begin{wn}
	Let $L$ be a first-order language, let $\fL$ be the category of all $L$-structures and all homomorphisms, and let $\fS \subs \fL$ be a subcategory of some finite $L$-structures and all embeddings between them.
	If $\fS$ is a weak Fraïssé category with a generic object $U$, then the following properties are equivalent.
	\begin{enumerate}[label=\rm(\alph*)]
		\item $\aut(U)$ is extremely amenable.
		\item For every $a \in \ob{\fS}$ there is an $\fS$-arrow $\map{\al}{a}{a'}$ such that for every $b \in \ob{\fS}$, for every $k \in \nat$ there exists $v \in \ob{\fS}$ such that for every coloring $\map{\phi}{\fS(\al,v)}{k}$ there exists $\map e b v$ with $\phi$ constant on $e \cmp \fS(\al,b)$. 
	\end{enumerate}
\end{wn}

\subsection{Amalgamation extension and arrow extension} \label{sec:extensions}

We briefly discuss the phenomenon of weak versions of certain notions like the amalgamation property and the Ramsey property from theoretical perspective.
In both situations the core property is localized at individual objects of a category and then generalized from objects to arrows.
Here we note that amalgamable/Ramsey arrows may be viewed as arrows factorizing through amalgamable/Ramsey objects in a certain extension category.

First, observe that for a full cofinal subcategory $\fC \subs \fC'$ we have that a $\fC$-arrow is amalgamable in $\fC$ if and only if it is amalgamable in $\fC'$.
In this section we shall work with full subcategories, and they will be sometimes identified with their classes of objects, e.g. $\fC' \setminus \fC$ denotes the full subcategory of $\fC'$ consisting of objects in $\ob{\fC'} \setminus \ob{\fC'}$.

Let $\fC$ be a category.
By $\am(\fC)$ we denote the full subcategory of $\fC$ consisting of all amalgamable objects.
Suppose that both $\am(\fC)$ and $\fC \setminus \am(\fC)$ are cofinal.
Then $\am(\fC)$ has the amalgamation property, $\fC$ has the cofinal amalgamation property~\cite{KweakFra} but not the amalgamation property, and $\fC \setminus \am(\fC)$ has the weak amalgamation property but not the cofinal amalgamation property.
In fact, $\fC \setminus \am(\fC)$ has no amalgamable objects.
So sometimes we may obtain a category with the weak amalgamation property without any amalgamable objects simply by removing the amalgamable objects.
Sometimes it even happens that every amalgamable arrow in $\fC \setminus \am(\fC)$ factorizes through an object in $\am(\fC)$.
Let us capture this situation by a definition.

\begin{df}
	By an \emph{amalgamation extension} of a category $\fC$ we mean a category $\fC' \sups \fC$ such that $\fC$ is a full cofinal subcategory of $\fC'$, every object of $\fC' \setminus \fC$ is amalgamable in $\fC'$, and every amalgamable $\fC$-arrow factorizes through an amalgamable object in $\fC'$.
	It follows that a $\fC$-arrow is amalgamable if and only if it factorizes through an amalgamable object in $\fC'$.
\end{df}

\begin{prop} \label{thm:amalgamation_extension}
	Let $\fS \subs \fS'$ be an amalgamation extension.
	\begin{enumerate}
		\item $\fS$ has the weak amalgamation property if and only if $\am(\fS')$ is cofinal in $\fS'$.
		\item Under the conditions in (i), $\fS$ is a weak Fraïssé category if and only if $\am(\fS')$ is a Fraïssé category.
			Moreover, an $\fS$-sequence $\vec{u}$ is weak Fraïssé in $\fS$ if and only if it is weak Fraïssé in $\fS'$, and an $\am(\fS')$-sequence $\vec{v}$ is Fraïssé in $\am(\fS')$ if and only if it is weak Fraïssé in $\fS'$.
			In this case, the sequences $\vec{u}$ and $\vec{v}$ are isomorphic, and so the topological groups $G(\fS)$ and $G(\am(\fS'))$ are isomorphic as well.
		\item Under the conditions in (i), $\fS$ has the weak Ramsey property if and only if $\am(\fS')$ has the Ramsey property.
	\end{enumerate}
	
	\begin{proof}
		Claim~(i) is clear from the fact that amalgamable arrows in $\fS$ are exactly arrows factorizing through an $\am(\fS')$-object.
		Claim~(ii) follows from the fact both $\fS$ and $\am(\fS')$ are full and cofinal in $\fS'$, and so directedness, weak domination by a countable subcategory, and the property of being a weak Fraïssé sequence is translated between $\fS'$ and the subcategories.
		For the rest see Construction~\ref{con:GS}.
		Claim~(iii) follows from two applications of Corollary~\ref{equiRamsey}.
	\end{proof}
\end{prop}

\begin{ex}
	Let $\fC$ be the category of all finite acyclic graphs and all embeddings.
	Then amalgamable objects are exactly connected graphs in $\fC$, i.e. finite trees.
	Moreover, $\fC$ is an amalgamation extension of $\fC \setminus \am(\fC)$.
	In other words, an embedding $e\maps G \to G'$ in the category of disconnected finite acyclic graphs $\fC \setminus \am(\fC)$ is amalgamable if and only if $e[G]$ lies in a single component of $G'$.
	This is because the only obstruction to amalgamation in $\fC$ is when two components of a given graph are connected by incompatible paths (e.g. of different lengths) in different extensions – the amalgamation would contain a cycle, which is forbidden.
\end{ex}

As seen in the example, amalgamation extensions may arise naturally.
On the other hand, every category admits at least the following “artificial” amalgamation extension.

\begin{df}
	For every category $\fC$ we define its \emph{arrow extension} $\arex{\fC}$ as follows.
	A $\arex{\fC}$-object is a $\fC$-arrow $\al\maps a \to a'$, sometimes written as a pair $(a, \al)$.
	A $\arex{\fC}$-arrow $f\maps (a, \al) \to (b, \beta)$ is a $\fC$-arrow $f\maps a \to b$ that factorizes through $\al$ (i.e. such that there is a $\fC$-arrow $f'$ with $f' \cmp \al = f$) or $\id{a}$ if $\al = \beta$, so we have identities in $\arex{\fC}$.
	The composition in $\arex{\fC}$ defined by the composition in $\fC$ is correct.
	In fact, for every $\arex{\fC}$-arrow $f\maps (a, \al) \to (b, \beta)$, every $\fC$-arrow $g\maps b \to c$, and every $\arex{\fC}$-object $(c, \gamma)$ we have that $g \cmp f$ is a $\arex{\fC}$-arrow $(a, \al) \to (c, \gamma)$.
	
	Note that the natural functor $F\maps \fC \to \arex{\fC}$ mapping $a \mapsto (a, \id{a})$ is fully faithful, and so $\fC$ may be identified with the full subcategory of $\arex{\fC}$ consisting of identities.
	On the other hand, we also have the faithful functor $U\maps \arex{\fC} \to \fC$ mapping $(a, \al) \mapsto a$.
	Moreover, since $U \cmp F = \id{\fC}$, $\fC$ is a retract of $\arex{\fC}$.
	Also note that $\al\maps (a, \al) \to (a', \id{a'})$ for every $\arex{\fC}$-object $\al\maps a \to a'$, so $\fC \subs \arex{\fC}$ is cofinal.
	Finally note that the notation $\fC(\al, b)$ used as a shortcut for $\fC(a', b) \cmp \al$ in the previous sections really corresponds to the actual homset $\arex{\fC}(\al, \beta)$ where $\dom(\beta) = b $.
\end{df}

\begin{prop}
	A $\fC$-arrow $\al\maps a \to a'$ is amalgamable if and only if it is an amalgamable object in $\arex{\fC}$.
	Hence, $\fC \cup \am(\arex{\fC})$ is an amalgamation extension of $\fC$.
	
	\begin{proof}
		For the first part it is enough to translate amalgamation spans (i.e. diagrams of the form $b \leftarrow a \to c$) and their solutions between $\fC$ and $\arex{\fC}$, as shown in Figure~\ref{fig:spanss} for the first implication.
		If $\al$ is an amalgamable arrow in $\fC$, $\beta\maps b \to b'$ and $\gamma\maps c \to c'$ are $\arex{\fC}$ objects, and $f\maps \al \to \beta$ and $g\maps \al \to \gamma$ are $\arex{\fC}$-arrows, then there are $\fC$-arrows $\tilde{f}\maps a' \to b$ and $\tilde{g}\maps a' \to c$ such that $f = \tilde{f} \cmp \al$ and $g = \tilde{g} \cmp \al$ (or one of $f, g$ is an identity arrow, in which case the amalgamation is trivial),
		and there are $\fC$-arrows $f'$ and $g'$ to a $\fC$-object $d$ such that $f' \cmp \beta \cmp \tilde{f} \cmp \al = g' \cmp \gamma \cmp \tilde{g} \cmp \al$.

\begin{figure}[ht]
	\centering
	
	\begin{tikzcd}
		& (b, \beta) \arrow[dr, "f' \cmp \beta", dashed] \\
		(a, \al)  \arrow[ur, "f"] \arrow[dr, "g"'] & & d \\
		& (c, \gamma) \arrow[ur, "g' \cmp \gamma"', dashed]
	\end{tikzcd}
	\qquad
	\begin{tikzcd}
		& & b \arrow[r, "\beta"] & b' \arrow[dr, "f'", dashed] \\
		a \arrow[r, "\al"] \arrow[urr, "f"] \arrow[drr, "g"']
			& a' \arrow[ur, "\tilde{f}"', dashed] \arrow[dr, "\tilde{g}", dashed]
			& & & d \\
		& & c \arrow[r, "\gamma"] & c' \arrow[ur, "g'"', dashed]
	\end{tikzcd}
	
	\caption{A span in $\arex{\fC}$ and the corresponding span in $\fC$.}
	\label{fig:spanss}
\end{figure}

		If $\al$ is an amalgamable object in $\arex{\fC}$ and $f\maps a' \to b$ and $g\maps a' \to c$ are $\fC$-arrows, then $f \cmp \al\maps \al \to b$ and $g \cmp \al\maps \al \to c$, and so there is a $\arex{\fC}$-object $(d, \delta)$ and $\arex{\fC}$-arrows $f'\maps b \to \delta$ and $g'\maps c \to \delta$ such that $f' \cmp f \cmp \al = g' \cmp g \cmp \al$.
		Since $f'$ and $g'$ may be also viewed as $\fC$-arrows $b \to d$ and $c \to d$, we are done.
		
		It follows that $\fC \cup \am(\arex{\fC})$ is and amalgamation extension of $\fC$.
		$\fC$ is full and cofinal in $\fC \cup \am(\arex{\fC})$ since it is so in $\arex{\fC}$.
		Every object in $\am(\arex{\fC})$ is amalgamable in $\fC \cup \am(\arex{\fC})$ since $\fC \cup \am(\arex{\fC})$ is full and cofinal in $\arex{\fC}$.
		Finally, every amalgamable arrow $\al\maps a \to a'$ in $\fC$ factorizes as $g \cmp f$ where $f\maps a \to \al$ corresponds to $\id{a}$ and $g\maps \al \to a'$ corresponds to $\al$.
	\end{proof}
\end{prop}

\begin{prop}
	A $\fC$-arrow $\al\maps a \to a'$ is Ramsey if and only if it is a Ramsey object in $\arex{\fC}$, i.e. if $\id{\al}$ is a Ramsey arrow in $\arex{\fC}$.
	
	\begin{proof}
		We have already observed that $\fC(\al, b) = \arex{\fC}(\al, \beta)$ for every $\fC$-arrow $\beta$ with domain $b$.
		Now the difference between the situation in $\fC$ and $\arex{\fC}$ is that in $\arex{\fC}$ more objects are allowed for $b$ as well as for the Ramsey witnessing object $v$.
		But that may be overcome by the fact that $\fC$ is cofinal in $\fC'$.
		
		Suppose $\al$ is a Ramsey arrow in $\fC$.
		Let $\beta\maps b \to b'$ be a $\arex{\fC}$-object, let $F \subs \arex{\fC}(\al, \beta)$ be finite, and let $k \in \omega$.
		There is a $\arex{\fC}$-arrow $f\maps \beta \to b'$ and a $\fC$-object $v$ such that for every coloring $\phi\maps \fC(\al, v) \to k$ there is a $\fC$-arrow $g\maps b' \to v$ such that $\phi$ is constant on $g \cmp f \cmp F$.
		
		On the other hand, suppose that $\al$ is a Ramsey object in $\arex{\fC}$.
		For every $\fC$-object $b$, finite $F \subs \fC(\al, b)$, and $k \in \omega$ there is a witnessing $\arex{\fC}$-object $\gamma\maps v \to v'$ and a $\arex{\fC}$-arrow $f\maps \gamma \to v'$.
		Every coloring $\phi\maps \fC(\al, v') \to k$ induces the coloring $\psi\maps \arex{\fC}(\al, \gamma) \to k$ by $\psi(g) := \phi(f \cmp g)$.
		Hence there is a $\arex{\fC}$-arrow $g\maps b \to \gamma$ such that $\psi$ is constant on $g \cmp F$, and so $\phi$ is constant on $f \cmp g \cmp F$, where $(f \cmp g)\maps b \to v'$.
	\end{proof}
\end{prop}

Together, we obtain the following.
\begin{wn}
	Let $\fC$ be a category with the weak Ramsey property and the weak amalgamation property.
	Then $\am(\arex{\fC})$ has Ramsey property and the amalgamation property, and both $\fC$ and $\am(\arex{\fC})$ are full cofinal subcategories of $\arex{\fC}$.
\end{wn}

\section{Applications}

We demonstrate the theory developed in the previous sections on several examples.
Two extreme kinds of categories -- in the sense of having degenerate hom-sets and degenerate class of objects, respectively -- are posets and monoids.
Recall that every poset $(P, \leq)$ may be regarded as a category $\fC$ with $\ob{\fC} = P$ and $\fC(x, y)$ being a singleton if $x \leq y$, and empty otherwise.
In general, a category $\fC$ such that every hom-set $\fC(a, b)$ is empty or a singleton is called \emph{thin}.
Clearly, every thin category has the Ramsey property.

\begin{df}
	We say that a category $\fC$ is \emph{weakly thin} if for every $\fC$-object $a$ there is a $\fC$-arrow $\al\maps a \to a'$ such that for every $\fC$-object $b$ we have $\card{\fC(a', b) \cmp \al} \leq 1$, i.e. there is at most one arrow $a \to b$ that factorizes through $\al$.
	
	Clearly, every weakly thin category has the weak Ramsey property.
\end{df}

\begin{ex}
	Let $\fC$ be the category of all finite graphs whose all cycles have pairwise disjoint sets of vertices and have different lengths, with all embeddings as morphisms.
	Then $\fC$ is a weakly thin (and so having the weak Ramsey property) hereditary class without the Ramsey property, though it is not weak Fraïssé since it is not directed and does not have the weak amalgamation property.
	
	\begin{proof}
		For every graph $G$ in $\fC$ we describe an inclusion $\al\maps G \to G'$ to a bigger graph $G'$ in $\fC$ such that every vertex of $G$ is definable (by an existential formula without parameters) in $G'$ as well as in every extension of $G'$ in $\fC$ (by the same formula across the extensions).
		Every cycle is definable (as a set) because of unique lengths.
		Since cycles are disjoint, there is at most one edge connecting two fixed cycles, and so the endpoints of that edge are definable.
		Finally, if at least two non-antipodal vertices on a cycle are definable, all vertices on the cycle are definable.
		Hence, to form $G'$ it is enough to add cycles and paths so that every vertex of $G$ is covered by one of the cases above.
		
		The category $\fC$ does not have the Ramsey property since its objects are not rigid – e.g. a cycle $C$ has non-trivial automorphisms and for every $G$ in $\fC$ we can color $\fC(C, G)$ so that every two different embeddings with the same image get different colors.
		
		The category $\fC$ is not directed and does not have the weak amalgamation property: let $C_n$ denote a cycle of length $n \geq 3$ and let $G_{a, b, c}$, $a \neq b \neq c \geq 3$, be a graph consisting of $C_a$ and $C_c$ both joined by an edge to the same vertex in $C_b$.
		Then $G_{a, b, c}$ and $G_{b, c, a}$ can never be jointly embedded into a graph in $\fC$.
		Since every $\fC$-object $H$ can be extended to $H \cup G_{a, b, c}$ and $H \cup G_{b, c, a}$ for suitably large $a, b, c$, $\fC$ does not have the weak amalgamation property.
	\end{proof}
\end{ex}

We shall look at monoids in the next section.

\subsection{Monoids as categories}

Recall that a \emph{monoid} is a triple $(M, \cdot, 1)$ where $\cdot$ is an associative operation on a set $M$, and $1 \in M$ is the unit: we have $x \cdot 1 = 1 \cdot x = x$ for every $x \in M$.
A monoid $M$ can be viewed as a category with a single object: the elements of $M$ become the endomorphisms, the multiplication $\cdot$ becomes the composition, and the unit $1$ becomes the identity on the single unnamed object.

Then, an element $\al \in M$ corresponds to a Ramsey arrow if and only if for every $k \in \nat$, for every finite $F \subs M\al$, and for every $\phi\maps M\al \to k$ there exists $e \in M$ such that $\phi\restr{eF}$ is constant.
The monoid $M$ has the weak Ramsey property if and only if there exists a Ramsey arrow, and $M$ has the Ramsey property if and only if the unit (and so every element) is a Ramsey arrow.

\begin{df}
	We say that an element $\al \in M$ of a monoid satisfies the \emph{left equalization condition} (LE) if for every finite $F \subs M\al$ there exists $e \in M$ such that $eF$ is a singleton.
	Since we may assume that $\al \in F$, this is equivalent to $eF = \set{e\al}$.
	Also, by induction, (LE) is equivalent to the property that for every $x, y \in M\al$ there is $e \in M$ such that $ex = ey$.
	This is because for $x, y, z \in M\al$ there is $e \in M$ such that $ex = ey$ and $e' \in M$ such that $e'(ey) = e'(ez)$.
	It follows that $e'e\set{x, y, z}$ is a singleton.
\end{df}

Note that if an element $\al \in M$ satisfies (LE), then it is a Ramsey arrow.
At several cases, the converse is true as well.

\begin{prop}
	Let $M$ be a monoid and $\al \in M$.
	Let $F$ be the submonoid $\set{f \in M: M\al f \subs M\al}$,
	and let $G$ be the graph of the right action of $F$ on $M\al$, i.e. the set of vertices is $M\al$, and we put an edge $x \mapsto_f xf$ for every $x \in M\al$ and $f \in F$.
	If $G$ has finitely many undirected components, then $\al$ is a Ramsey arrow if and only if it satisfies (LE).
	This includes the following cases: (i) $M$ is finite, (ii) $M$ is commutative, (iii) $\al$ is left-invertible, in particular if $\al$ is the unit or $M$ is a group.
	
	\begin{proof}
		Suppose that $\al$ is a Ramsey arrow.
		First we show that if $x, y$ lie in the same component of $G$, then there is $e \in M$ such that $ex = ey$.
		By induction, it is enough to consider the case $y = xf$ for some $f \in F$.
		Let $G_f$ be the subgraph of $G$ of the right action of $f$ on $M\al$, i.e. we consider only the edges $x \mapsto xf$ for every $x \in M\al$.
		For every $x \in M\al$ the path $(xf^n)_{n \in \omega}$ in $G_f$ is an infinite ray or finite cycle with a finite initial segment attached.
		We may inductively define a coloring $\phi\maps M\al \to 3$ such that $\phi(x) \neq \phi(xf)$ unless $x = xf$.
		We need the third color only because of possible cycles of odd length.
		Since $\al$ is a Ramsey arrow, there is $e \in M$ such that $\phi(ex) = \phi(exf)$, and so $ex = exf$.
		
		Now let $x, y \in M\al$ be arbitrary.
		Since $G$ has finitely many components, there is a coloring $\psi\maps M\al \to k$ for some $k \in \omega$, such that points from different components take different colors.
		Since $\al$ is a Ramsey arrow, there is $e \in M$ such that $\psi(ex) = \psi(ey)$, and so $ex$ and $ey$ lie in the same component.
		So by the previous claim, there is $e' \in M$ such that $e'ex = e'ey$.
		
		It remains to show that $G$ has finitely many components in the cases (i), (ii), (iii).
		This is clear if $M$ is finite.
		If $M$ is commutative, then $G$ has only one component: every $x \in M\alpha$ is of the form $f\al = \al f$ for some $f \in M = F$, and hence $x$ is in the component of $\al$.
		Similarly, if $M\al = M \ni 1$, then $F = M$ and so $M\al = 1 F$, and again $G$ has only one component.
	\end{proof}
\end{prop}

\begin{wn}
	A monoid $M$ has the Ramsey property if and only if for every $x, y \in M$ there is $e \in M$ such that $ex = ey$.
\end{wn}

\begin{ex}
	A monoid with a \emph{left zero}, i.e. an element $0$ such that $0 \cdot x = 0$ for every $x$, has the Ramsey property.
	A monoid with a \emph{right zero} has the weak Ramsey property, since a right zero is a Ramsey arrow.
	In fact, a monoid with a right zero is a weakly thin category.
	
	But not every monoid with the Ramsey property has a left zero, and not every monoid with the weak Ramsey property has a right zero.
	For example, consider the commutative monoid on $\omega$ given by $a\cdot b = \max\set{a, b}$.
	It is clear this monoid satisfies the condition of the above corollary, however, there is no left or right zero.
	The infinitude of $\omega$ is a necessity for this which we see in the next observation.
\end{ex}

\begin{observation}
	A finite monoid is Ramsey if and only if it has a left zero.
	
	\begin{proof}
		It suffices to show that a finite monoid with the Ramsey property has a left zero.
		There is are $e, s\in M$ such that $e M = \{s\}$.
		Since $M$ has the unit, we have that $e = s$ is a left zero.
	\end{proof}
\end{observation}

\begin{ex}
	Let $M$ be the monoid $\set{x^n, 0x^n: n \in \omega}$ where the operation is the concatenation of words with discarding everything to the left of an occurrence of $0$, i.e. $M$ is the free monoid generated by $x$ with a right zero $0$ freely added.
	The monoid $M$ has the weak Ramsey property since it has a right zero, but $M$ does not have the Ramsey property since there is no $e \in M$ such that $e = ex$.
\end{ex}

\begin{ex}
	In a left-cancellative monoid $M$, i.e. a monoid whose elements are monomorphisms, an element $\al$ satisfies (LE) if and only if it is a right zero.
	This is because if $x\al \neq \al$ for some $x \in M$, then there is no $e \in M$ such that $ex\al = e\al$.
	At the same time, $M$ has no idempotents but the unit, and a right zero is an idempotent.
	It follows that no non-trivial left-cancellative monoid has an element satisfying (LE).
	In particular, it cannot have the Ramsey property, and if $M$ is a group or is commutative, then it has not even the weak Ramsey property since in those cases Ramsey arrows satisfy (LE).
\end{ex}

\begin{question}
	Is there a non-trivial left-cancellative monoid with the weak Ramsey property?
\end{question}

\begin{ex}
	A non-trivial free monoid $M$ does not have the weak Ramsey property.
	We can explicitly construct a coloring: let $x \in M$ be a generator, and let $\phi\maps M \to 2$ be the parity of the number of occurrences of $x$ in a given word.
	Then for every $\al, e \in M$ we have $\phi(ex\al) \neq \phi(e\al)$, and so $\al$ is not a Ramsey arrow.
\end{ex}

\begin{df}
	For every monoid $M$ we may define the \emph{left absorption relation} $x \geq y$ by $x = xy$ for $x, y \in M$.
	The relation is transitive: if $x = xy$ and $y = yz$, then $x = xy = x(yz) = (xy)z = xz$.
	We have $x \leq x$ if and only if $x$ is an idempotent, so $\leq$ is a genuine pre-order if and only if $M$ consists of idempotents.
	Nevertheless, we shall use the pre-order terminology even in the general case.
	
	The unit $1$ is the unique minimum: clearly it is a minimum, and if $x \leq 1$, then $1 = 1x = x$.
	The condition (LE) for an element $\al \in M$ says that $(M\al, \leq)$ is directed.
	Note that in a left-cancellative monoid the relation $\leq$ is trivial: all elements $x \neq 1$ are non-idempotent pairwise incomparable maxima.
	If $x \in M$ is an idempotent, then $ex \geq x$ for every $e \in M$ since $(ex)x = e(xx) = ex$.
\end{df}

\begin{prop}
	An element $\al$ of a monoid $M$ of idempotents is a Ramsey arrow if and only if it satisfies (LE) if and only if it is amalgamable.
	
	\begin{proof}
		The left absorption relation is a pre-order, and for every $x, y \in M\al$ we have that $Mx$ and $My$ are $\leq$-upper subsets of $M\al$.
		Suppose that $\al$ is a Ramsey arrow.
		We prove that $Mx \cap My \neq \emptyset$.
		Otherwise, there is a coloring $\phi\maps M\al \to 2$ such that $Mx$ and $My$ are colored by different colors.
		Since $\al$ is Ramsey, there is $e \in M$ such that $\phi(Mx) \ni \phi(ex) = \phi(ey) \in \phi(My)$, which is a contradiction.
		Hence, $x, y \in M\al$ have an upper bound, so we have (LE).
		
		Clearly, if $\al \in M$ satisfies (LE), then it is amagamable.
		On the other hand, if for $x, y \in M\al$ there are elements $x', y' \in M$ such that $x'x = y'y =: e$, then $ex = ey$.
	\end{proof}
\end{prop}

\begin{ex}
	Let $M$ be a monoid of idempotents.
	If $M$ is also commutative, then $xy = yx$ is the supremum of $\set{x, y}$: clearly $xy$ is an upper bound since $(xy)x = y(xx) = yx = xy$ and similarly for $y$.
	The fact that for $x, y \leq z$ we have also $xy \leq z$ holds in general: we have $z = zx = zy$, and so $z = zy = (zx)y = z(xy)$.
	Moreover, $\leq$ is an order: if $x \leq y \leq x$, then $x = xy = yx = y$.
	Together, a commutative monoid of idemponents with the left absorption order becomes a join semilattice with the minimum $1$.
	Conversely, every join semilattice with a minimum is a monoid.
	We shall call these \emph{semilattice monoids}.
	Since the order is directed (we have even suprema), semilattice monoids have the Ramsey property.
\end{ex}

\begin{ex}
	Another special class of monoids of idempotents are \emph{left-zero monoids}, i.e. monoids $M$ such that for every $x, y \in M \setminus \set{1}$ we have $xy = x$.
	This means that the left absorption pre-order trivializes: all elements $x \neq 1$ are $\leq$-equivalent maxima.
	Hence, left-zero monoids have the Ramsey property.
	Similarly, a \emph{right-zero monoid} consists of the unit and of right zeros, and so has the weak Ramsey property.
	In this case, if $x \geq y \neq 1$, then $x = xy = y$, so non-unit elements are pairwise incomparable maxima.
	Hence, non-trivial right-zero monoids do not have the Ramsey property.
\end{ex}

\begin{question}
	Is being a Ramsey arrow equivalent to (LE) in a general monoid?
\end{question}

\subsection{Almost linear orders}

We revisit Pouzet's example (mentioned in a note by Pabion~\cite{Pabion}, see also \cite{WeakEx}) of a weak Fraïssé category without the cofinal amalgamation property and observe that it also has the weak Ramsey property.
The basic idea is to consider the linear order of rationals $(\Qyu, \leq)$, which is known to be a Fraïssé limit and known to have extremely amenable automorphism group, and to interpret it using a different language, namely the ternary relation $R(x, y, z)$ defined by $(x < y) \meet (x < z) \meet (y \neq z)$.

Let $\LO$ denote the category of all linear orders and all embeddings.
Note that an order $P$ is linear if and only if every two-element subset $\set{x, y} \subs P$ has a minimum.
For the sake of this example we say that an order $P$ is \emph{almost linear} if every three-element subset $\set{x, y, z} \subs P$ has a minimum.
We define $\aLO$ to be the category of all almost linear orders and all one-to-one homomorphisms.
Note that every one-to-one homomorphism $f\maps P \to Q$ between orders is already an embedding if $P$ is linear: if $f(x) \leq f(y)$, then either $x \leq y$ and we are done, or $y \leq x$, so $f(y) \leq f(x)$ and $f(x) = f(y)$, and so $x = y$.
It follows that $\LO$ is a full subcategory of $\aLO$.

Let us view ordinals as linear orders, for example $2 = \set{0 < 1}$, and for orders $X, Y$ let $X + Y$ denote the ordered sum, i.e. the disjoint union of $X$ and $Y$ with $x < y$ for every $x \in X$ and $y \in Y$, so for example $1 + 2$ is a linear order isomorphic to $3$.
Let $B$ denote an order consisting of two incomparable elements.

\begin{observation} \label{almost_linear_objects}
	Every almost linear order $X$ is either linear or of the form $L + B$ for some linear order $L$.

	\begin{proof}
		For every $x \in X$ we have that $\set{y \in X: y \leq x}$ is linearly ordered since for every $y, z < x$ with $y \neq z$ the set $\set{x, y, z}$ has a minimum, which cannot be $x$.
		Also, for a similar reason, for every incomparable elements $x \neq y \in X$ and $z \in X \setminus \set{x, y}$ we have $z < x, y$.
		Together, there may be at most one pair of incomparable elements, and in that case they are maxima and everything below them is linear.
	\end{proof}
\end{observation}

Every almost linear order $X$ with two maxima $x, y$ admits exactly two linear refinements: we decide which of $x, y$ will be greater, and extend the relation appropriately.
The two corresponding homomorphisms will be called \emph{refinements} here.

\begin{observation}
	Every $\aLO$-arrow $f\maps X \to Y$ is either an embedding, or a refinement followed by an embedding.
	
	\begin{proof}
		As observed earlier, if $X$ is linearly ordered, then $f$ is an embedding.
		Otherwise, let $x, y$ be the two maxima of $X$.
		If $f(x), f(y)$ are incomparable, then they are the two maxima of $Y$, and $f$ is again an embedding.
		If $f(x), f(y)$ are comparable, then $f = f' \cmp g$ where $g\maps X \to X'$ is the refinement of $X$ corresponding to the order of $f(x), f(y)$, and the uniquely determined map $f'$ is an embedding.
	\end{proof}
\end{observation}

\begin{con}
	For an almost linear order $(X, \leq)$ we define a ternary relation $R$ on $X$ by 
	\[
		R(x, y, z) \letiff (x < y) \meet (x < z) \meet (y \neq z),
	\]
	meaning that $x$ is the minimum of the three-point set $\set{x, y, z}$.
	The relation $R$ satisfies the following properties:
	\begin{enumerate}
		\item $R(x, y, z) \implies \card{\set{x, y, z}} = 3$,
		\item $R(x, y, z) \iff R(x, z, y)$,
		\item $R(x, y, w) \meet R(y, z, w') \implies R(x, z, w')$,
		\item $\card{\set{x, y, z}} = 3 \implies R(x, y, z) \join R(y, z, x) \join R(z, x, y)$,
	\end{enumerate}
	which are in this context called \emph{antireflexivity}, \emph{symmetry}, \emph{transitivity}, and \emph{linearity}, respectively.

	Let $\tLO$ denote the category whose objects are all structures $(X, R)$ with a ternary relation satisfying the axioms above, and whose morphisms are all first-order embeddings.
	In a moment we shall make it clear that $\tLO$ can be identified with a full subcategory of $\aLO$ such that $\aLO = \LO \cup \tLO$ (see Observation~\ref{LO_union}).
	
	The construction $(X, \leq) \mapsto (X, R)$ induces a functor $F\maps \aLO \to \tLO$: it maps an $\aLO$-arrow $f\maps (X, \leq_X) \to (Y, \leq_Y)$ to the $\tLO$-arrow $f\maps (X, R_X) \to (Y, R_Y)$ represented by the same map between the underlying sets (i.e. it is a concrete functor).
	We need to show that a one-to-one homomorphism $f$ between almost linear orders is indeed an embedding of the corresponding $R$-structures.
	
	Clearly if $x$ is the minimum of a three-point set $\set{x, y, z}$, then $f(x)$ is the minimum of the three-point set $\set{f(x), f(y), f(z)}$.
	On the other hand, if $f(x)$ is the minimum of a three-point set $\set{f(x), f(y), f(z)}$, then the set $\set{x, y, z}$ is also three-point, and so has a minimum, which has to be $x$ by the first implication.
	
	Now let $(X, R)$ be an $\tLO$-object.
	We put
	\[
		x < y \letiff \exists w\, R(x, y, w) \quad \text{and} \quad x \leq y \letiff (x = y) \join \exists w\, R(x, y, w).
	\]
	The relation $<$ is antireflexive and transitive by the corresponding properties of $R$, and so it is a strict order.
	Hence, $\leq$ is an order and $<$ is its strict part.
	By the symmetry and linearity of $R$ we have that $\leq$ is an almost linear order.
	Moreover, if $f\maps (X, R_X) \to (Y, R_Y)$ is an embedding, then it is a one-to-one homomorphism of the induced almost linear orders.
	The one-to-one part is clear.
	Suppose $x \leq_X y$.
	We want to show that $f(x) \leq_Y f(y)$.
	This is clear if $x = y$, so we may suppose that $R_X(x, y, w)$ for some $w \in X$.
	But then $R_Y(f(x), f(y), f(w))$, and so $f(x) \leq_Y f(y)$.
	Together, the construction $(X, R) \mapsto (X, \leq)$ induces a concrete functor $G\maps \tLO \to \aLO$.
\end{con}

\begin{prop}
	We have $F \cmp G = \id{\tLO}$, so $G$ is a full embedding.
	Moreover, for every $\aLO$-object $X$ we have that $G(F(X))$ is $X$ with the modified order forgetting the order of the largest and the second largest element (if these exist).
	The corresponding arrows $\eps_X\maps G(F(X)) \to X$ are universal, so we have an adjunction counit $\eps\maps G \cmp F \to \id{\aLO}$.
	Together, $\tLO$ may be identified with a full coreflective subcategory of $\aLO$.
	
	\begin{proof}
		Let $(X, R)$ be an $\tLO$-object, $(X, \leq) := G(X, R)$, and $(X, R') := F(X, \leq)$.
		We want to show $R'(x, y, z) \iff R(x, y, z)$.
		By the definition, $R'(x, y, z)$ is equivalent to $(x < y) \meet (x < z) \meet (y \neq z)$, which is in turn equivalent to $\exists w, w'\, R(x, y, w) \meet R(x, z, w') \meet (y \neq z)$.
		This is clearly implied by $R(x, y, z)$ since we may simply put $w := z$ and $w' := y$.
		On the other hand, by the linearity of $R$ we have $R(x, y, z) \join R(y, z, x) \join R(z, x, y)$, so it is enough to show that $R(y, z, x) \join R(z, x, y)$ is not possible.
		Also $R(z, x, y)$ is not possible since by $R(x, z, w')$ and the transitivity we would obtain $R(z, z, w')$, which contradicts the antireflexivity.
		Similarly, $R(y, z, x)$, which is equivalent with $R(y, x, z)$ contradicts $R(x, y, w)$.
		
		For the second part, let $(X, \leq)$ be an $\aLO$-object and let $(X, \leq') := G(F(X, \leq))$.
		Clearly $x <' y$ if and only if $x < y$ and there is some $z \neq y$ such that $x < z$, so $\leq'$ only forgets the order of the two consecutive $\leq$-largest elements (if they exist).
		We want to show that $\id{X}\maps (X, \leq') \to (X, \leq)$ is the universal arrow, i.e. that every map $f\maps Y \to X$ for an almost linear order $(Y, \leq)$ of the form $G(Y, R)$ is an $\aLO$-arrow $(Y, \leq) \to (X, \leq)$ if and only if it is an $\aLO$-arrow $(Y, \leq) \to (X, \leq')$.
		The backward implication is trivial.
		For the forward implication, the only possible difference between $\leq$ and $\leq'$ is that $(X, \leq)$ is a linear order $X \setminus \set{a, b} < a < b$ for some $a, b \in X$, while $X \setminus \set{a, b} <' a, b$ with $\leq'$-incomparable $a, b$.
		So the only possible fail is when $a = f(a')$ and $b = f(b')$ for some comparable elements $a', b' \in Y$.
		Since we suppose that $f\maps (Y, \leq) \to (X, \leq)$ is a one-to-one homomorphism, necessarily $Y \setminus \set{a', b'} < a' < b'$.
		But that is a contradiction with the fact that $(Y, \leq) = G(Y, R)$.
	\end{proof}
\end{prop}

As a consequence of the previous proposition and Observation~\ref{almost_linear_objects} we obtain the following summary.
\begin{wn} \label{LO_union}
	When we view $\tLO \subs \aLO$, we have $\aLO = \LO \cup \tLO$.
	Moreover, $\LO \setminus \tLO$ consists exactly of linear orders of the form $L + 2$, while $\tLO \setminus \LO$ consists exactly of almost linear orders of the form $L + B$.
	For every linear order $L$ we have $F(L + 2) = L + B$.
\end{wn}

Let $\FinLO$, $\FinaLO$, and $\FintLO$ denote the full subcategories of $\LO$, $\aLO$, and $\tLO$, consisting of all finite structures.
We will show that $\FintLO$, which is a hereditary class of first-order structures, is a weak Ramsey category with almost no amalgamable objects (only the empty and singleton structures are amalgamable).
Recall the notion of amalgamation extension from Section~\ref{sec:extensions}.

\begin{prop}
	Both $\FinLO$ and $\FintLO$ are full cofinal subcategories of $\FinaLO = \FinLO \cup \FintLO$.
	A $\FinaLO$-object is amalgamable if and only if it is a linear order, i.e. a $\FinLO$-object.
	A $\FintLO$-arrow is amalgamable if and only if it factorizes through a linear order, which happens if and only if it is not of the form $f + B\maps L + B \to L' + B$ for an $\FinLO$-arrow $f\maps L \to L'$.
	Together, $\FinaLO$ is an amalgamation extension of $\FintLO$.
	
	\begin{proof}
		Both $\FinLO$ and $\FintLO$ are cofinal: a non-linear almost linear order $L + B$ admits a refinement $L + B \to L + 2$; a linear order $L$ admits an embedding $L \to L + B$.
		The rest of the first claim follows from Corollary~\ref{LO_union}.
		
		It is well-known that $\FinLO$ has the amalgamation property, and so every $\FinLO$-object is amalgamable in $\FinaLO$ since $\FinLO$ is a full cofinal subcategory.
		On the other hand, if a $\FinaLO$-arrow $\al\maps X \to Y$ does not factorize through a linear order, then $X$ has two maxima that are mapped to the two maxima of $Y$.
		The two linear refinements $f, g\maps Y \to Y'$ are not amalgamable over $\al$, and so $\al$ is not an amalgamable arrow.
	\end{proof}
\end{prop}

It follows from Proposition~\ref{thm:amalgamation_extension} that $\FintLO$ is a weak Fraïssé category with no amalgamable objects but the degenerate ones (the empty and singleton orders).
Moreover, it is well-known that $\FinLO$ has the Ramsey property – this corresponds to the classical finite Ramsey theorem, and so $\FintLO$ has the weak Ramsey property.

Let us also describe the situation with the (weak) Fraïssé limit.
Let $\sigma\FinLO$, $\sigma\FinaLO$, and $\sigma\FintLO$ be the corresponding $\sigma$-closures (Construction~\ref{sigma_closure}), which in this case are the full subcategories of $\LO$, $\aLO$, and $\tLO$ of all countable structures.
By the general theory (Remark~\ref{rm:Flim}), each pair $(\FinLO, \sigma\FinLO)$, $(\FinaLO, \sigma\FinaLO)$, and $(\FintLO, \sigma\FintLO)$ has a weak Fraïssé limit.
By Proposition~\ref{thm:amalgamation_extension}, the corresponding weak Fraïssé sequences are isomorphic, and so are their colimits, so all three pairs have a common weak Fraïssé limit, necessarily being an object of $\sigma\FinLO \cap \sigma\FintLO$.
Of course, the limit is the order of rationals $(\Qyu, \leq)$ as the well-known Fraïssé limit of finite linear orders.
The classical KPT correspondence translates the Ramsey property of $\FinLO$ to the extreme amenability of $\aut(\Qyu, \leq)$, which is then translated by the weak KPT correspondence (Theorem~\ref{thm:KPT}) to the weak Ramsey property of $\FintLO$.
This is an alternative proof of the weak Ramsey property of $\FintLO$.

\newcommand{\Tree}{\mathsf{Tree}} 
\newcommand{\below}{{\leftarrow}} 
\renewcommand{\above}{{\rightarrow}} 

\newcommand{\SplTree}{\mathsf{SplTree}} 
\newcommand{\FSplTree}{\mathsf{FSplTree}} 
\newcommand{\spl}{\operatorname{spl}} 
\newcommand{\Spl}{\operatorname{Spl}} 
\newcommand{\Lab}[1]{\overline{#1}} 
\newcommand{\Dec}[1]{\widehat{#1}} 
\newcommand{\DSplTree}{\mathsf{DSplTree}} 
\newcommand{\dspl}{\operatorname{\overline{spl}}} 

\newcommand{\LevTree}{\mathsf{LevTree}} 
\newcommand{\Lev}{\operatorname{Lev}} 
\newcommand{\lev}{\operatorname{lev}}
\newcommand{\lv}{^\text{lev}} 
\newcommand{\eqlv}{\strictiff} 

\newcommand{\LexTree}{\mathsf{LexTree}} 
\newcommand{\lex}{^\text{lex}} 

\newcommand{\StrTree}{\mathsf{StrTree}} 
\newcommand{\LexStrTree}{\mathsf{LexStrTree}} 
\newcommand{\LexLevTree}{\mathsf{LexLevTree}} 
\newcommand{\LexSplTree}{\mathsf{LexSplTree}} 

\newcommand{\T}{\mathfrak{T}} 
\newcommand{\Tw}{\T_M} 
\newcommand{\Tc}{\Lab\T\vphantom{\T}_M} 
\newcommand{\Ta}{\mathfrak{D}_M} 
\newcommand{\U}{U_M} 
\newcommand{\LT}{\mathfrak{L}\T_M}
\newcommand{\LTw}{\Tw'} 
\newcommand{\LTc}{\Tc'}
\newcommand{\LTa}{\Ta'}
\newcommand{\LU}{U'_M}
\newcommand{\WD}[1]{W_{#1}} 

\newcommand{\TermExt}{\vartriangleleft} 
\newcommand{\NTermExt}{\vartriangleright} 

\newcommand{\ord}{\operatorname{ord}} 
\newcommand{\Br}{\operatorname{Br}} 

\subsection{Trees}

It turns out there are many notions of trees and embeddings of trees.
In this section we consider certain categories $\Tw$ of structured finite trees and \emph{strong embeddings} related to the classical Milliken's theorem~\cite{MilTrees}.
We show that the categories $\Tw$ are weak Fraïssé and describe the countable trees $\U$ that are their weak Fraïssé limits.
After we forget part of the structure, namely the levels, the corresponding categories $\LTw$ are still weak Fraïssé, and moreover they have the weak Ramsey property.
By the KPT correspondence (Theorem~\ref{thm:KPT}), the corresponding limit trees $\LU$ have extremely amenable automorphism groups.

Since we shall consider both finite and infinite trees with various extra structure, we carefully define our categories in several steps.
By a \emph{tree} we mean a triple $(T, \leq, \meet)$ where $(T, \leq)$ is a partially ordered set such that every set $(\below, x]_T := \set{y \in T: y \leq x}$ is linearly ordered and such that every pair $x, y \in T$ has the $\leq$-meet $x \meet y$.
By an \emph{embedding of trees} we mean a one-to-one map $f\maps S \to T$ between trees such that $f(x \meet y) = f(x) \meet f(y)$ for every $x, y \in S$ (and so $x \leq y \iff f(x) \leq f(y)$ for every $x, y \in S$), i.e. a first-order embedding in the language $\set{\leq, \meet}$.
The category of all trees and all embeddings is denoted by $\Tree$.

An element of a tree is often called a \emph{node}, the minimum element (if it exists) is called the \emph{root}, and a maximal element is called a \emph{terminal node}.
A node $s$ is an \emph{immediate successor} of a node $t$ if $s > t$ and there is no $x$ such that $s > x > t$.
We use interval notation with respect to the tree order, i.e. $[t, \above)_T := \set{x \in T: x \geq t}$, $(s, t)_T := \set{x \in T: s < x < t}$, etc.
We omit the $T$ subscript when the tree is clear from the context.
A tree $T$ is \emph{well-founded} if every set $(\below, t]$ for $t \in T$ is well-ordered.
Trees considered in set-theory are often well-founded, but here we consider also trees that are not well-founded, e.g. the set $(\below, t)$ may be isomorphic to $(\Qyu, \leq)$.

\begin{df}[Splitting degree and splitting preserving embeddings]
	For every tree $T$ and $t \in T$ we may consider the \emph{splitting equivalence} defined on $(t, \above)$ by $x \sim y \iff x \meet y > t$.
	(The transitivity follows since $x \meet z \in \set{x \meet y, y \meet z}$.)
	Let $\Spl(t)$ denote the set of the equivalence classes $(t, \above)/{\sim}$, which correspond to the connected parts strictly above $t$, and let $\spl(t)$ denote the \emph{splitting degree} at a non-terminal node $t$ defined as the cardinality of $\Spl(t)$.
	Of course, if the tree $T$ is well-ordered, $\Spl(t)$ corresponds to immediate successors of $t$, but in general it is not the case.
	Also, maximal antichains in $(t, \above)$ are exactly transversals of $\Spl(t)$.
	A tree is \emph{finitely splitting} if every set $\Spl(t)$, $t \in T$, is finite.
	
	For every embedding $f\maps T \to S$ we have that a class $[x] \in \Spl(t)$ gets mapped into the class $[f(x)] \in \Spl(f(t))$ and that different classes get mapped into different classes, so $f$ induces an injective map $\Spl_T(t) \to \Spl_S(f(t))$.
	We say that $f$ \emph{preserves splitting} if for every non-terminal node $t \in T$ the induced map $\Spl_T(t) \to \Spl_S(f(t))$ is a bijection.
	For finitely splitting trees, this is equivalent to preserving the splitting degree of all non-terminal nodes of $T$.
	Let $\SplTree \subs \Tree$ be the wide subcategory of all trees and all splitting preserving embeddings, and let $\FSplTree \subs \SplTree$ be the full subcategory of finitely splitting trees.
	
	It will be useful to consider also trees whose splitting degree has been decided on terminal-nodes.
	Let $\Lab\FSplTree$ denote the category of all trees $T$ endowed with unary relations $\set{R_n}_{n \in \Nat^+}$ such that for every $t \in T$ at most one of $R_n(t)$ holds, and if $t$ is non-terminal, then $R_{\spl(t)}(t)$ holds.
	The morphisms are first-order one-to-one homomorphisms in the language $\set{\leq, \meet} \cup \set{R_n}_{n \in \Nat^+}$.
	Nodes $t \in T$ for which some $R_n(t)$ holds are called \emph{decided}.
	Clearly, $\FSplTree$ can be identified with the full subcategory of $\Lab\FSplTree$ of trees with undecided terminal nodes.
	On the other hand, by $\DSplTree$ we denote the full subcategory of $\Lab\FSplTree$ of trees whose all terminal nodes are decided.
\end{df}

\begin{df}[Leveled trees]
	By a \emph{leveled tree} we mean a tree $T$ equipped with an equivalence $\eqlv$ encoding the information of which nodes are on the same level, i.e. the level set $\Lev(T) := T/{\eqlv}$ with the induced order is linearly ordered, and for every $t \in T$ the \emph{level function} $\lev\maps T \to \Lev(T)$ induces an order isomorphism $(\below, t]_T \to (\below, \lev(t)]_{\Lev(T)}$.
	By $\LevTree$ we denote the category of all leveled trees and first-order embeddings (or equivalently one-to-one homomorphisms) in the language $\set{\leq, \meet, \eqlv}$.
	Note that a $\LevTree$-embedding $f\maps S \to T$ induces an order embedding $\Lev(S) \to \Lev(T)$.
	
	The preorder on $T$ induced by $\Lev(T)$ is denoted by $\leq\lv$, i.e. $x \leq\lv y \iff \lev(x) \leq \lev(y)$.
	The set $\set{t \in T: \lev(t) = \al}$ is denoted by $T(\al)$.
	For every $\al \leq \lev(t)$ the unique node $s \leq t$ such that $\lev(s) = \al$ is denoted by $t\restr{\al}$.
	Sometimes we implicitly identify the level set $\Lev(T) = T/{\eqlv}$ with another linearly ordered set, e.g. an ordinal or the rational numbers.
	
	Clearly, every well-ordered tree $T$ admits a unique level structure with $\Lev(T)$ being an ordinal, yet a $\Tree$-map between well-ordered trees is not necessarily level preserving.
	On the other hand, there are trees admitting several non-isomorphic level structures, e.g. two copies of $(\Qyu, \leq)$ above a common root with one level structure gluing the copies of $\Qyu$ and another level structure making one copy a strict initial segment of the other.
	There are also trees admitting no level structure, e.g. $(\Qyu, \leq)$ and a terminal node above a common root.
	
	A leveled tree is called \emph{balanced} if for every pair of nodes $t <\lv s$, there is a node $t'$ such that $t < t' \eqlv s$.
	A finite tree is balanced if and only if all its terminal nodes are at the same level.
\end{df}

\begin{df}[Lexicographic trees]
	By a \emph{lexicographic tree} we mean a tree $T$ endowed with a linear order $\leq\lex$ that is coherent with the splitting structure of $T$, i.e. for every $t \in T$ there is a linear order $\leq^t$ on $\Spl(t)$ such that $x \leq\lex y$ if and only if $x \leq y$ or $x, y > x \meet y$ and $[x] <^{x \meet y} [y]$.
	Informally, the tree is linearly ordered from bottom to top and from left to right.
	By $\LexTree$ we denote the category of all lexicographic trees and all first-order embeddings (or equivalently one-to-one homomorphisms) in the language $\set{\leq, \meet, \leq\lex}$.
\end{df}
 
\begin{df}[Strong embeddings]
	By a \emph{strong embedding} we mean and embedding of leveled trees preserving both splitting and levels.
	A \emph{strong subtree} of a leveled tree $T$ is a subset $S$ such that the inclusion $S \subs T$ is a strong embedding.
	By $\StrTree$ we denote the category of leveled trees (called \emph{strong trees} in this context) and strong embeddings, 
	and by $\LexStrTree$ we denote the category of lexicographic leveled trees and strong embeddings also preserving the lexicographic order. 
	We add the lexicographic order in order to ensure the rigidity needed for the weak Ramsey property.
	By $\LexLevTree$ and $\LexSplTree$ we denote the expansions of $\LevTree$ and $\SplTree$ by the lexicographic order.
	We summarize the categories in Figure~\ref{fig:tree_cats}.
	
	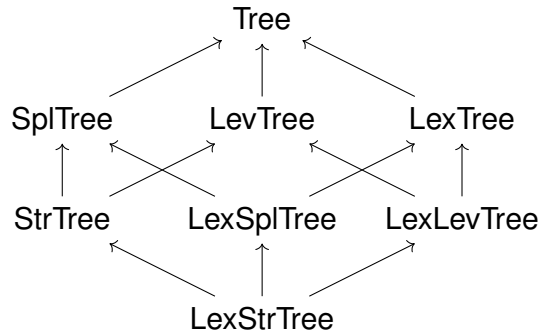
\begin{figure}[ht]
	\centering
	
	\begin{tikzpicture}[
			text diagram,
			x = {(14ex, 0)},
			y = {(0, -7ex)},
		]
		\node (t) at (0, 0) {$\Tree$};
		\node (spt) at (-1, 1) {$\SplTree$};
		\node (lvt) at (0, 1) {$\LevTree$};
		\node (lxt) at (1, 1) {$\LexTree$};
		\node (str) at (-1, 2) {$\StrTree$};
		\node (lxspt) at (0, 2) {$\LexSplTree$};
		\node (lxlvt) at (1, 2) {$\LexLevTree$};
		\node (lxstr) at (0, 3) {$\LexStrTree$};
		
		\graph{
			{(spt), (lvt), (lxt)} -> (t),
			(str) -> {(spt), (lvt)},
			(lxspt) -> {(spt), (lxt)},
			(lxlvt) -> {(lvt), (lxt)},
			(lxstr) -> {(str), (lxspt), (lxlvt)},
		};
	\end{tikzpicture}
	
	\caption{Categories of trees and forgetful functors between them.}
	\label{fig:tree_cats}
\end{figure}
\end{df}

\begin{remark}
	Observe that an inclusion of a subset $S \subs T$ of a finite tree with the induced order preserves meets (and splitting) if and only if for every non-terminal node $s \in S$ and every its immediate successor $t \in T$ we have that $[t, \above)_T$ contains at most one (exactly one) immediate successor of $s$ in $S$.
	Therefore, $\Tree$-embeddings, $\SplTree$-embeddings, and $\StrTree$-embeddings correspond to weakly embedded, embedded, and strongly embedded subtrees as defined in the original Milliken's paper \cite{MilTrees}.
	The lexicographic order is also considered in \cite{MilTrees} by the name extended order.
\end{remark}

\begin{df}[Categories of finite trees]
	We finally define the categories of finite strong trees we are interested in.
	Let $M \subs \Nat^+$ be a nonempty set.
	$\Tw$ denotes the full subcategory of $\LexStrTree$ consisting of finite trees such that the splitting degree takes values only in $M$.
	Note that if $M = \set{1}$, then $\Tw$ consists of finite linearly ordered sets, and strong embeddings are just increasing embeddings.

	Moreover, $\Tc$ denotes the variant of $\Tw$ where terminal nodes may have decided splitting without actually being split at, i.e. the $\FSplTree$-component of the category is replaced by $\Lab\FSplTree$.
	Of course, $\Tw$ is identified with full a subcategory of $\Tc$.
	By $\Ta$ we denote the full subcategory of $\Tc$ of trees where all terminal nodes are decided.
\end{df}

\subsubsection{Extensions and amalgamations}

In this section we describe and classify arrows in $\Tw$, $\Tc$, and $\Ta$ for a fixed nonempty set $M \subs \Nat^+$, and characterize amalgamable objects and arrows.

First, observe that without loss of generality we may view a $\Tc$-arrow $f\maps S \to T$ as an inclusion $S \subs T$ (formally we compose $f$ with an isomorphism $g\maps T \to T'$ such that $g \cmp f\maps S \subs T'$).
So instead of embeddings we may work with \emph{extensions}, which simplifies the proofs.
Second, given an extension $S \subs T$ there are two aspects:
one, newly added nodes, and two, \emph{deciding} the splitting degree of existing terminal nodes without adding new nodes above them, which is possible in $\Tc$ and required in $\Ta$.
We shall focus on the first aspect and mention the second one only when important.
So from now on we keep in mind that we are working in any of $\Tw$, $\Tc$, $\Ta$, without necessarily specifying which one.
By a \emph{tree} we mean an object of our category (in particular it is always a finite tree), by an \emph{embedding} we mean a morphism, and by an \emph{extension} we mean an embedding that is also an inclusion of sets.

\begin{df}
	We say that an extension $S \subs T$ is \emph{terminal} if $S$ is a lower subset of $T$, i.e. there is no newly added node $t \in T \setminus S$ below any existing node $s \in S$.
\end{df}

\begin{con}[Terminal planting]
	Let $S, T$ be trees, let $s \in S$ be a terminal node, and let $r$ denote the root of $T$.
	We consider the tree $S \TermExt_s T$ resulting from planting the tree $T$ at the terminal node $s \in S$, like on Figure~\ref{fig:planting}.
	
	\begin{figure}[ht]
	\centering
	
	\begin{tikzpicture}[trees]
		\path[dot nodes] node (root) {}
			child[shift children=-2\ru] {node {}
				child {node {}}
				child {node {}}
			}
			child[subtree={blue}{up triangle node}] {node[red, diamond node] (s) {}
				child {node {}
					child {node {}}
					child {node {}}
				}
				child {node {}}
			}
			child {node {}}
		;
		
		\node[below left=1\ru] at (root-1) {$S$};
		\node[red, right=0.8\ru] at (s) {$s$};
		\node[blue, above right=0.8\ru] at (s-2) {$T$};
	\end{tikzpicture}
	
	\caption{A terminal planting $S \TermExt_s T$.}
	\label{fig:planting}
\end{figure}
	
	The trees $S$ and $T$ can be replaced by isomorphic copies, so we may suppose that $s = r$ and $S \cap T = \set{s}$.
	Then there is a unique lexicographic strong tree structure on $S \cup T$ that extends both $S$ and $T$: for every $u \in S$ and $t \in T$ we have $u \meet t = u \meet_S s$; the splitting at every node is realized either in $S$ or $T$, and so the lexicographic order is uniquely determined, and since our trees are finite, there is a unique level structure.
	For convenience, given the original trees with the sets $S, T$ we can take $S \TermExt_s T = S \cup \set{(s, t): r \neq t \in T}$ as the representing set so that we have an extension $S \subs (S \TermExt_s T)$.
	
	Note that up to a canonical isomorphism we have $(S \TermExt_s T) \TermExt_t R = S \TermExt_s (T \TermExt_t R)$ for $s$ a terminal node of $S$ and $t$ a terminal node of $T$.
	Similarly, $(S \TermExt_{s_1} T_1) \TermExt_{s_2} T_2 = (S \TermExt_{s_2} T_2) \TermExt_{s_1} T_1$ for $s_1, s_2$ distinct terminal nodes of $S$.
	Finally, more trees can be planted at distinct terminal nodes at once: for a set of terminal nodes $A \subs S$ and trees $\set{T_a: a \in A}$ we consider the extension $S \TermExt_{a \in A} T_a$.
\end{con}

\begin{prop} \label{thm:terminal_extension}
	For every set $A \subs S$ of terminal nodes and for nonempty trees $T_a$ we have that $S \TermExt_{a \in A} T_a$ is a terminal extension of $S$.
	On the other hand, every terminal extension $S \subs T$ is canonically of the form $S \TermExt_{a \in A} T_a$ with nondegenerate trees $T_a$ (unless $S = \emptyset$).
	
	\begin{proof}
		The first claim is clear from the definition.
		For the second claim let $\emptyset \neq S \subs T$ be a terminal extension.
		For a node $t \in T \setminus S$ let $a = s(t)$ be the maximum of $(\below, t] \cap S$.
		Note that $T_a := [a, \above)$ is a strong subtree of $T$.
		Moreover, we have that $a$ is a terminal node of $S$ and so $T_a \cap S = \set{a}$.
		Otherwise, $a$ would have an immediate successor from $S$ by the terminality of $S \subs T$ and also an immediate successor from $T \setminus S$ by the definition of $a$.
		That would contradict $\spl_S(a) = \spl_T(a)$.
		Together, we have $T = (S \TermExt_{a \in A} T_a)$ for $A := \set{s(t): t \in T \setminus S}$.
	\end{proof}
\end{prop}

\begin{df}
	By a \emph{bush} we mean a tree with exactly two levels: the root and the nonempty set of its immediate successors.
	
	By a \emph{one-step terminal extension} of a tree $S$ we mean an extension of the form $S \TermExt_s B$ where $B$ is a bush.
	Note that every terminal extension $S \subs T$ can be realized as a finite sequence/composition of one-step terminal extensions $S \TermExt_s B_1 \TermExt_{b_1} \cdots \TermExt_{b_{k - 1}} B_k$.
	In particular, every nonempty tree $T$ can be build by planting bushes and starting with a singleton tree $S = \set{*}$.
\end{df}

\begin{df}
	By a \emph{non-terminal extension} we mean the complete opposite of the terminal extension: it is an extension $S \subs T$ such that for every decomposition into two consecutive extensions $S \subs T' \subs T$ such that $T' \subs T$ is a terminal extension, we have $T' = T$.
\end{df}

\begin{prop} \label{thm:extension_decomposition}
	Every extension $S \subs T$ can be canonically decomposed as $S \subs T' \subs T$ in a way that $S \subs T'$ is a non-terminal extension and $T' \subs T$ is a terminal extension.
	
	\begin{proof}
		Let $S' := \bigcup_{s \in S} (\below, s]_T$ be the lower subset of $T$ generated by $S$.
		Clearly, every suitable tree $T'$ has to contain $S'$.
		But $S'$ may not be a strong subtree.
		It is closed under meets and levels, but typically not under splitting: we may have $\spl_{S'}(s) < \spl_T(s)$ for some $s \in S' \setminus S$, so we need to add more nodes witnessing the splitting.
		Fortunately, since our tree is finite and so every level but the top one has a successor level, we have canonical witnesses.
		It is enough to add all $T$-immediate successors of all nodes $s \in S' \setminus S$.
		Note that all $T$-immediate successors of non-terminal nodes of $S$ are already in $S'$.
		So let $S'' := S' \cup \set{t \in T: t$ an immediate successor of a node form $S' \setminus S}$.
		
		We have that $S''$ is the smallest strong subtree of $T$ containing $S'$.
		It follows that every suitable $T'$ has to contain $S''$.
		Also, $S'' \subs T$ is a terminal extension since $S''$ is a lower subset.
		Hence, $S'' \subs T'$ is a terminal extension, and necessarily $T' = S''$.
		It remains to show that $S \subs S''$ is a non-terminal extension.
		Let $N := S'' \setminus T''$ be the set of newly added nodes for some terminal extension $T'' \subs S''$ such that $S \subs T''$.
		Clearly, no node from $S'$ is in $N$.
		Similarly, no node from $S'' \setminus S'$ is in $N$ since its predecessor is in $S' \setminus S \subs T''$ and so would have to be a terminal node of $T''$, which it is not.
		So $N = \emptyset$ and $T'' = S''$.
	\end{proof}
\end{prop}

From the previous proof we obtain the following.

\begin{wn} \label{nonterminal}
	An extension $S \subs T$ is non-terminal if and only if every node $t \in T \setminus S$ is below an $S$-node or its immediate predecessor is a node from $T \setminus S$ below an $S$-node.
\end{wn}

\begin{con}[Tree surgery]
	By a \emph{pointed bush} $B^b$ we mean a pair $(B, b)$ where $B$ is a bush and $b$ is one of its top-level nodes.
	
	For a tree $S$, its level $\al$, and a family of pointed bushes $(B_s^{b_s})_{s \in S(\al)}$, we consider the extension $S \NTermExt_\al (B_s^{b_s})_{s \in S(\al)}$ obtained by performing a (one-step) \emph{tree surgery} on $S$.
	Again, for simplicity we may suppose that $b_s = s$ and $S \cap B_s = \set{s}$ for every $s \in S(\al)$ and that the bushes $B_s$ are pairwise disjoint.
	There is a unique lexicographic strong tree structure on $S \cup \bigcup_{s \in S(\al)} B_s$ that is an extension of $S$ and of every $B_s$.
	The roots of the bushes $B_s$ form a new level immediately preceding $\al$, while the level $\al$ is extended by the top levels of the bushes $B_s$ with the nodes $b_s$ being identified with the old level $S(\al)$, see Figure~\ref{fig:surgery}.
	In the case that $\al$ is the root level, and so we have a single pointed bush $B^b$, we write just $S \NTermExt B^b$.
	Note that this non-terminal extension of $S$ is also the terminal extension $B \TermExt_b S$ of $B$.
	
	\begin{figure}[ht]
	\centering
	
	\begin{tikzpicture}[trees]
		\path[dot nodes] node (root) {}
			child[subtree={blue}{up triangle node}, shift children=-2.5\ru] {node {}
				child {node {}}
				child[subtree={black}{circle node}] {node[red, diamond node] {}
					child {node {}}
					child {node {}}
				}
				child {node {}}
			}
			child[subtree={blue}{up triangle node}] {node {}
				child[subtree={black}{circle node}] {node[red, diamond node] {}
					child {node {}}
					child {node {}}
				}
			}
			child[subtree={blue}{up triangle node}, shift children=2.5ex] {node {}
				child[subtree={black}{circle node}] {node[red, diamond node] {}
					child {node {}}
				}
				child {node {}}
			}
		;
		
		\node[left=1\ru] at (root-1-2-1) {$S$};
		\node[blue, below=1.5\ru] at (root-1-1) {$B_s$};
		\node[red, right=0.8\ru, fill=white, inner sep=0.5\ru, outer sep=0.4\ru] at (root-1-2) {$s$};
		
		\begin{scope}[on background layer]
			\path (root-1-1) ++(-3\ru, 0) coordinate (start)
				(root-3-2) ++(3\ru, 0) coordinate (end) node[right, red] {$\al$};
			\draw[dashed, red] (start) -- (end);
		\end{scope}
	\end{tikzpicture}
	
	\caption{A tree surgery $S \NTermExt_\al (B_s^s)_{s \in S(\al)}$.}
	\label{fig:surgery}
\end{figure}
	
	We may iterate one-step tree surgeries and consider
	\[
		\bigl(\underbrace{\bigl(S \NTermExt_\al (B_{1, s}^{b_{1, s}})_{s \in S(\al)}\bigr)}_{T_1} 
			\NTermExt_{\beta_1} \cdots \bigr) \NTermExt_{\beta_{k - 1}} (B_{k, s}^{b_{k, s}})_{s \in T_{k - 1}(\beta_{k - 1})}
	\]
	where $\beta_i$ is the new level in $T_i$ for $1 \leq i \leq k$.
	We may simplify this by introducing another helper notion.
	A \emph{bush-column} is a tree of the form
	\[
		C = \bigl(((B_0 \NTermExt B_1^{b_1}) \NTermExt \cdots ) \NTermExt B_k^{b_k}\bigr)
			= \bigl(B_k \TermExt_{b_k} B_{k - 1} \TermExt_{b_{k - 1}} \cdots \TermExt_{b_1} B_0\bigr),
	\]
	and a \emph{pointed bush-column} $C^c$ is a pair $(C, c)$ where $C$ is a bush-column and $c$ is a top-level node of $C$.
	Iterated tree surgery may be written as $S \NTermExt_\al (C_s^{c_s})_{s \in S(\al)}$ where $C_s^{c_s}$ for $s \in S(\al)$ are pointed bush-columns of the same height.
	Finally, we may perform tree surgery simultaneously on several levels and use the notation $S \NTermExt_{\al \in A} (C_{\al, s}^{c_{\al, s}})_{s \in S(\al)}$ or just $S \NTermExt_{\al \in A} C_\al$ for a set $A \subs \Lev(S)$ and a sequence of pointed bush-columns of the same height for every $\al \in A$.
\end{con}

\begin{prop} \label{thm:nonterminal_extension}
	Every extension $S \subs \bigl(S \NTermExt_{\al \in A} (C_{\al, s}^{c_{\al, s}})_{s \in S(\al)}\bigr)$ is non-terminal.
	On the other hand, for every non-terminal extension $S \subs T$, $T$ can be canonically written as $S \NTermExt_{\al \in A} (C_{\al, s}^{c_{\al, s}})_{s \in S(\al)}$.
	It follows that every non-terminal extension can be realized as a finite sequence/composition of one-step tree surgeries.
	
	\begin{proof}
		The first part is clear from the definition of tree surgery and Corollary~\ref{nonterminal}.
		For the second part let $S \subs T$ be a non-terminal extension.
		Let $N$ be the set of all nodes $t \in T \setminus S$ such that there is a node $s(t) \in S$ above $t$.
		Moreover, let $s(t)$ denote the least such node, which exists since $S$ is closed under meets.
		Also, let $N' := T \setminus (S \cup N)$, so we have a decomposition $T = S \cup N \cup N'$.
		By Corollary~\ref{nonterminal}, every $t \in N'$ has no $S$-node above and its predecessor is in $N$.
		It follows that $t$ is a terminal node of $T$ (its immediate successor would have to be in $N'$, which would contradict $t \in N'$).
		
		For every $t \in N$ let $B(t) \subs T$ be the bush consisting of $t$ and its immediate successors in $T$.
		We have $N' \subs \bigcup_{t \in N} B(t)$.
		Also, the top level of $B(t)$ may contain at most one node not in $N'$ – either $s(t)$ or another $t' \in N$ with $s(t') = s(t)$.
		This follows from the fact that $S$ is closed under meets.
		
		Next, observe that for every level $\al$ if $T(\al) \cap N \neq \emptyset$, then $T(\al) \subs N \cup N'$.
		This is because $t \leq s(t)$, so if $t \eqlv s' \in S$, then $t \in S$ as $S$ is closed under levels.
		So immediately below every level $\al \in A := \set{\lev_T(s(t)): t \in N}$, there are $T$-levels $\beta_{\al, 1} > \beta_{\al, 2} > \cdots > \beta_{\al, k_\al}$ consisting of new nodes.
		We have $N = \set{s\restr{\beta_{\al, i}}: \al \in A, s \in S(\al), 1 \leq i \leq k_\al}$.
		
		For every $\al \in A$ and $s \in S(\al)$ let $C_{\al, s} := \bigcup\set{B(t): t \in N \meet s(t) = s}$.
		Now it is clear that every $C_{\al, s}$ is a bush-column, every two bush-columns $C_{\al, s}, C_{\al, s'}$ have the same height, and $T = \bigl(S \NTermExt_{\al \in A} (C_{\al, c}^s)_{s \in S(\al)})\bigr)$.
	\end{proof}
\end{prop}

\begin{df}
	By a \emph{one-step extension} we mean either a one-step terminal extension or a one-step tree surgery.
	It follows that a \emph{one-step non-terminal extension} is just a different name for a one-step tree surgery.
\end{df}

\begin{wn} \label{thm:extension}
	Every extension $S \subs T$ is canonically of the form
	\[
		\bigl(S \NTermExt_{\al \in A} C_\al\bigr) \TermExt_{s \in B} T_s
	\]
	where $A$ is a set of $S$-levels, $B$ is a set of terminal nodes of the non-terminal extension tree, $C_\al$ are suitable sequences of pointed bush-columns, and $T_s$ are suitable trees.
	Moreover, every extension is a finite composition of one-step extensions.
\end{wn}

\aseparator

Next we characterize amalgamable arrows in the categories $\Tw$, $\Tc$, and $\Ta$.
We are in a situation as in Section~\ref{sec:extensions}.
The situation can be summarized as follows.

\begin{tw} \label{thm:tree_amalgamation}
	Let $M \subs \Nat^+$ be nonempty.
	Both $\Tw$ and $\Ta$ are full cofinal subcategories of $\Tc$.
	$\Ta$ has the amalgamation property, so $\Ta \subs \am(\Tc)$.
	Moreover, every amalgamable $\Tc$-arrow factorizes through an amalgamable $\Tc$-object, and so (the full subcategory generated by) $\Tw \cup \Ta$ is an amalgamation extension of $\Tw$.
	
	\begin{enumerate}
		\item If $M = \set{m}$ for some $m \in \Nat^+$, then forgetting the decided splitting degree at terminal nodes is an isomorphism of categories $\Ta \to \Tw$, and the whole category $\Tc$ has the amalgamation property.
		\item Otherwise, if $\card{M} \geq 2$, then a $\Tc$-object is amalgamable if and only if it is fully decided, i.e. $\am(\Tc) = \Ta$.
		In particular, a $\Tw$-arrow $f\maps S \to T$ is amalgamable if and only if for every terminal node $s \in S$, $f(s)$ is not terminal in $T$.
	\end{enumerate}
\end{tw}

The goal is to prove the theorem above.

\begin{df}
	Let $f_1\maps S \to T_2$ and $f_2\maps S \to T_2$ be $\Tc$-arrows.
	By a \emph{node of incompatibility} for $f_1, f_2$ we mean a node $s \in S$ such that $f_1(s)$ and $f_2(s)$ are decided, but $\dspl(f_1(s)) \neq \dspl(f_2(s))$.
	Necessarily, $s$ is an undecided terminal node of $S$.
	The pair of embeddings $f_1, f_2$ is \emph{compatible} if there is no node of incompatibility.
\end{df}

\begin{prop}
	A pair of $\Tc$-embeddings $f_1\maps S \to T_1$ and $f_2\maps S \to T_2$ is amalgamable (i.e. there are embeddings $g_i\maps T_i \to T$, $i \in \set{1, 2}$, such that $g_1 \cmp f_1 = g_2 \cmp f_2$) if and only if the pair is compatible.
	
	\begin{proof}
		Clearly, if $s \in S$ is a node of incompatibility, then $\dspl(g_1(f_1(s))) = \dspl(f_1(s)) \neq \dspl(f_2(s)) = \dspl(g_2(f_2(s)))$, so $f_1, f_2$ cannot be amalgamated.
		The other, nontrivial, implication follows from Construction~\ref{con:amalgamations} below.
	\end{proof}
\end{prop}

\begin{proof}[Proof of Theorem~\ref{thm:tree_amalgamation}]
	We already know that $\Tw, \Ta \subs \Tc$ are full.
	The cofinality of $\Ta$ is clear – we just refine the relations by adding $R_m(s)$ for a fixed $m \in M$ and every undecided terminal node $s$ of a given tree $S$.
	For the cofinality of $\Tw$ let $S$ be any $\Tc$-tree.
	We consider the terminal extension $S \TermExt_{s \in D} B_{\dspl(s)}$ where $D$ is the set of decided terminal nodes of $S$, and for every $m \in M$, $B_m$ is an undecided bush whose root has the splitting degree $m$.
	
	Clearly, every pair $f_i\maps S \to T_i$, $i \in \set{1, 2}$, of $\Tc$-arrows is compatible if $S$ is decided or if $M = \set{m}$, so every $\Ta$-object is amalgamable, and in the case~(i), $\Tc$-has the amalgamation property.
	
	The map $F\maps \Ta \to \Tw$ that forgets the decided splitting degrees at terminal nodes on objects and that acts as identity on arrows is always a faithful functor surjective on objects.
	But if $M = \set{m}$, then $F$ is also one-to-one on objects and full since there is only one way how to decide a splitting degree, and so $F$ is an isomorphism.
	
	Finally, if $m_1 \neq m_2 \in M$ and $\al\maps S \to S'$ is a $\Tc$-arrow such that $\al(s)$ is not decided for a node $s \in S$, then $f_1 \cmp \al, f_2 \cmp \al$ where $f_i\maps S' \to T_i$ is the refinement by $R_{m_i}(\al(s))$ are not amalgamable, and so $\al$ is not an amalgamable arrow.
	Hence, an amalgamable tree is necessarily decided in (ii), and every amalgamable arrow factorizes through an amalgamable object (which is trivial in (i)).
\end{proof}

It remains to construct an amalgamation $g_i\maps T_i \to T$, $i \in \set{1, 2}$, for a pair $f_i\maps S \to T$ of compatible embeddings.
As before, we may suppose for simplicity that $f_i$ is an extension $S \subs T_i$.
We will be able to construct an amalgamation such that $g_1(t_1) = g_2(t_2)$ if and only if $t_i = f_i(s)$ for a node $s \in S$ and $i \in \set{1, 2}$.
In this case, the embeddings $g_i$ may be replaced by extensions $T_i \subs T$, so we would have $T_1 \cup T_2 \subs T$ with $T_1 \cap T_2 = S$.
We shall call such situation a \emph{non-gluing amalgamation}.

We have seen in Corollary~\ref{thm:extension} that the extension $S \subs T_i$ may be viewed as a result of a certain construction: $T_i = E_i(S)$.
It may happen that the same construction $E_i$ can be applied also to an extended tree $S' \sups S$.
A non-gluing amalgamation is a \emph{free amalgamation} if $E_1(E_2(S)) = E_2(E_1(S)) = T_1 \cup T_2 = T$.
This happens for example if $T_i = S \TermExt_{s \in A_i} V_s$ for $A_1 \cap A_2 = \emptyset$.

\begin{con} \label{con:amalgamations}
	Let $S \subs T_i$, $i \in \set{1, 2}$, be compatible extensions.
	We shall construct a non-gluing amalgamation $T \sups T_1 \cup T_2$.
	We consider several cases.
	
	\begin{enumerate}
		\item Suppose $T_i = S \TermExt_s B_i$ for $i = 1, 2$ are one-step terminal extensions planting at the same terminal node $s \in S$.
			Clearly, if $\spl_{T_1}(s) \neq \spl_{T_2}(s)$, then there is no amalgamation.
			But this is not the case since the extensions are compatible.
			It follows that the bushes $B_i$ are isomorphic.
			Of course we could amalgamate simply by identifying the bushes $B_i$, but we are interested in a non-gluing amalgamation since the bushes $B_i$ may be just parts of bigger extensions we consider in later cases.
			
			We may suppose $B_1 \cap B_2 = \set{s}$.
			Let $\al$ denote the top level of the bushes $B_i$ as well as the corresponding level of $T_i$.
			Hence, $T_i(\al)$ is the disjoint union $S(\al) \cup B_i(\al)$ with the possibility that $S(\al)$ is empty.
			Let $(b_{i, j}: j < k)$ for $i = 1, 2$ be the $\leq\lex$-increasing enumeration of $B_i(\al)$.
			Furthermore, for every $s \in S(\al)$ let $C_{1, s} = C_{2, s}$ be a bush containing $s$ at the top level,
			and for every $j < k$ let $C_{1, b_{1, j}} = C_{2, b_{2, j}}$ be a bush containing $b_{1, j}$ and $b_{2, j}$ as distinct top-level nodes (here we suppose $M \neq \set{1}$, otherwise the amalgamation is done as for linear orders).
			Suppose the bushes are otherwise disjoint with all the other bushes considered.
			Our amalgamation is 
			\[
				T :=\, \underbrace{(S \TermExt_s B_1)}_{T_1} \NTermExt_\al (C_{1, t}^t)_{t \in T_1(\al)}
				\,=\, \underbrace{(S \TermExt_s B_2)}_{T_2} \NTermExt_\al (C_{2, t}^t)_{t \in T_2(\al)}
			\]
			as in Figure~\ref{fig:TT_amalgamation}.
			Note that the extensions $T_i \subs T$ for $i = 1, 2$ are one-step non-terminal.
			
			\begin{figure}[ht]
	\centering
	
	\begin{tikzpicture}[
			trees,
			glued/.style 2 args = {fill=#1, draw=#2, ultra thick},
		]
		
		\path[dot nodes] node (T1) at (-25ex, 2.5\ru) {}
			child[subtree={green}{left triangle node}] {node {}
				child {node {}}
				child {node {}}
				child {node {}}
			}
			child[shift children=2.5\ru] {node {}
				child {node {}}
			}
		;
		
		\path[dot nodes,
			level 1/.style = {sibling distance = 9\ru},
			level 2/.style = {sibling distance = 6\ru},
			level 3/.style = {sibling distance = 3\ru},
		] node (T) {}
			child[subtree={green}{left triangle node}] {node {}
				child[subtree={blue}{diamond node}] {node {}
					child {node[green, left triangle node] {}}
					child {node[red, right triangle node] {}}
				}
				child[subtree={blue}{diamond node}] {node {}
					child {node[green, left triangle node] {}}
					child {node[red, right triangle node] {}}
				}
				child[subtree={blue}{diamond node}] {node {}
					child {node[green, left triangle node] {}}
					child {node[red, right triangle node] {}}
				}
			}
			child[shift children=3\ru] {node {}
				child[subtree={blue}{diamond node}] {node {}
					child {node[black, circle node] {}}
					child {node {}}
				}
			}
		;
		
		\path[dot nodes] node (T2) at (25ex, 2.5\ru) {}
			child[subtree={red}{right triangle node}] {node {}
				child {node {}}
				child {node {}}
				child {node {}}
			}
			child[shift children=2.5\ru] {node {}
				child {node {}}
			}
		;
		
		\node[green, left=1.5\ru] at (T1-1) {$B_1$};
		\node[red, left=1.5\ru] at (T2-1) {$B_2$};
		\node[blue, left=1\ru] at (T-1-1) {$C_{i, t}$};
	\end{tikzpicture}
	
	\caption{An amalgamation of two one-step terminal extensions.}
	\label{fig:TT_amalgamation}
\end{figure}
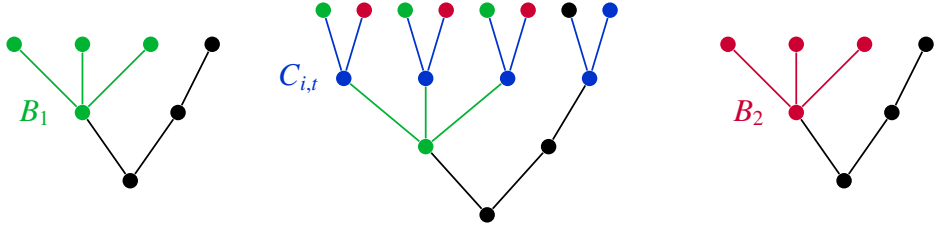
		
		\item To amalgamate general terminal extensions $S \TermExt_s V_i$ at the same node $s \in S$ for $i = 1, 2$ we write $V_i$ as $B_i \TermExt_{t \in A_i} V'_t$ where $B_i$ is the root bush of $V_i$ and $A_i \subs B_i$ is a subset of the top level of the bush (we suppose that $V_1 \cap V_2 = \set{s}$, and so $A_1 \cap A_2 = \emptyset$).
		Then we consider the free amalgamation $S' \TermExt_{t \in A_1 \cup A_2} V'_t$ of the extensions $S' \TermExt_{t \in A_i} V'_t$ for $i = 1, 2$ where $S'$ is a non-gluing amalgamation of $S \TermExt_s B_1$ and $S \TermExt_s B_2$ from (i).
		
		\item To amalgamate completely general terminal extensions $S \TermExt_{s \in A_i} V_{i, s}$ for $i = 1, 2$ we consider the free amalgamation $T := S \TermExt_{s \in A_1 \cup A_2} V'_s$ of the extensions $S \TermExt_s V'_s$ for $s \in A_1 \cup A_2$ where 
			\[
				S \TermExt_s V'_s = \begin{cases}
					S \TermExt_s V_1 & \text{ if $s \in A_1 \setminus A_2$}, \\
					S \TermExt_s V_2 & \text{ if $s \in A_1 \setminus A_2$}, \\
					\text{the amalgamation of $S \TermExt_s V_1$ and $S \TermExt_s V_2$ from (ii)} & \text{ if $s \in A_1 \cap A_2$}.
				\end{cases}
			\]
			This is a free amalgamation of the original extensions $S \TermExt_{s \in A_i} V_{i, s}$ if and only if $A_1 \cap A_2 = \emptyset$.
		
		\item If $T_1 = (S \TermExt_{a \in A} V_a)$ is a terminal extension and $T_2 = \bigl(S \NTermExt_{\beta \in B} (C_{\beta, s}^s)_{s \in S(\beta)}\bigr)$ is a non-terminal extension, then there is a non-gluing amalgamation $T$ such that $T_1 \subs T$ is non-terminal and $T_2 \subs T$ is terminal.
			
			For every $a \in A$ let $B_a := \set{\beta \in B: V_a \cap T_1(\beta) \neq \emptyset}$ and
			for every $\beta \in B$ let $A_\beta := \set{a \in A: V_a \cap T_1(\beta) \neq \emptyset}$.
			We have $T_1(\beta) = S(\beta) \cup \bigcup_{a \in A_\beta} V_a(\beta)$.
			For every $\beta \in B$ and $t \in T_1(\beta)$ let $\bar{C}_{\beta, t}$ be the bush-column $C_{\beta, t}$ if $t \in S(\beta)$, or a new bush-column of the same height containing $t$ as a top-level node.
			For every $a \in A$ we put $\bar{V}_a := V_a \NTermExt_{\beta \in B_a} (\bar{C}_{\beta, t}^t)_{t \in V_a(\beta)}$.
			We consider the amalgamation
			\[
				T :=\, \underbrace{(S \TermExt_{a \in A} V_a)}_{T_1}
					\NTermExt_{\beta \in B} (\bar{C}_{\beta, t}^t)_{t \in T_1(\beta)}
				\,=\, \underbrace{\bigl(S \NTermExt_{\beta \in B} (C_{\beta, s}^s)_{s \in S(\beta)}\bigr)}_{T_2} 
					\TermExt_{a \in A} \bar{V}_a
			\]
			as in Figure~\ref{fig:NT_amalgamation}.
			The amalgamation is free if and only if every $V_a$ is disjoint with every $T_1(\beta)$.
			
			\begin{figure}[ht]
	\centering
	
	\begin{tikzpicture}[trees]
		
		\path[dot nodes,
			level 1/.style = {sibling distance = 9\ru},
			level 2/.style = {sibling distance = 3.5\ru},
		] node (T1) at (-24ex, 2.5\ru) {}
			child[subtree={green}{left triangle node}] {node {}
				child {node {}}
				child {node {}}
			}
			child {node {}
				child {node {}}
				child {node {}}
				child {node {}}
			}
		;
		
		\path[dot nodes,
			level 1/.style = {sibling distance = 12\ru},
			level 3/.style = {sibling distance = 3\ru},
		] node (T) {}
			child[subtree={green}{left triangle node}] {node {}
				child[subtree={blue}{diamond node}, shift children=-2\ru] {node {}
					child {node[green, left triangle node] {}}
				}
				child[subtree={blue}{diamond node}, shift children=-2\ru] {node {}
					child {node[green, left triangle node] {}}
					child {node {}}
				}
			}
			child {node {}
				child[subtree={red}{right triangle node}] {node {}
					child {node[black, circle node] {}}
					child {node {}}
				}
				child[subtree={red}{right triangle node}] {node {}
					child {node[black, circle node] {}}
				}
				child[subtree={red}{right triangle node}] {node {}
					child {node[black, circle node] {}}
					child {node {}}
				}
			}
		;
		
		\path[dot nodes,
			level 1/.style = {sibling distance = 9\ru},
			level 3/.style = {sibling distance = 3\ru},
		] node (T2) at (26ex, 0) {}
			child {node {}}
			child {node {}
				child[subtree={red}{right triangle node}] {node {}
					child {node[black, circle node] {}}
					child {node {}}
				}
				child[subtree={red}{right triangle node}] {node {}
					child {node[black, circle node] {}}
				}
				child[subtree={red}{right triangle node}] {node {}
					child {node[black, circle node] {}}
					child {node {}}
				}
			}
		;
		
		\node[green, left=1\ru] at (T1-1) {$V_a$};
		\node[red, left=1\ru] at (T2-2-1) {$C_{\beta, s}$};
		\node[red, right=1\ru] at (T-2-3) {$\bar{C}_{\beta, t}$};
		\node[green, left=1\ru] at (T-1) {$\bar{V}_a$};
	\end{tikzpicture}
	
	\caption{An amalgamation of a terminal and a non-terminal extension.}
	\label{fig:NT_amalgamation}
\end{figure}
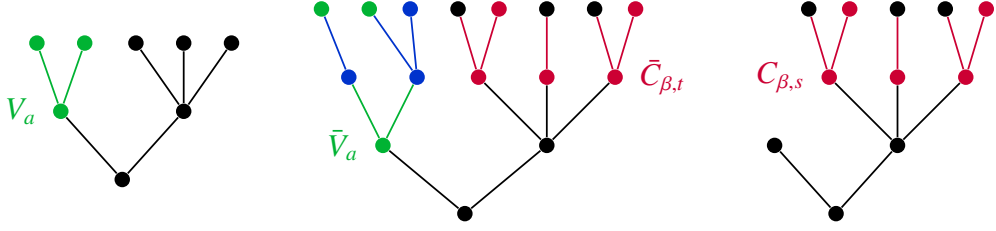
		
		\item Suppose $T_1 = \bigl(S \NTermExt_\al (C_s^s)_{s \in S(\al)}\bigr)$ and $T_2 = \bigl(S \NTermExt_\al (D_s^s)_{s \in S(\al)}\bigr)$ are two non-terminal extensions extending the same level $\al$.
			
			Let $\beta_{i, 0} < \beta_{i, 1} < \cdots < \beta_{i, k_i}$, $i \in \set{0, 1}$, be the enumeration of the levels of the bush-columns $C_s$ and $D_s$, respectively, so in $T_i$ the newly added levels are $\beta_{i, j}$, $j < k_i$, while $\beta_{i, k_i} = \al$.
			In our amalgamation $T$ of $T_1$ and $T_2$ we shall have $\beta_{2, 0} < \beta_{2, 1} < \cdots < \beta_{2, k_2 - 1} < \beta_{1, 0} < \beta_{1, 1} < \dots < \beta_{1, k_1} = \al = \beta_{2, k_2}$.
			
			For every $t \in T_2(\al) = \bigcup_{s \in S(\al)} D_s(\al)$ let $\bar{C}_t$ denote the bush-column $C_t$ if $t \in S(\al)$, or a completely new bush-column of the same height with $t$ being a top-level node.
			Let $r(t)$ denote the root of $\bar{C}_t$.
			Moreover, for every $s \in S(\al)$ let $\bar{D}_s$ be an isomorphic copy of the bush-column $D_s$ with every top-level node $t$ replaced by $r(t)$.
			We suppose all the bush-columns are as disjoint as needed.
			
			We shall write the amalgamation $T$ in two ways so it is clear that it is an extension of both $T_1$ and $T_2$ agreeing on $S$.
			Let $\beta := \beta_{1, 0}$.
			First we consider the extension $T' := T_1 \NTermExt_\beta (\bar{D}_{s(r)}^r)_{r \in T_1(\beta)}$ where $s(r)$ denotes the unique $s \in S(\al)$ with $r(s) = r$.
			Hence, $\beta$ becomes the top level of the bush-columns $\bar{D}_s$ in $T'$.
			For every $r \in T'(\beta) = \bigcup_{s \in S(\al)} \bar{D}_s(\beta)$ let $t(r)$ denote the unique $t \in T_2(\al)$ with $r(t) = r$.
			Our amalgamation is
			\let\displaystyle\textstyle
			\begin{align*}
				T :\kern-0.5ex&= \overbrace{\bigl(
						\underbrace{\bigl(S \NTermExt_\al (C_s^s)_{s \in S(\al)}\bigr)}_{T_1} 
						\NTermExt_\beta (\bar{D}_{s(r)}^r)_{r \in T_1(\beta)}
					\bigr)}^{T'}
					\TermExt_{r \in T'(\beta) \setminus T_1(\beta)} \bar{C}_{t(r)} \\
				&=
				\underbrace{\bigl(S \NTermExt_\al (D_s^s)_{s \in S(\al)}\bigr)}_{T_2} \NTermExt_\al (\bar{C}_t^t)_{t \in T_2(\al)}
			\end{align*}
			as in Figure~\ref{fig:NN_amalgamation}.
			Note that the expansion $T_2 \subs T$ is non-terminal and that the amalgamation $T \sups T_1 \cup T_2$ is non-gluing.
			
			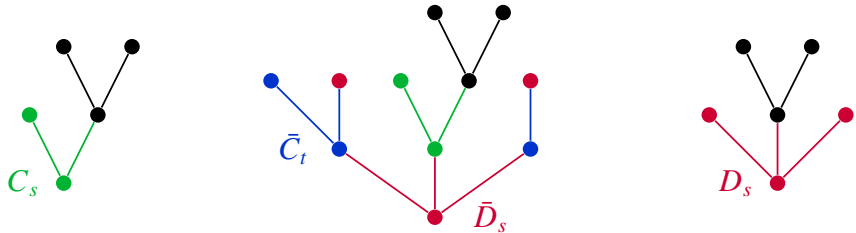
\begin{figure}[ht]
	\centering
	
	\begin{tikzpicture}[trees]
		
		\path[dot nodes, 
			level 1/.style = {sibling distance = 5\ru},
			subtree={green}{left triangle node},
		] node (T1) at (-26ex, 2.5\ru) {}
			child {node {}}
			child[subtree={black}{circle node}] {node {}
				child {node {}}
				child {node {}}
			}
		;
		
		\path[dot nodes,
			subtree={red}{right triangle node},
		] node (T) {}
			child[subtree={blue}{diamond node}, shift children=-2.5\ru] {node {}
				child {node {}}
				child {node[red, right triangle node] {}}
			}
			child[subtree={green}{left triangle node}] {node {}
				child {node {}}
				child[subtree={black}{circle node}] {node {}
					child {node {}}
					child {node {}}
				}
			}
			child[subtree={blue}{diamond node}] {node {}
				child {node[red, right triangle node] {}}
			}
		;
		
		\path[dot nodes, 
			level 1/.style = {sibling distance = 5\ru},
			subtree={red}{right triangle node},
		] node (T2) at (24ex, 2.5\ru) {}
			child {node {}}
			child[subtree={black}{circle node}] {node {}
				child {node {}}
				child {node {}}
			}
			child {node {}}
		;
		
		\node[green, left=1\ru] at (T1) {$C_s$};
		\node[red, left=1\ru] at (T2) {$D_s$};
		\node[blue, left=1.5\ru] at (T-1) {$\bar{C}_t$};
		\node[red, right=2\ru] at (T) {$\bar{D}_s$};
	\end{tikzpicture}
	
	\caption{An amalgamation two non-terminal extensions.}
	\label{fig:NN_amalgamation}
\end{figure}
		
		\item To amalgamate general non-terminal extensions $T_i = (S \NTermExt_{\al \in A_i} C_{i, \al})$ for $i = 1, 2$ we consider the free amalgamation $T := (S \NTermExt_{\al \in A_1 \cup A_2} \bar{C}_\al)$ of the extensions $S \NTermExt_\al \bar{C}_\al$ for $\al \in A_1 \cup A_2$ where 
			\[
				S \NTermExt_\al \bar{C}_\al = \begin{cases}
					S \NTermExt_\al C_{1, \al} & \text{ if $\al \in A_1 \setminus A_2$}, \\
					S \NTermExt_\al C_{2, \al} & \text{ if $\al \in A_1 \setminus A_2$}, \\
					\text{the amalgamation of $S \NTermExt_\al C_{1, \al}$ and $S \NTermExt_\al C_{2, \al}$ from (v)} 
						& \text{ if $\al \in A_1 \cap A_2$}.
				\end{cases}
			\]
			This is always a non-gluing amalgamation of the original extensions $S \NTermExt_{\al \in A_i} C_{i, \al}$, and it is a free amalgamation  if and only if $A_1 \cap A_2 = \emptyset$.
			Note that by (v) one of $T_i \subs T$ is non-terminal.
		
		\item Finally, for general extensions $S \subs T_i$ for $i = 1, 2$ we consider the canonical decompositions $S \subs S_i \subs T_i$ such that $S \subs S_i$ is non-terminal and $S_i \subs T_i$ is terminal, and we perform several immediate amalgamations as in Figure~\ref{fig:complete_amalgamation}.
			By (vi) we take an amalgamation $S_1 \cup S_2 \subs S'$ such that $S_1 \subs S'$ is non-terminal, and we consider the canonical decomposition $S_2 \subs S'' \subs S'$, so $S_2 \subs S''$ is non-terminal and $S'' \subs S'$ is terminal.
			By (iv) we take an amalgamation $T_1 \cup S' \subs T_1'$, so $S' \subs T'_1$ and $S'' \subs T'_1$ are terminal.
			Again by (iv) we take an amalgamation $S'' \cup T_2 \subs T_2''$, so $S'' \subs T_2''$ is terminal.
			We obtain $T$ as the amalgamation of $T_1' \cup T_2''$ performed by (iii).
			Note that since all the immediate amalgamations are non-gluing, the compatibility of $S \subs T_1$ and $S \subs T_2$ implies the compatibility of $S'' \subs T_1'$ and $S'' \subs T_2''$.
			Also, the resulting amalgamation is non-gluing.
			
			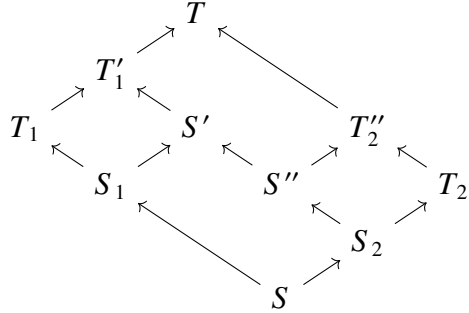
\begin{figure}[ht]
	\centering
	
	\begin{tikzpicture}[
			text diagram,
			x = {(-6ex, 4ex)},
			y = {(6ex, 4ex)},
		]
		\node (S) at (0, 0) {$S$};
		\node (S1) at (2, 0) {$S_1$};
		\node (T1) at (3, 0) {$T_1$};
		\node (S2) at (0, 1) {$S_2$};
		\node (S'') at (1, 1) {$S''$};
		\node (S') at (2, 1) {$S'$};
		\node (T1') at (3, 1) {$T_1'$};
		\node (T2) at (0, 2) {$T_2$};
		\node (T2'') at (1, 2) {$T''_2$};
		\node (T) at (3, 2) {$T$};
		
		\graph{
			(S) -> (S1) -> (T1) -> (T1') -> (T),
			(S) -> (S2) -> (T2) -> (T2'') -> (T),
			(S2) -> (S'') -> (S') -> (T1'),
			(S1) -> (S'),
			(S'') -> (T2''),
		};
	\end{tikzpicture}
	
	\caption{The composition of immediate amalgamations for decomposed extensions.}
	\label{fig:complete_amalgamation}
\end{figure}
	\end{enumerate}
\end{con}

\subsubsection{Countable trees and the generic tree}

We have shown that $\Tw$ is a weak Fraïssé category, that $\Ta$ is a Fraïssé category, and that both are full cofinal subcategories of $\Tc$.
Let $\omega\Tw$, $\omega\Ta$, and $\omega\Tc$ be the corresponding categories of countable strong trees.
Note that the objects of $\omega\Tw \cap \omega\Ta$ are exactly trees from $\omega\Tc$ with no terminal nodes.

Since $\omega\Tc$ is a category of first-order structures and all one-to-one homomorphisms (not necessarily embeddings since deciding a splitting degree of a terminal node is allowed), colimits of sequences essentially correspond to unions of increasing chains.
Note that not every countable tree is a colimit of a sequence finite trees – the colimits are exactly \emph{locally finite} trees, i.e. trees $T$ such that for every finite subset $F \subs T$ there is a subtree $S \subs T$ containing $F$.
In our case of lexicographic strong subtrees, we must close $F$ under the meet operation and under levels, which is trivial, but also add witnessing nodes for splitting at non-terminals, which may complicate things.

\begin{ex}
	Consider the tree $T = \set{0, 0', 1, 1', \ldots, \omega}$ where every node $n \in \nat$ has two immediate successors: $n + 1$ and $n'$, every node $n'$ is terminal, and $\omega$ is a terminal node above the chain $\set{0 < 1 < \ldots}$.
	The subset $F := \set{0, \omega}$ is not contained in any finite strong subtree $S \subs T$.
	This is because $S$ would need to contain $0'$ as the witness of the splitting at $0$, and then $1$ as the node below $\nat$ on the level of $0'$, and then $1'$ as the witness of the splitting at $1$, and so on.
\end{ex}

\begin{lm} \label{balanced_locally_finite}
	Balanced trees from $\omega\Tc$ are locally finite.
	More precisely, let $T$ be an $\omega\Tc$-object.
	For every finite $F \subs T$ there is a $\Tc$-object $S$ with $F \subs S \subs T$.
	Moreover, if $T$ is in $\omega\Tw$ or $\omega\Ta$, then $S$ can be taken in $\Tw$ or $\Ta$, respectively.
	
	\begin{proof}
		We proceed in several steps.
		Let $S_0$ be the closure of $F$ under the meet operation in $T$.
		Clearly $S_0$ is finite, it becomes a tree when endowed with the $\set{\leq, \meet}$-structure inherited from $T$, and the inclusion $S_0 \subs T$ is a $\Tree$-arrow.
		
		Let $S_1$ be the closure of $S_0$ under $T$-levels, i.e. $S_1 := \set{s'\restr{\lev_T(s)}: s, s' \in S_0$ with $\lev_T(s) \leq \lev_T(s')}$.
		It is easy to see that $S_1$ is finite, that it becomes a leveled tree with the inherited $\set{\leq, \meet, \eqlv}$-structure, and that the inclusion $S_1 \subs T$ is a $\LevTree$-arrow.
		
		For every non-terminal node of $s \in S_1$, let $\al_s \in \Lev(S_1) \subs \Lev(T)$ denote the level corresponding to immediate successors of $s$ in $S_1$, and let $A_s \subs T(\al_s)$ be a transversal of $\Spl_T(s)$ extending $S_1(\al_s)$.
		Such transversal exists since $T$ is balanced.
		We put $S := S_1 \cup \bigcup_{s \in S_1} A_s$ and endow with the inherited $\set{\leq, \meet, \eqlv, \leq\lex, \set{R_m}_{m \in M}}$-structure.
		We have that $S$ is a $\Tc$-object and that the inclusion $S \subs T$ is an $\omega\Tc$-arrow.
		If $T$ is an $\omega\Ta$-object, then $S$ is a $\Ta$-object.
		If $T$ is an $\omega\Tw$-object, then $S$ is a $\Tw$-object after forgetting the relations $R_m$.
	\end{proof}
\end{lm}

Let $\sigma\Tw$, $\sigma\Tc$, and $\sigma\Tw$ denote the full subcategories of all locally finite trees in $\omega\Tw$, $\omega\Tc$, and $\omega\Tw$, respectively.
Every sequence in $\Tw$ has a colimit in $\sigma\Tw$, and every $\sigma\Tw$-object is a colimit of a $\Tw$-sequence.
Also, every coned sequence in $(\Tw, \sigma\Tw)$ is a matching sequence, see Construction~\ref{sigma_closure}.
The same is true for $(\Tc, \sigma\Tc)$ and $(\Ta, \sigma\Ta)$.
It is not hard to see that both $\sigma\Tw$ and $\sigma\Ta$ are cofinal in $\sigma\Tc$.
Therefore, by the general theory, there is a unique (up to isomorphism) object $U_M$ in $\sigma\Tc$ that is a weak Fraïssé limit in $(\Tc, \sigma\Tc)$ as well as in $(\Tw, \sigma\Tw)$ and a Fraïssé limit in $(\Ta, \sigma\Ta)$.
Moreover, by Theorem~\ref{thm:KPT} $\aut(U_M)$ is extremely amenable if and only if $\Tw$ has the weak Ramsey property if and only if $\Ta$ has the Ramsey property.

\begin{prop}
	The generic tree $U_M$ is cofinal in $\omega\Tc$.
	
	\begin{proof}
		The tree $U_M$ is cofinal in $\sigma\Ta$ as a Fraïssé limit in $(\Ta, \sigma\Ta)$.
		Moreover, it is cofinal in $\sigma\Tc$ since $\sigma\Ta$ is cofinal in $\sigma\Tc$.
		It is enough to prove that every tree $S$ in $\omega\Tc$ can be strongly embedded into a balanced countable tree $T$ since by Lemma~\ref{balanced_locally_finite} $T$ is in $\sigma\Tc$.
		
		By saying that $S$ is \emph{balanced at a node $s$} we mean that for every level $\al > \lev(s)$ there is $s' > s$ such that $\lev(s') = \al$.
		Suppose $a \in S$ is a terminal node at which $S$ is not balanced and let $Y := [\lev(a), \above)_{\Lev(S)}$.
		By Construction~\ref{con:XY_tree} there is a balanced tree $T_a$ in $\omega\Tw$ with $\Lev(T_a) = Y$ (just take $V_{m, Y}$ where $M \ni m = \set{0 < \cdots < m - 1}$).
		By planting the tree $T_a$ at $a \in S$, i.e. by considering an infinite version of a terminal extension $S \TermExt_a T_a$, we obtain a strong supertree of $S$ balanced at $a$ as well as at every newly added node.
		
		Similarly, suppose that $B \subs S$ is an unbounded branch (i.e. a maximal linearly ordered subset without maximum) such that the upper set $Y := \Lev(S) \setminus \lev[B]$ is nonempty (where $\lev[B] = \set{\lev(b): b \in B}$).
		We shall again consider a balanced tree $T_B$ in $\omega\Tw$ with $\Lev(T_s) = Y$ and plant it above $B$, i.e. all nodes of $T_B$ will be strictly above all nodes of $B$.
		The resulting tree $S \TermExt_B T_B$ is balanced at every node in $B$ as well as at every newly added node.
		
		Together, there is a set $A \subs S$ of terminal nodes and a countable set $\mathcal{B}$ of unbounded branches such that every node $s \in S$ at which $S$ is not balanced is below a node $a \in A$ or contained in a branch $B \in \mathcal{B}$.
		The tree $T := S \TermExt_{a \in A} T_a \TermExt_{B \in \mathcal{B}} T_B$ is the desired balanced extension.
	\end{proof}
\end{prop}

Next, we shall characterize $U_M$, but first consider the following definition.

\begin{df}
	For every countable set of colors $C$ let $\Qyu_C$ denote $(\Qyu, \leq)$ endowed with a \emph{generic $C$-coloring} $\phi\maps \Qyu \to C$, meaning that every monochromatic set $\phi^{-1}(c)$ is $\leq$-dense.
	Such structure $\Qyu_C$ is unique up to isomorphism as a Fraïssé limit of all $C$-colored finite linear orders and color-preserving embeddings.
	
	For every $M \subs \Nat^+$ we put $M^< := \set{(k, m) \in \Nat_0 \times M: k < m}$, so we can use $\Qyu_M$ and $\Qyu_{M^<}$ later.
	For every $s < t \in T \in \ob{\omega\Tc}$ let $\spl(s, t)$ denote the pair $(k, \spl(s)) \in M^<$ such that $t \in C_k$ where $(C_i: i < \spl(s))$ is the $\leq\lex$-increasing enumeration of $\Spl(s)$.
\end{df}

\begin{tw} \label{thm:U_M}
	$U_M$ is the unique (up to isomorphism) $\omega\Tc$-object satisfying
	\begin{enumerate}[label=\rm(T\arabic*), start=0]
		\item $U_M$ is balanced with $\Lev(U_M)$ being isomorphic to $(\Qyu, \leq)$,
		\item for every $\al \in \Lev(U_M)$, every $t' <\lex t''$ in $U_M(\al)$, and $m \in M$ there is $t \in U_M(\al)$ such that $t' <\lex t <\lex t''$ and $\spl(t) = m$,
		\item for every $\beta < \al \in \Lev(U_M)$, finite $F \subs U_M(\al)$, and $\phi\maps F \to M^<$ there is $\al' \in (\beta, \al)$ such that $\spl(t\restr{\al'}, t) = \phi(t)$ for every $t \in F$.
	\end{enumerate}
	It follows that $U_M$ satisfies also
	\begin{enumerate}[resume*]
		\item every branch $B$ with the coloring $\phi\maps t \mapsto \spl(t) \in M$ is isomorphic to $\Qyu_M$,
		\item for every $t \in U_M$ the interval $(\below, t)$ with the coloring $\phi\maps s \mapsto \spl(s, t) \in M^<$ is isomorphic to $\Qyu_{M^<}$,
		\item for every $t \in U_M$, $\al > \lev(t)$, and $C \in \Spl(t)$ the sets $U_M(\al) \cap C$, $U_M(\al) \cap [t, \above)$, and $U_M(\al)$ with the lexicographic order and with the coloring $\phi\maps t \mapsto \spl(t) \in M$ are all isomorphic to $\Qyu_M$ unless $M = \set{1}$,
		\item for every finite $F \subs U_M(\al)$ there is $\beta < \al$ such that $(\beta, \al)$ with the coloring $\phi\maps \gamma \mapsto (\spl(t\restr{\gamma}, t))_{t \in F} \in (M^<)^F$ is isomorphic to $\Qyu_{(M^<)^F}$.
	\end{enumerate}
	
	\begin{proof}
		$U_M$ is characterized by being a Fraïssé limit of $(\Ta, \sigma\Ta)$.
		Since we have the initial object (the empty tree), cofinality of $U_M$ in $(\Ta, \sigma\Ta)$ already follows from injectivity.
		So $U_M$ is characterized by being injective in $(\Ta, \sigma\Ta)$.
		Recall that the injectivity means that for every inclusion $S \subs U_M$ that is a $\sigma\Ta$-arrow from a $\Ta$-object (i.e. a finite decided strong subtree) and every (one-step) $\Ta$-extension $S \subs T$ we find a $\sigma\Ta$-arrow $f\maps T \to U_M$ such that $f\restr{S} = \id{}$.
		Note that without loss of generality we may always suppose $\Lev(S) \subs \Lev(U_M)$.
		
		Let $U$ be an $\omega\Tc$-object.
		First we show that if $U$ is injective in $(\Ta, \sigma\Ta)$, then it satisfies (T0), (T1), (T2).
		$U$ is balanced since for every $s \in U$ and $\al > \lev(s)$ we may consider a finite strong subtree $S \subs U$ such that $s \in S$ and $S \cap U(\al) \ne \emptyset$ (since $U$ is a $\sigma\Ta$-object and so locally finite) with a node $t \in [s, \above)_S$ such that $\lev_S(t)$ is as high as possible below $\al$.
		If $\lev_S(t) = \al$, we are done.
		Otherwise we consider an extension $T := S \TermExt_t B$ for a bush $B$ and the extending map $f\maps T \to U$.
		The image $S' := f[T] \sups S$ contains a node $t' > t$, whose level is strictly closer to $\al$.
		
		$U$ has no maximal level since if $s \in U$ was a node at a maximal level, we could consider an extending map $f\maps \set{s} \TermExt_s B \to U$ for a bush $B$.
		Similarly, $U$ has no minimal level since if $r \in U$ was the root, we could consider an extending map $f\maps \set{r} \NTermExt B^b \to U$ for a pointed bush $B^b$.
		It follows from (T2) (to be proved) that the order of $\Lev(U)$ is dense.
		Together, $\Lev(U)$ is isomorphic to $(\Qyu, \leq)$, and to we have (T0).
		
		We shall prove (T1) and (T2) together.
		Let $\beta < \al$ be levels of $U$, let $F \subs U(\al)$ be finite, and let $\phi\maps F \to M^<$.
		In the case we are proving (T1) we make sure that $t', t'' \in F$ and that $\phi(t') = (0, m')$ for some $m' > 1$ (note that $M \neq \set{1}$ since $t' <\lex t''$).
		Since $U$ is locally finite, there is a finite strong subtree $S \subs U$ such that $F \subs S$ and $S \cap U(\beta) \neq \emptyset$.
		For every $s \in S(\al)$ let $(k_s, m_s)$ be an element of $M^<$ that is equal to $\phi(s)$ if $s \in F$.
		We consider an extending map $f\maps S \NTermExt_\al (B_s^{b_{s, k_s}})_{s \in S(\al)} \to U$ where $B_s$ are bushes with a root $r_s$ and $\leq\lex$-enumerated terminal nodes $(b_{s, i}: i < m_s)$ such that $b_{s, k_s} = s$ and $\dspl(b_{t', 1}) = m$ in the case we are proving (T1).
		Let $\al' \in \Lev(U)$ be the level containing images of the roots $\set{f(r_s): s \in S(\al)}$.
		We have $\beta < \al' < \al$ since $S \cap U(\beta) \neq \emptyset$, and $\spl(s\restr{\al'}, s) = \phi(s)$ for every $s \in F$ since $s\restr{\al'} = f(r_s)$.
		Hence, we have proved (T2).
		In the case we were also proving (T1), we have $b_{t', 0} = t' <\lex t := f(b_{t', 1}) < \lex t''$ and $\spl(t) = \dspl(b_{t', 1}) = m$.
		
		We have proved that the injectivity of $U$ implies (T0), (T1), (T2).
		Next, we shall prove that the conditions (T0), (T1), (T2) imply (T3), (T4), (T5), (T6) for any $U \in \ob{\omega\Tc}$.
		The condition (T3) follows from the fact that $\Lev(U)$ is isomorphic to $(\Qyu, \leq)$ and from (T2).
		(T4) follows similarly.
		To prove (T5) we first show that $U(\al) \cap C$ is isomorphic to $(\Qyu, \leq)$.
		Let $\beta := \lev(t) < \al$ and $t' <\lex t'' \in C$.
		Let also $1 \neq m \in M$.
		By putting $\phi(t') := (1, m)$ and $\phi(t'') := (0, m)$, by applying (T2), and using the fact that $U$ is balanced, we obtain nodes $s', s'' \in C$ with $s' <\lex t' <\lex t'' <\lex s''$.
		By putting $\phi(t') := (0, m)$ instead, we obtain $t' < s''' < t'' \in C$.
		Hence, $U(\al) \cap C$ is isomorphic to $(\Qyu, \leq)$.
		It follows that $U(\al) \cap [t, \above)$ and $U(\al)$ are isomorphic to $(\Qyu, \leq)$ as well since they are covered by the sets of the form $U(\al) \cap C$.
		We conclude (T5) by applying (T1).
		To prove (T6) we consider a level $\beta < \al$ such that every meet of elements from $F$ is strictly below the level $\beta$, so that the map $\gamma \in (\beta, \al) \mapsto (t\restr{\gamma}: t \in F)$ is one-to-one.
		By (T0), the interval $(\beta, \al)$ is isomorphic to $(\Qyu, \leq)$ and by (T2), the coloring $\phi\maps (\beta, \al) \to (M^<)^F$ is generic.
		
		In the last part of the proof we suppose that $U \in \ob{\omega\Tc}$ satisfies (T0), (T1), (T2), and we show that $U$ is injective in $(\Ta, \sigma\Ta)$, and hence it is isomorphic to the Fraïssé limit $U_M$.
		By (T0) $U$ is has no terminal nodes, so $U \in \ob{\omega\Ta}$, and it is balanced, so $U \in \ob{\sigma\Ta}$ by Lemma~\ref{balanced_locally_finite}.
		To prove the injectivity, for every inclusion $S \subs U$ of a finite strong subtree and every one-step $\Ta$-extension $S \subs T$ we find an extending $\sigma\Ta$-arrow $f\maps T \to U$.
		
		Let $T = S \TermExt_s B$ be a terminal one-step extension by a decided bush $B$ with $(b_i: i < m)$ being the $\leq\lex$-increasing enumeration of the top level of $B$.
		We have $m = \dspl_S(s) = \spl_U(s)$ since $S$ is decided.
		Let $\al \in \Lev(U)$ be the $S$-successor level to $\lev(s)$ if $\lev(s) \neq \max(\Lev(S))$.
		Otherwise let $\al$ be an arbitrary $U$-level strictly above $\lev(s)$, which exists since $U$ has no terminal nodes.
		Let $(C_i: i < m)$ be the $\leq\lex$-increasing enumeration of $\Spl_U(s)$.
		By (T5) there is $a_i \in U(\al) \cap C_i$ such that $\spl(a_i) = \dspl(b_i)$ for every $i < m$.
		It is enough to put $f(b_i) := a_i$, $i < m$, to obtain the desired map $f\maps T \to U$.
		
		Next, let $T = S \NTermExt_\al (B_s^s)_{s \in S(\al)}$ be a non-terminal one-step extension where $r_s$ denotes the root of the bush $B_s$, $(b_{s, i}: i < m_s)$ is the $\leq\lex$-increasing enumeration of the top level of $B_s$, and $k_s < m_s$ is the index such that $b_{s, k_s} = s$, for every $s \in S(\al)$.
		Let $\beta \in \Lev(U)$ be the $S$-level immediately preceding $\al$ if it exists, or any $U$-level $\beta < \al$ otherwise.
		By (T2) there is $\al' \in (\beta, \al)$ such that $\spl(s\restr{\al'}, s) = (k_s, m_s)$ for every $s \in S(\al)$.
		Fix $s \in S(\al)$.
		We put $f(r_s) := s\restr{\al'} \in U$ and we denote the $\leq\lex$-increasing enumeration of $\Spl(f(r_s))$ by $(C_{s, i}: i < m_s)$.
		By (T5) for every $i < m_s$ there is a node $a_{s, i} \in U(\al) \cap C_{s, i}$ with $\spl(a_{s, i}) = \dspl(b_{s, i})$.
		Of course, for $i = k_s$ we take $a_{s, i} = s$.
		Putting $f(b_{s, i}) := a_{s, i}$, $i < m_s$, concludes the construction of the desired map $f\maps T \to U$.
		
		Altogether, we have proved that any $\omega\Tc$-object $U$ is injective in $(\Ta, \sigma\Ta)$, which is equivalent to being isomorphic to $U_M$, if and only if it satisfies (T0), (T1), (T2), in which case it satisfies also (T3), (T4), (T5), (T6).
	\end{proof}
\end{tw}

We have proved that the generic tree $U_M$ exists and characterized it up to isomorphism.
Still, it may be useful to give a concrete description of $U_M$.

\begin{con} \label{con:XY_tree}
	For countable linearly ordered sets $S, Y$ and a distinguished point $0 \in S$ we describe a countable balanced strong lexicographic tree $V_{S, Y}$ with the set of levels $Y$ and every $\Spl(t)$ with the lexicographic order isomorphic to $S$.
	
	Let $X$ be the countable family of all total maps $x\maps Y \to S$ with finite support $\supp(x) := \set{y \in Y: x(y) \neq 0}$.
	We endow $X$ with the lexicographic order, i.e. $x \leq x'$ if $x = x'$ or $x(y) < x'(y)$ where $y = \min\set{y' \in Y: x(y') \neq x'(y')}$ (note the minimum exists as the supports are finite).
	We put $V_{S, Y} := \set{x\restr{(\below, y)}: x \in X$ and $y \in Y}$, i.e. the family of all partial maps $t\maps (\below, y) \to S$, $y \in Y$, with finite support.
	The idea is that $0$ corresponds to the canonical splitting direction and that a node is a code for the splitting path from the “trunk” to the node itself.
	The tree order $\loe$ is the extension of maps, i.e. $\subseteq$ when maps are viewed as sets,
	and it admits meets since $\dom(\bigcup\set{v \in V_{S, Y}: v \subs t \cap t'})$ is of the form $(\below, y)$ because of the finite supports.
	
	For every $t \in V_{S, V}$ let $x(t) \in X$ denote the zero-extension of $t$ to a map $Y \to S$, and let $y(t) := \sup(\dom(t)) \in Y$.
	The map $x \times y\maps V_{S, V} \to X \times Y$ is one-to-one since $t = x(t)\restr{(\below, y(t))}$.
	For every incomparable $t', t''$ we put $t' <\lex t''$ if $x(t') < x(t'')$.
	This defines the lexicographic order on $V_{S, V}$.
	Clearly, for every $t \in V_{S, V}$ we have that $\Spl(t)$ with the lexicographic order is isomorphic to $S$.
	The linearly ordered set $Y$ serves as the level set: we put $t \eqlv t'$ if $y(t) = y(t')$, so $\lev(t) = y(t)$ for every $t \in V_{S, V}$.
	Clearly, the tree $V_{S, V}$ is balanced with respect to this level structure.
	
	Note that $V_{1, Y}$ is isomorphic to $Y$ for every linearly ordered set $Y$.
	Also $V_{2, \nat}$ is the full binary tree $2^{<\nat}$.
	
	Every two embeddings of linear orders $f\maps S \to S'$ and $g\maps Y \to Y'$ induce an embedding $e_{f, g}\maps V_{S, Y} \to V_{S', Y'}$ preserving meets, levels, and the lexicographic order, where every node $t\maps (\below, y)_Y \to S$ is mapped to the zero-extended lifting $t'\maps (\below, g(y))_{Y'} \to S'$, i.e. $t' \cmp g = f \cmp t$ and $t'(y') = 0$ for every $y'$ not in the range of $g$.
	Note that if $f, g$ are inclusions, then $t'$ just extends $t$ from $(\below, y)_Y$ to $(\below, y)_{Y'}$ by zeros.
	If $f\maps S \to S'$ is an isomorphism, then $e_{f, g}$ also preserves splitting.
	This construction yields a functor $\LO^2 \to \LexLevTree$ (where $\LO$ denotes the category of all linear orders and embeddings), and for every fixed $S$ we have a functor $\LO \to \LexStrTree$.
	
	Let $C_S$ denote the set of colors $\set{c \subs S: 0 \in c}$, i.e. colors are arbitrary subsets of $S$ containing $0$ (though we will be considering mostly finite subsets later).
	A coloring $\phi\maps V_{S, Y} \to C_S$ induces a pruning $V_{S, Y, \phi} \subs V_{S, Y}$.
	We put $t \in V_{S, Y, \phi}$ if and only if $t(y) \in \phi(t\restr{y})$ for every $y \in \dom(t)$ (equivalently for every $y \in \supp(t)$ since $0 \in c$ for every color $c$).
	In fact, for every $t \in V_{S, Y}$ with $\supp(t) \neq \emptyset$ we put $y'(t) := \max(\supp(t))$ and consider the canonical predecessor $p(t) := t\restr{y'(t)}$.
	We have $t \in V_{S, Y, \phi}$ if and only if $p(t) \in V_{S, Y, \phi}$ and $y'(t) \in \phi(p(t))$ or $\supp(t) = \emptyset$.
	The subtree $V_{S, Y, \phi}$ is balanced since every node can be extended by zeros.
	Also, we have $\spl(t) = \card{\phi(t)}$ for every node $t$.
	The coloring $\phi$ may be given also as $\phi(t) = \phi'(x(t), y(t))$ for a coloring $\phi'\maps X \times Y \to C_S$.
\end{con}

\begin{con}[A concrete model of the generic tree]
	The generic tree $U_M$ may be represented as $T_\phi := V_{\Qyu, \Qyu, \phi}$ for a suitable coloring $\phi\maps T \to C_M$ where $T := V_{\Qyu, \Qyu}$ and $C_M := \set{c \subs \Qyu: 0 \in c$ and $\card{c} \in M}$ is the set of colors corresponding to the allowed splitting degrees.
	Clearly, every tree $T_\phi$ is an object of $\sigma\Ta$ satisfying (T0).
	If $M = \set{1}$, then there is only the constant coloring $\phi$ taking the value $c_0 := \set{0}$, which works since the conditions (T1) and (T2) are trivial in this case.
	Otherwise, we fix $c_1 \in C_M$ with $1 \in c_1$.
	
	Let $A := \set{(t', t'', m) \in T \times T \times M: t' \eqlv t''$ and $t' <\lex t''}$ be the countable set of all situations considered in (T1).
	For every $a = (t', t'', m) \in A$ we choose nodes $u_a, v_a \in T$ such that $t' \meet t'' < u_a < t'$ and $\lev(u_a) > \max(\supp(t'))$.
	Moreover, we let $v_a$ be the extension of $u_a$ defined by $v_a(\lev(u_a)) := 1$ and $v_a(\al) := 0$ for $\al \in (\lev(u_a), \lev(t'))$, so $t' \eqlv v_a \eqlv t''$ and $t' <\lex v_a <\lex t''$, and also $u_a$ is the canonical predecessor $p(v_a)$.
	Since the set $A$ is countable and every interval $(t\restr{\al}, t)$ for $\al < \lev(t)$ is infinite, we may choose the points $u_a$ so that the map $a \mapsto \lev(u_a)$ is one-to-one and its range is disjoint from an order-dense subset $D \subs \Qyu$.
	It follows that the map $a \mapsto v_a$ is also one-to-one since $p(v_a) = u_a$.
	Moreover, we may assure that $\set{u_a: a \in A} \cap \set{v_a: a \in A} = \emptyset$.
	Now it is correct to define $\phi_1(u_a) := c_1$ and $\card{\phi_1(v_a)} := m$ for every $a = (t', t'', m) \in A$, and $\phi_1(t) := c_0$ otherwise.
	The tree $T_{\phi_1}$ satisfies (T1) since for every $a = (t', t'', m) \in A$, if $t' \in T_{\phi_1}$, then $u_a \in T_{\phi_1}$, and so $v_a \in T_{\phi_1}$.
	However, $T_{\phi_1}$ does not satisfy (T2): let $a = (t', t'', m) \in A$ be such that $t' \in T_{\phi_1}$, and let $b := (t', v_a, m)$.
	We have $u_b > u_a$.
	There is no $\al \in (\lev(u_b), \lev(t'))$ such that $\spl(v_a\restr{\al}) \neq 1 \neq \spl(v_b\restr{\al})$.
	This is because if $\phi_1(t) \neq c_0$ for some $t \in (u_a, v_a)$, then $t = u_c$ or $t = v_c$ for some $c \in A$, and the latter is not possible since we would have $u_a = p(v_a) = p(v_c) = u_c$.
	The same is true for $(u_b, v_b)$, but every level can contain at most one node of the form $u_c$.
	
	To assure the condition (T2) we define $\phi_2(t) := \psi(\lev(t))(x(t))$ for every $t \in T$ with $\lev(t) \in D$ where $x(t) \in X$ is the zero-extension of $t$ defined in Construction~\ref{con:XY_tree} and $\psi\maps D \to H_M$ is a coloring defined as follows.
	$H_M$ is the countable set of ``hypercolors'' $\set{h \in (C_M)^X: \supp(h) := \set{x \in X: h(x) \neq c_0}$ is finite$}$.
	Note that $D$ is an isomorphic copy of $\Qyu$.
	We let $\psi$ be a generic $H_M$-coloring of $D$.
	Then for every $\beta < \al \in \Lev(T_{\phi_2})$, every finite $F \subs T_{\phi_2}(\al)$, and $\set{(k_t, m_t): t \in F} \subs M^<$ we consider the hypercolor $h \in H_M$ defined by $h(x(t)) := c_t := \set{-k_t, \cdots, m_t - k_t - 1}$ (so that $c_t$ contains $m_t$ numbers and $0$ is at the position $k_t$) for every $t \in F$, and $h(x) := c_0$ otherwise.
	There is $\beta' \in (\beta, \al)$ such that $\max(\supp(t)) > \beta'$ for every $t \in F$, and there is $\al' \in (\beta', \al) \cap D$ such that $\psi(\al') = h$, and so $\phi_2(t\restr{\al'}) = c_t$ and $\spl_{T_{\phi_2}}(t\restr{\al'}, t) = (k_t, m_t)$ for every $t \in F$.
	This shows (T2).
	On the other hand, if we define $\phi_2(t) = c_0$ for every $t$ with $\lev(t) \notin D$, we have levels $\al$ such that $\spl(t) = 1$ for every $t \in T_{\phi_2}(\al)$, so (T1) does not hold.
	
	To assure both (T1) and (T2) we combine the colorings $\phi_1$ and $\phi_2$.
	Observe that for a coloring $\phi\maps T \to C_M$ to assure (T2) it is enough that $\phi(t) = \phi_2(t)$ for every $t \in H_{\al}$ and $\al \in D$ for some finite sets $H_{\al} \subs T(\al)$.
	This is because we can choose different witnessing levels $\al' \in (\beta, \al)$ for different (T2) situations $(\al, \beta, F, \set{(k_t, m_t): t \in F})$ and put $H_{\al'} := \set{t\restr{\al'}: t \in F}$.
	We also put $H_{\al} := \emptyset$ for every $\al \in \Qyu \setminus D$.
	So, $\phi$ defined as $\phi(t) := \phi_2(t)$ if $t \in H_{\lev(t)}$, and $\phi(t) := \phi_1(t)$ otherwise, assures (T2).
	It also assures (T1) since for $(t', t'', m) \in A$ there is $t \in (t', t'']_{\leq\lex}$ with $t \eqlv t'$ such that $(t', t)_{\leq\lex} \cap H_{\lev(t')} = \emptyset$, and so $\phi(v_{a'}) = \phi_1(v_{a'})$ for $a' := (t', t, m) \in A$.
	Clearly also $\phi(u_a) = \phi_1(u_a)$ since $H_{\lev(u_a)} = \emptyset$.
	This concludes the proof that $T_\phi$ is isomorphic to $U_M$ by Theorem~\ref{thm:U_M}.
\end{con}

\subsubsection{The weak Ramsey property and dendrites}

We work with the trio of categories of finite trees $\Tw$, $\Tc$, $\Ta$ for a fixed nonempty set $M \subs \Nat^+$ of allowed splitting degrees.
We have shown that $\Tw, \Ta \subs \Tc$ are full cofinal subcategories, that $\Tw$ is a weak Fraïssé category, $\Ta$ is a Fraïssé category, and we have described the common (weak) Fraïssé limit $\U$.
In this situation, $\Tw$ has the weak Ramsey property if and only if $\Ta$ has the Ramsey property if and only if $\aut(\U)$ is extremely amenable.
We were unable to show that the equivalent properties hold.
In this section we give partial results in this direction and pose the general case as an open question.

The special case $M = \set{m}$, in which $\Tw$ has the amalgamation property since there is only one option for the splitting degree (Theorem~\ref{thm:tree_amalgamation}~(i)), is covered by the finite version of Milliken's theorem~\cite[Corollary~1.5]{MilTrees}.

\begin{tw}[Milliken]
	For every positive integers $m, a, b, k$ there exists a positive integer $N = N(m, a, b, k)$ with the following property.
	If $T$ is a finite balanced tree of height $N$ and splitting degrees $\leq m$, and all its balanced strong subtrees of height $a$ are colored by (partitioned into) $k$ colors, then there exists a balanced strong subtree $S \subs T$ of height $b$ such that all balanced strong subtrees of $S$ of height $a$ have the same color.
\end{tw}

\begin{wn}
	For $m \in \Nat^+$ and $M = \set{m}$, the category $\Tw$ has the Ramsey property. 
\end{wn}

\begin{proof}
	For every $n \in \Nat$ let $T_n$ denote the (unique up to isomorphism) lexicographically ordered balanced $m$-splitting tree of height $n$.
	The full subcategory $\set{T_n: n \in \Nat^+} \subs \Tw$ is cofinal and so it suffices to show that for any $a$, the tree $T_a$ is Ramsey.
	Note that for $a, b \in \Nat$, balanced strong subtrees $S \subs T_b$ of height $a$ are in one-to-one correspondence with $\Tw(T_a, T_b)$ by taking an embedding to its image.
	Fix $a, b, k \in \Nat^+$ and set $N = N(m, a, b, k)$.
	Take a coloring $\phi\maps \Tw(T_a, T_N) \to k$.
	By the Milliken's theorem there is a balanced strong subtree $S \subs T_N $ of height $b$ and the corresponding $\Tw$-embedding $f\maps T_b \to T_N$ such that $\phi$ is constant on $f \circ \Tw(a, b)$.
\end{proof}

The generic tree $\U$ shares a lot of connections with the generalized Ważewski dendrite $\WD{M + 1}$.
Known related results by Duchesne~\cite{Duchesne} and Kwiatkowska~\cite{Kwiatkowska} allow us to obtain the desired Ramsey property.
However we need to modify our categories.
Namely, we need to forget the level structure.

\begin{df}
	For $M \subs \Nat^+$ we consider the trio $\LTw$, $\LTc$, $\LTa$ of \emph{leveless} variants of the categories $\Tw$, $\Tc$, $\Ta$, i.e. the levels are not part of the structure and they are not preserved by embeddings.
	In particular, $\LTw$ is the full subcategory of $\LexSplTree$ (as opposed to $\LexStrTree$) consisting of finite trees with splitting degrees in $M$; $\LTc$ additionally allows and $\LTa$ requires deciding the splitting degrees at terminal nodes (formally via unary relations $R_m$ as in the definition of $\Lab\FSplTree$).
	
	Let $F\maps \Tc \to \LTc$ be the functor forgetting the level structure.
	Since every tree in $\LTc$ is finite, it admits a unique level structure, so the functor $F$ is not only faithful, but also bijective on objects.
	Hence, we may view $\Tc$ as a subcategory of $\LTc$ with the same objects, but fewer morphisms.
	This way we have $\Tw = \Tc \cap \LTw$ and $\Ta = \Tc \cap \LTa$ as in Figure~\ref{fig:leveless}, and $\LTw, \LTa \subs \LTc$ are full cofinal subcategories as with the original trio.
	
	We will show in Proposition~\ref{thm:level_domination} that the subcategory $\Tc \subs \LTc$ (and so also $\Tw \subs \LTw$ and $\Ta \subs \LTa$) is dominating, i.e. for an embedding $f \in \LTc$ there is $g \in \LTc$ such that $g \cmp f \in \Tc$.
	It follows that every amalgamable $\Tc$-arrow is also amalgamable in $\LTc$ (Lemma~\ref{thm:amalgamable_arrow}), so $\LTw$ is a weak Fraïssé category and $\LTa$ is a Fraïssé category.
	Moreover, a (weak) Fraïssé sequence in $\Tw$ or $\Ta$ is a (weak) Fraïssé sequence in $\LTw$ or $\LTa$ (Lemma~\ref{thm:fseq_in_dominating}), and so the common (weak) Fraïssé limit of $\LTw$ and $\LTa$ is the generic tree $\U$ with its level structure removed.
	We denote it by $\LU$.
\end{df}

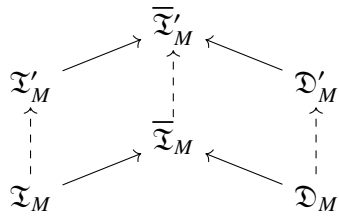
\begin{figure}[ht]
	\centering
	
	\begin{tikzpicture}[
			text diagram,
			x = {(10ex, 0)},
			y = {(0, -8ex)},
		]
		\node (tc') at (0, 0.5) {$\LTc$};
		\node (tw') at (-1, 1) {$\LTw$};
		\node (tc) at (0, 1.5) {$\Tc$};
		\node (ta') at (1, 1) {$\LTa$};
		\node (tw) at (-1, 2) {$\Tw$};
		\node (ta) at (1, 2) {$\Ta$};
		
		\graph{
			{(tw'), (ta')} -> (tc'),
			{(tw), (ta)} -> (tc),
			(tw) ->[dashed] (tw'),
			(tc) ->[dashed] (tc'),
			(ta) ->[dashed] (ta'),
		};
	\end{tikzpicture}
	
	\caption{Inclusions between the trios of categories. Diagonal inclusions are full cofinal, vertical inclusions are wide dominating.}
	\label{fig:leveless}
\end{figure}

\begin{remark}
	Following \cite[Definition~5.1]{KPT} we may call a faithful functor $F\maps \fC \to \fC'$ (e.g. representing the language reduct between classes of structures) \emph{reasonable} if for every $\fC$-object $A$ and every $\fC'$-arrow $f'\maps F(A) \to B'$ there is a $\fC$-arrow $f\maps A \to B$ with $F(f) = f'$.
	This may be viewed as a strong form of the absorption property (D1).
	However, in the case $F$ is the inclusion of a wide subcategory (as above), $F$ being reasonable would already imply $\fC = \fC'$.
	In particular, our $F\maps \Tc \to \LTc$ is not reasonable.
\end{remark}

\begin{prop} \label{thm:level_domination}
	The subcategory $\Tc$ is dominating in $\LTc$.
\end{prop}

\begin{proof}
	Let $S$ be a $\LTc$-subtree of $T$.
	The inclusion $S \subs T$ may not preserve levels.
	We find an $\LTc$-extension $T \subs \hat{T}$ such that the inclusion $S \subs \hat{T}$ will preserve levels, as in Figure~\ref{fig:level_domination}.
	Namely, for every non-terminal node $s \in S$ and an $S$-immediate successor $s' > s$ we let $h_{s, s'}$ denote the cardinality of $(s, s')_T$, we put $h_s := \max\set{h_{s, s'}: s'$ an immediate successor of $s$ in $S}$, and we form $\hat{T} \sups T$ by adding a chain of $h_s - h_{s, s'}$ new nodes between $s'$ and its predecessor in $T$ for every $s \in S$ and an $S$-immediate successor $s' > s$.
	Also, for every newly added node $t$ we decide its splitting degree $m \in M$ and add $m - 1$ new immediate successors of $t$.
\end{proof}
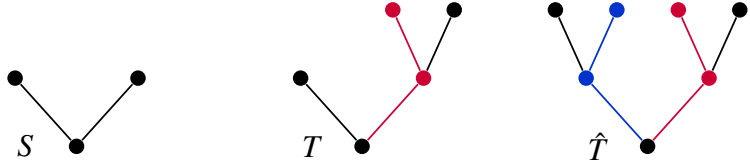
\begin{figure}[htp]
	\centering

	\begin{tikzpicture}[
			trees,
			level 1/.style = {sibling distance = 9\ru},
			level 2/.style = {sibling distance = 4.5\ru},
		]
		
		\path[dot nodes] node (S) {}
			child {node {}}
			child {node {}}
		;
		
		\path[dot nodes] node (T) at (20ex, 0) {}
			child {node {}}
			child[red] {node[left triangle node] {}
				child {node[left triangle node] {}}
				child[black] {node {}}
			}
		;
		
		\path[dot nodes] node (T') at (40ex, 0) {}
			child[blue] {node[right triangle node] {}
				child[black] {node {}}
				child {node[right triangle node] {}}
			}
			child[red] {node[left triangle node] {}
				child {node[left triangle node] {}}
				child[black] {node {}}
			}
		;
		
		\node[left=2\ru] at (S) {$S$};
		\node[left=2\ru] at (T) {$T$};
		\node[left=2\ru] at (T') {$\hat{T}$};
	\end{tikzpicture}
	
	\caption{Domination by level-preserving embeddings.}
	\label{fig:level_domination}
\end{figure}

We can obtain a characterization of the generic leveless tree $\LU$ similar to Theorem~\ref{thm:U_M}.
But first we need the following observation.

\begin{observation}[Leveless extensions]
	We reuse our results and notation from Section~\ref{sec:extensions} to quickly describe how the extensions in the leveless categories $\LTw$, $\LTc$, and $\LTa$ look like.
	The definition of a terminal and non-terminal extension makes sense in the leveless context.
	As in Proposition~\ref{thm:extension_decomposition} every extension can be canonically decomposed into a non-terminal extension followed by a terminal extension, and as in Proposition~\ref{thm:terminal_extension} every terminal extension is of the form $S \TermExt_{s \in A} T_a$ for a set $A$ of terminal nodes.
	Also, every terminal extension can be decomposed into one-step terminal extensions, which are of the form $S \TermExt_s B$ for a bush $B$.
	The situation with non-terminal extensions is different – in the leveless context we are allowed to perform tree surgery at individual nodes rather than whole levels.
	Similarly to Proposition~\ref{thm:nonterminal_extension} it can be shown that every non-terminal extension is canonically of the form $T = S \NTermExt_{s \in A} C_s^{c_s}$ where $A \subs S$ and every $C_s^{c_s}$ is a pointed bush-column, and so $T$ can be decomposed into one-step non-terminal extensions of the form $S \NTermExt_s B^b$ for a pointed bush $B^b$.
\end{observation}

\begin{tw} \label{thm:leveless_generic}
	$\LU$ is the unique (up to $\LexSplTree$-isomorphism) countable lexicographic tree such that every branch is isomorphic to $(\Qyu, \leq)$,
	and for every $s < t \in \LU$ and $(k, m) \in M^<$ there is $u \in (s, t)$ such that $\spl(u, t) = (k, m)$.
	
	\begin{proof}
		Clearly, by Theorem~\ref{thm:U_M}, $\LU$ satisfies the characterizing conditions.
		For the converse implication we may use Fraïssé theory and prove that a countable lexicographic tree $U$ satisfying the conditions is injective with respect to $\LTa$.
		Let $S \subs U$ be a decided lexicographic subtree.
		By the previous observation it is enough to consider one-step extensions.
		Suppose $T = S \TermExt_s B$ for an $S$-terminal node $s$ and a bush $B$.
		Since $S$ is decided, we have $m := \spl_B(s) = \dspl_S(s) = \spl_U(s)$.
		The last equality uses the assumption that $s \in U$ is not terminal.
		It is enough to define the extension $f\maps T \to \LU$ extending $\id{S}$ by putting $f(b_i) \in C_i$ for every $i < k$ where $(b_i)_{i < m}$ is the $\leq\lex$-increasing enumeration of the top level of $B$ and $(C_i)_{i < m}$ is the $\leq\lex$-increasing enumeration of $\Spl_U(s)$.
		Next, suppose $T = S \NTermExt_s B^s$ for a node $s \in S$ and a pointed bush $B^s$.
		Let $r$ be the root of $B$, let $(b_i)_{i < m}$ be the $\leq\lex$-increasing enumeration of the top level of $B$, and let $k < m$ be the index with $b_k = s$.
		By the assumption that $U$ has no root and by the characterizing property, there is a node $t < s$ with $\spl_U(t, s) = (k, m)$.
		Hence, we may take a $\leq\lex$-increasing enumeration $(C_i)_{i < m}$ of $\Spl_U(t)$ with $s \in C_k$.
		Putting $f(b_i) \in C_i$ for $i < m$ with $f(b_k) = f(s) = s$ defines a desired extension $f\maps T \to U$ of $\id{S}$.
		
		Alternatively, we could show that a countable lexicographic tree $U$ satisfying the conditions admits a level structure such that the expansion satisfies the conditions in Theorem~\ref{thm:U_M}.
		The idea is to fix a family $\set{B_n}_{n \in \nat}$ of branches covering $U$, let $B_0$ correspond to the set of levels, and construct suitable order-isomorphisms $B_n \to B_0$.
		We have $B_0 \cap B_1 = (\below, b]$ for some $b \in U$.
		We choose an order-isomorphism $g\maps (b, \above)_{B_n} \to (b, {\above})_{B_0}$ such that every combination $(m, m') \in M^2$ appears as the value of $(\spl_U(t), \spl_U(g(t))$ for densely many nodes $t \in (b, \above)$.
		And we continue similarly with other branches $B_n$.
	\end{proof}
\end{tw}

We now begin our discussion on dendrites and their connection to the generic leveless tree $\LU$. 
Recall that a \emph{dendrite} (see e.g. \cite[\S X]{Nadler}) is a locally connected metrizable continuum $D$ such that every two points $x, y \in D$ are connected by a unique arc denoted by $[x, y]$ or equivalently there are no (non-trivial) closed curves in $D$.
Also for every three points $x, y, z \in D$ there is a unique point in the intersection $[x, y] \cap [y, z] \cap [z, x]$ denoted by $\meet(x, y, z)$ and called a \emph{ternary meet} or a \emph{median}.

Every point $x \in D$ has the \emph{order} $\ord(x) \in \set{1, 2, 3, \ldots, \omega}$ equal to the number of connected components of $D \setminus \set{x}$.
Points of order $1$ are called \emph{end points}, points of order $2$ are called \emph{ordinary points}, and points of order $\geq 3$ are called \emph{branch points}.
We denote the set of all branch points, which is countable, by $\Br(D)$.

Finally, for every $M \subs \set{3, 4, 5, \ldots, \omega}$ there is a unique (up to homeomorphism) dendrite $\WD{M}$ called \emph{(generalized) Ważewski dendrite} such that the order of every branch point is in $M$ and for every $m \in M$ and every arc $A \subs \WD{M}$ the set of branch points of order $m$ is dense in $A$.
Generalized Ważewski dendrites were introduced in \cite[\S 6]{CharatonikDilks}.
The universal Ważewski dendrite $\WD{\set{\omega}}$ originates from \cite{Wazewski}.

\begin{con}[Rooting a dendrite]
	Let $D$ be a dendrite and let $r \in D$.
	There is a natural way to turn $D$ into a tree by rooting it at $r$.
	We define the tree order by $x \leq y \letiff [r, x] \subs [r, y]$.
	This is a well-defined order that has meets: $x \meet y = \meet(r, x, y)$.
	Clearly, $[r, z] \subs [r, x] \cap [r, y]$ if and only if $[r, z] \subs [r, \meet(r, x, y)]$.
	
	Moreover, for every $x \in D$, $\Spl(x)$ consists of the connected components of $D \setminus \set{x}$ omitting the one containing $r$ unless $x = r$, and hence $\spl(x) = \ord(x) - 1$ if $x \neq r$, and $\spl(r) = \ord(r)$.
	This is because the components $C$ of $D \setminus \set{x}$ as well as the corresponding sets $C \cup \set{x}$ are arcwise connected, so for $y \in C$ and $z \in C'$, the arc $[y, z]$ goes through $x$ if $C \neq C'$, and stays in $C$ if $C = C'$.
	Hence, $y < x$ for $y \in C$ if $C$ is the component containing $r$, and $y \meet z = x$ if $y \in C$ and $z \in C'$ and $C \neq C'$ are components not containing $r$, and $y \meet z > x$ for $y, z \in C \notowns r$.
\end{con}

\begin{con}
	For a nonempty set $M \subs \Nat^+$ we build a $\LexSplTree$-isomorphic copy of $\LU$ from the Ważewski dendrite $\WD{M'}$ where $M' = \set{m + 1: 1 \neq m \in M} \subs \set{3, 4, 5, \ldots}$.
	
	We pick an end point $r \in \WD{M'}$ and root the dendrite at $r$ according to the previous construction.
	Note that in the resulting tree every branch has the order type of the closed real interval $[0, 1]$ and contains densely many branch points of every order from $M'$ as well as densely many ordinary points (there are only countably many branch point in a dendrite).
	Also note that the end points of the dendrite correspond to maxima of the tree and to the root.
	
	Let $T \subs \WD{M'}$ be the countable subset consisting of all branch points and, if $1 \in M$, also of an ordinary point $z \in (x, y)$ for every $x < y \in \Br(\WD{M'})$.
	For every $x, y \in T$ we have $x \meet y = \meet(r, x, y)$, which is either one of $x, y$ or a branch point.
	Since $T$ contains all branch points, it follows that $T$ is closed under meets.
	Since also the branch points are dense in every arc of the dendrite, we have that $\spl_T(t) = \spl_{\WD{M'}}(t)$ for every $t \in T$.
	Hence, the inclusion $T \subs \WD{M'}$ is a $\SplTree$-embedding.
	Moreover, every branch in $T$ is ordered like the rationals since $T$ contains no endpoints, and for every $x < y \in T$ and $m \in M$ there is $z \in (x, y)_T$ with $\spl_T(z) = m$.
	
	To construct the desired $\LexSplTree$-isomorphic copy of $\LU$, we endow $T$ with a lexicographic order such that the resulting tree satisfies the characterizing condition of Theorem~\ref{thm:leveless_generic}.
	To define a compatible lexicographic order it is enough to fix a linear order on every $\Spl_T(x)$ for $x \in T$, but to ensure the condition we choose a point $w_{x, y, k, m} \in (x, y)_T$ of splitting degree $m$ for every $x < y \in T$ and $(k, m) \in M^<$ such that the map $(x, y, k, m) \mapsto w_{x, y, k, m}$ is one-to-one, and we make sure that $y \in C_k$ where $(C_i)_{i < m}$ is the $\leq\lex$-increasing enumeration of $\Spl_T(w_{x, y, k, m})$.
\end{con}

\aseparator

Finite rooted subtrees obtained from the branch points of dendrites have been extensively studied, often to the great success of computing the universal minimal flow of the Ważewski dendrites, see Duchesne~\cite{Duchesne} and Kwiatkowska~\cite{Kwiatkowska}.
In fact, a rephrasing of a result by Kwiatkowska gives us the Ramsey property of $\LTa$.

\begin{tw}
	For every nonempty $M \subs \Nat^+$ the category $\LTa$ has the Ramsey property.
	Hence, $\LTw$ has the weak Ramsey property and $\aut(\LU)$ is extremely amenable.
	
	\begin{proof}
		The second part follows from our general theory.
		The first part is a reformulation of \cite[Theorem~3.6]{Kwiatkowska}.
		Namely, let $\LTa'$ be the modified version of the category $\LTa$ where the predetermined splitting degree $\dspl$ may be strictly greater (but still from $M$) than the actual splitting degree $\spl$.
		The $\LTa'$-arrows preserve the relations $R_m$, $m \in M$, that encode $\dspl$, so it is the predetermined splitting degree rather than the actual one that is being preserved.
		Nevertheless, $\LTa \subs \LTa'$ is a full cofinal subcategory since we may simply add the missing immediate successors to any non-terminal node $s$ such that $\spl(s) < \dspl(s)$.
		Hence, $\LTa$ has the Ramsey property if and only if $\LTa'$ has the Ramsey property.
		
		The statement of \cite[Theorem~3.6]{Kwiatkowska} is that a certain category $\mathcal{T}_P^*$ has the Ramsey property.
		We argue that $\LTa'$ is equivalent to $\mathcal{T}_P^*$ for $P = M + 1$.
		The objects of $\mathcal{T}_P^*$ are finite trees, but the unordered graph-theoretic tree structure is encoded by a quaternary relation $D(a, b, c, d)$ holding if the finite paths from $a$ to $b$ and from $c$ to $d$ do not intersect, and the tree order is encoded by a ternary relation $C(a, b, c)$ defined by $D(a, b, c, r)$ where $r$ is a fixed root.
		By \cite[Proposition~3.3]{Kwiatkowska}, a map between trees preserves the relations $C$ and $D$ if and only if it preserves the tree order and meets.
		The trees are labeled by unary relations $K_p$, $p \in P$, giving an upper bound for a splitting degree and directly corresponding to our relations $R_m$ where $p = m + 1$.
		Finally, the lexicographic order is encoded by binary relations $G_i(a, b)$, $i \in \Nat^+$: each node $a$ with a predetermined splitting degree $m$ has “slots” for immediate successors indexed by $i \in \set{1, \ldots, m}$, and $G_i(a, b)$ holds if $b$ is an actual immediate successor of $a$ occupying the $i$th slot.
		
		Note that the statement of \cite[Theorem~3.6]{Kwiatkowska} does not cover the cases when $1 \in M$ (though it allows infinite bound on a splitting degree), however, the proof could be rewritten using our language to directly show that $\LTa$ has the Ramsey property for any $M$.
	\end{proof}
\end{tw}

The following question remains open as the level structure poses a new challenge. The above however, seems like a step forward to proving this fact and deducing another structural variant of the classic Milliken's theorem. 

\begin{question}
	Does the category $\Tw$ have the weak Ramsey property?
	Equivalently, does $\Ta$ have the Ramsey property?
	Equivalently, is $\aut(\U)$ extremely amenable?
\end{question}

\end{document}